\documentclass[10pt,english,a4paper]{article}

\usepackage{amssymb}
\usepackage{amsfonts}
\usepackage{graphicx}
\usepackage{amstext}

\usepackage{amsmath,amscd, amsthm}
\usepackage{mathrsfs}
\usepackage{bbm}

\usepackage{import}

\usepackage[mathscr]{euscript}
\usepackage[latin9]{inputenc}
\usepackage{geometry}
\usepackage{pdfsync}

\usepackage{hyperref}
\usepackage[capitalise]{cleveref}

\usepackage[all]{xy}
\usepackage[title,titletoc]{appendix}

\usepackage{stmaryrd}

\usepackage{tikz} 
\usetikzlibrary{arrows,decorations.pathmorphing,backgrounds,positioning,fit,petri,shapes.misc,decorations.markings,shapes,shapes.multipart,calc,cd}
\usepackage{tikz-cd}

\newcommand{\Z}{\mathbb{Z}}

\newcommand{\R}{\mathbb{R}}

\newcommand{\Hom}{\operatorname{Hom}}
\newcommand{\Map}{\operatorname{Map}}

\newcommand{\Nat}{\operatorname{Nat}}
\newcommand{\Fun}{\operatorname{Fun}}

\newcommand{\iso}{\overset{\sim}{\longrightarrow}}
\renewcommand{\=}{\cong}

\newcommand{\colim}{\operatorname{\underrightarrow{colim}}}
\newcommand{\Gcolim}{G\operatorname{-\underrightarrow{colim}}}

\newcommand{\Psh}{\ul{\operatorname{Psh}}_G}

\newcommand{\xto}{\xrightarrow}

\newcommand{\set}[1]{\left\{{#1}\right\}}

\newcommand{\B}{\mathscr{B}}

\newcommand{\C}{\mathscr{C}}
\newcommand{\D}{\mathscr{D}}
\newcommand{\E}{\mathscr{E}}
\newcommand{\EE}{\mathbb{E}}
\newcommand{\G}{\mathscr{G}}

\newcommand{\Sp}{\mathbf{Sp}}
\newcommand{\Ss}{\mathscr{S}}
\newcommand{\Top}{\mathbf{Top}}
\newcommand{\GTop}{\mathbf{Top}^G}

\newcommand{\adj}{\leftrightarrows}

\newcommand{\Cat}{\mathscr{C}\mathrm{at}}
\newcommand{\ul}[1]{\underline{#1}}
\newcommand{\ol}[1]{\overline{#1}}
\newcommand{\ulTopG}{\ul{\mathbf{Top}}^G}

\newcommand{\OG}{\mathcal{O}_G}
\newcommand{\OGop}{{\OG^{op}}}
\newcommand{\GFin}{\ul{\mathbf{Fin}}^G}
\newcommand{\mfld}{\mathbf{Mfld}}
\newcommand{\ulmfld}{\ul{\mathbf{Mfld}}}
\newcommand{\mfldGD}{\mfld^{G,\sqcup}}

\newcommand{\GmfldD}{\ul{\mfld}^{G,\sqcup}}
\newcommand{\disk}{\mathbf{Disk}}

\newcommand{\orb}{\operatorname{Orbit}}

\newcommand{\Fin}{\mathbf{Fin}}
\newcommand{\ulFin}{\ul{\Fin}}
\newcommand{\ulFun}{\ul{\Fun}}
\newcommand{\ultimes}{\ul{\times}}

\newcommand{\HF}{H\underline{\mathbb{F}}}
\newcommand{\HZ}{H\underline{\mathbb{Z}}}

\newcommand{\Pic}{\ul{\operatorname{Pic}}}
\renewcommand{\B}{\Pic(\ul\Sp^G)}
\newcommand{\Th}{\mathbf{Th}}
\newcommand{\Alg}{\operatorname{Alg}}
\newcommand{\CAlg}{\operatorname{CAlg}}
\newcommand{\CMon}{\operatorname{CMon}}
\newcommand{\fgt}{\mathit{fgt}}
\newcommand{\THR}{\operatorname{THR}}
\newcommand{\THH}{\operatorname{THH}}

\numberwithin{equation}{subsection}
\newtheorem{thm}[equation]{Theorem}
\newtheorem{prop}[equation]{Proposition}
\newtheorem{cor}[equation]{Corollary}
\newtheorem{lem}[equation]{Lemma}

\newtheorem{mydef}[equation]{Definition}
\theoremstyle{definition}

\newtheorem{rem}[equation]{Remark}
\newtheorem{ex}[equation]{Example}
\newtheorem{notation}[equation]{Notation}
\newtheorem{construction}[equation]{Construction}

\newcommand{\cof}{\hookrightarrow}

\newcommand{\fib}{\twoheadrightarrow}

\newcommand{\pullbackcorner}[1][dr]{\save*!/#1-1.2pc/#1:(1,-1)@^{|-}\restore}

\newcommand{\pbcorner}{\arrow[dr, phantom, "\ulcorner" description, very near start]}

\newcommand{\myemph}{\textit}

\usepackage{caption}
\usepackage{enumerate}
\usepackage{ulem}

\begin{document}

\title{Equivariant nonabelian Poincar\'e duality and equivariant factorization homology of Thom spectra}

\author{Asaf Horev, Inbar Klang, and Foling Zou \\
Appendix by Jeremy Hahn and Dylan Wilson}
\date{\today}

\maketitle

\begin{abstract}
In this paper, we study genuine equivariant factorization homology and its interaction with equivariant Thom spectra, which we construct using the language of parametrized higher category theory. We describe the genuine equivariant factorization homology of Thom spectra, and use this description to compute several examples of interest. A key ingredient for our computations is an equivariant non-abelian Poincar\'e duality theorem, in which we prove that factorization homology with coefficients in a $G$-space is given by a mapping space. We compute the Real topological Hochschild homology ($\THR$) of the Real bordism spectrum $MU_\mathbb{R}$ and of the equivariant Eilenberg--MacLane spectra $\HF_2$ and $\HZ_{(2)}$, as well as factorization homology of the sphere $S^{2\sigma}$ with coefficients in these Eilenberg--MacLane spectra. In Appendix B, Jeremy Hahn and Dylan Wilson compute $\THR(\HZ)$. 
\end{abstract}

\tableofcontents

\section{Introduction}

In this paper, we study the equivariant factorization homology of Thom
spectra. Factorization homology has emerged as a fruitful topic of research in
recent years; its roots lie in the study of configuration spaces and their
relation to mapping spaces, but it has also proven valuable in studying
topological field theories, and as a unified way to treat Hochschild homology
theories. Here we primarily take the axiomatic perspective on factorization
homology, introduced by Ayala--Francis \cite{AF}. Ayala--Francis describe
factorization homology with coefficients in an $\EE_n$-algebra $A$, $\int_{-}
A$, as a ``homology theory'' for $n$-manifolds: it satisfies a variation
of the Eilenberg--Steenrod axioms, including functoriality and excision, and is determined by these axioms.

In the case $n=1$, $\int_{S^1}A$ agrees with Hochschild homology of a ring
$A$. Furthermore, If $A$ is a commutative ring spectrum, the factorization
homology $\int_M A$ agrees with the Loday construction \cite{Pirashvili}, which
gives higher Hochschild homology for $M = S^n$ and iterated Hochschild homology
for $M = \mathbb{T}^n$. As such, it is reasonable to expect that factorization
homology might be of use in understanding invariants related to algebraic
K-theory, and recently, Ayala--Mazel-Gee--Rozenblyum used factorization homology
to obtain a new description of the cyclotomic trace from K-theory to topological
cyclic homology (see, e.g., \cite{AMR_GeometryCycTrace}.)

\medskip
In this paper, we consider genuine equivariant factorization homology,
introduced by the first named author \cite{GFH} in his thesis; other definitions
of equivariant factorization homology have also been introduced independently by
Weelinck \cite{Weelinck} and by the third named author \cite{Zou}. The paper
\cite{GFH} uses parametrized higher category theory to define equivariant
factorization homology axiomatically as a homology theory for $G$-manifolds,
where $G$ is a finite group. We consider equivariant factorization homology for
$V$-framed $G$-manifolds, where $V$ is a finite dimensional
$G$-representation. In this theory, the coefficients are $\EE_V$-algebras.
For 1-dimensional manifolds, this construction recovers Real topological
Hochschild homology of \cite{Real_THH} and $C_n$-relative topological Hochschild
homology of \cite{TC_via_the_norm}. Therefore, equivariant factorization
homology provides a new perspective from which to study these invariants, and
other invariants of equivariant ring spectra.

\medskip

A particularly nice class of structured ring spectra is given by Thom spectra,
for which properties of the base space of a spherical fibration can be used to
deduce spectrum-level results. Lewis \cite{lewis1986equivariant} showed that the
Thom spectrum of an $n$-fold loop map is an $\EE_n$-ring
spectrum. Blumberg--Cohen--Schlichtkrull \cite{BCS} showed that the Thom spectrum
functor respects the cyclic bar construction, and used this to describe the
topological Hochschild homology of Thom spectra and to compute several
examples. Schlichtkrull \cite{Sch} generalized this to higher Hochschild
homology of commutative ring spectra. The second named author
\cite{KlangFactorizarionHomologyThom} showed that the Thom spectrum functor
respects factorization homology, and used this to describe the factorization
homology and $\EE_n$ topological Hochschild cohomology of Thom spectra and to
compute examples.

\medskip

In this paper, we apply this philosophy to equivariant Thom spectra and equivariant factorization
homology. A point-set model of the equivariant Thom spectrum functor was constructed by Lewis and May in
\cite[X.3.1]{lewis1986equivariant}. We give an alternative construction of it, reminiscent of the Thom spectrum construction of
Ando--Blumberg--Gepner--Hopkins--Rezk \cite{ABGHR}, as a colimit in parametrized higher
categories. We then prove the properties of this equivariant Thom spectrum functor that govern its interaction with equivariant factorization homology. Here $\B$ denotes the Picard $G$-space of the $G$-category of $G$-spectra; at the orbit $G/H$, its value is the $H$-space $\operatorname{Pic}(\Sp^H)$.
  
  \paragraph{Theorem 1.}  The $G$-Thom spectra functor $\Th$ has
  the following properties:
\begin{enumerate}[(1)]
\item (\cref{Th_of_representables}) Let  $E\in \Sp^G$ be an invertible genuine $G$-spectrum and  $e \colon X \to \B$ be a
  $G$-map  from a $G$-space $X$ such that $e$ is $G$-homotopic to the constant map with value $E$.
  Then $\Th(e) \simeq E\otimes \Sigma^{\infty}_+X\in \Sp^G$.
\item  (\cref{thm:G_Thom_as_GSM_functor}) It is $G$-symmetric monoidal.
\end{enumerate}

The equivariant Thom spectrum functor (e.g. of Lewis--May) is known to be symmetric monoidal; however, to our knowledge, the $G$-symmetric monoidal property has not yet appeared elsewhere in the literature. This property ensures that $\Th$ commutes with Hill--Hopkins--Ravenel norms and preserves $G$-operadic algebra structures:

\paragraph{Corollary 2.}
\begin{enumerate}[(1)]
\item  Let $X$ be an $H$-space and $f: X \to \operatorname{Pic}(\Sp^H)$ be an
  $H$-map. Then $\mathrm{Th}(\mathrm{N}_H^Gf) \simeq \mathrm{N}_H^G
  \mathrm{Th}(f)$, where
  $$\mathrm{N}_H^Gf: \mathrm{N}_H^GX \to \mathrm{N}_H^G\operatorname{Pic}(\Sp^H)
  \to \operatorname{Pic}(\Sp^G).$$
\item  Let $X$ be a $G$-space and $\Omega^Vf: \Omega^VX \to
  \operatorname{Pic}(\Sp^G)$ be a $V$-fold loop $G$-map. Then
  $\mathrm{Th}(\Omega^Vf)$ is an $\EE_V$-algebra spectrum.
\end{enumerate}

These properties ensure
that $\Th$ respects equivariant factorization homology: 

\paragraph{Theorem 3. } (\cref{thm-fh-thom}) Let $X$ be a pointed $G$-space and $\Omega^V f: \Omega^V X \to \B$ be a map of $\EE_V$-algebras.
  Then for every $V$-framed $G$-manifold $M$, the genuine $G$-factorization homology $\int_M \Th(\Omega^V f)$ is equivalent to the genuine $G$-spectrum given by 
  \begin{align*}
    \Th \left( \int_M \Omega^V X \xto{f_*} \int_M \B \to \B \right) ,
  \end{align*}
  where $\Th \colon \ulTopG_{/\B} \to \ul\Sp^G$ is the Thom $G$-functor we construct. 
  
 \medskip
Thus we can leverage knowledge of equivariant
factorization homology on the level of spaces to determine equivariant factorization homology of
Thom spectra. This is made particularly useful by the equivariant non-abelian Poincar\'e duality
theorem of the third named author \cite{Zou}, which we improve here in the axiomatic context:

\paragraph{Theorem 4.}(\cref{thm:NPD}) For $M$ a $V$-framed $G$-manifold and $X$ a pointed
$G$-space such that $\pi_k(X^H) = 0$ for all subgroups $H \subset G$ and $k < \mathrm{dim}(V^H)$, there is a natural equivalence of $G$-spaces
$$\int_M \Omega^{V}X \simeq \Map_{*}(M^+, X).$$
Here, $M^+$ is the one-point compatification of $M$; $\Map_{*}$ is the $G$-space
of based non-equivariant maps with $G$ acting by conjugation.
\medskip

This theorem generalizes the non-abelian Poincar\'e duality theorem of Salvatore
\cite[Theorem 6.6]{Salvatore2001configuration}, Lurie \cite[5.5.6.6]{HA}, and
Ayala--Francis \cite[Corollary 4.6]{AF}. It describes equivariant factorization homology of an equivariant algebra in the category of $G$-spaces
as a compactly supported mapping space. From our equivariant non-abelian Poincar\'e duality theorem, we recover equivariant Atiyah duality and equivariant Poincar\'e duality for $V$-framed $G$-manifolds.
(See \cref{sec:NAPD}.)
\medskip

We use our structural results on the equivariant Thom
spectrum functor, along with the equivariant non-abelian Poincar\'e duality
theorem, to make several computations of interest.
\paragraph{Theorem 5.} (\cref{thm-comp}) Let  $V$ and $W$ be finite
  dimensional $G$-representations and $A$ be the $G$-Thom spectrum of an $\EE_{V \oplus W}$-map,
   $$\Omega^{V \oplus W} f: \Omega^{V \oplus W}X \to \Pic(\ul\Sp^G),$$ and let
 $M$ be a smooth $G$-manifold such that $M \times W$ equivariantly embeds in $V \times W$.
   With some assumptions on $X$ and $M$,
$$\int_{M \times W} A \simeq A \otimes \Sigma^\infty_+ \text{(an explicit
  mapping space)}.$$
Using the theorem, we compute the factorization homology of representation
spheres with coefficients in the Real bordism spectrum $MU_\mathbb{R}$, the Real
topological Hochschild homology of $\HF_2$ and $\HZ_{(2)}$, and the equivariant
factorization homology of the representation spheres $S^{2\sigma}$ with
coefficients in $\HF_2$ and $\HZ_{(2)}$. For example, we recover $\THR(\HF_2)$:
\begin{equation*}
\THR(\HF_2) \simeq \bigoplus_{k \ge 0} \Sigma^{k\rho} \HF_2.
\end{equation*}
Appendix B, written by Jeremy Hahn
and Dylan Wilson, uses these theorems to compute the Real topological Hochschild
homology of $\HZ$:
\paragraph{Theorem 6.}(\cref{app:main-thm})
There is an equivalence of $\HZ$-module spectra
$$\THR(\HZ) \simeq \HZ \oplus \bigoplus_{k \ge 2} \Sigma^{k\rho-1} H\underline{\mathbb{Z}/k}.$$

\medskip

The computations in this paper rely on the two main theorems quoted above: equivariant non-abelian Poincar\'e duality (\cref{thm:NPD}) and the behavior of the Thom spectrum functor under equivariant factorization homology (\cref{thm-fh-thom}). The reader mainly interested in computations can keep these theorems in mind while focusing on \cref{sec:computation} and Appendix B.

\paragraph{Structure of the paper.}
In \cref{sec:topo-func}, we give preliminaries on $G$-symmetric monoidal $G$-$\infty$-categories. We also review the genuine operadic nerve
  construction, which constructs a $G$-symmetric monoidal $G$-$\infty$-category from $G$-objects in a symmetric  monoidal topological category.
In \cref{sec:GFH}, we review the construction of genuine equivariant
factorization homology and its properties. In \cref{sec:NAPD}, we prove the
equivariant nonabelian Poincar\'e duality theorem, and recover equivariant
Atiyah duality for $V$-framed $G$-manifolds. In \cref{sec:G_Thom_spectra}, we
define the $G$-Thom spectrum functor, with preliminaries on parametrized
$\infty$-category theory given in \cref{sec:Parametrized_prelim},
and show that our $G$-Thom spectrum functor respects equivariant factorization homology. In \cref{sec:v-fold-loop}, we explain how the Thom spectrum of a $V$-fold loop map gives rise to an  $\EE_V$-algebra, and give a description of the equivariant factorization homology of the $G$-Thom spectrum of a $V$-fold loop map. In \cref{sec:computation}, we give computations of the equivariant factorization homology of certain Thom spectra using results from the previous sections. Appendix B, by Jeremy Hahn and Dylan Wilson, gives a computation of $\THR(\HZ)$.

\paragraph{Foundations.}
We take an eclectic viewpoint on the foundations. When dealing with
  manifolds and in \cref{sec:NAPD}, we work with topological categories. For the
  rest, we
use Joyal's quasi-categories as a theory of $\infty$-categories, developed in
\cite{HTT} and \cite{HA}.
We make extensive use of the theory of parametrized-$\infty$-categories of
Barwick--Dotto--Glasman--Nardin--Shah, developed in \cite{GCats_intro},
\cite{Expose1}, \cite{Expose2}, \cite{Expose4}, \cite{Shah_thesis},
\cite{Nardin_thesis}. \Cref{sec:prelim} is dedicated to an exposition of the
language.

\paragraph{Notation.}

We use underlines to indicate parametrized notions and constructions.
For example, we write $\ulTopG$ for the $G$-$\infty$-category of $G$-spaces and $\ulFin^G_*$ for the $G$-$\infty$-category of finite based $G$-spaces.
Following \cite{HA}, we write $\Sp$ for the $\infty$-category of spectra, and use the symbol $\otimes$ to denote the smash product symmetric monoidal structure on $\Sp$.
We use the notation $E \otimes \Sigma^\infty_+ X$ or $E \otimes X$ for the smash product of a spectrum $E$ and a space $X$ (exhibiting $\Sp$ as tensored over spaces). We denote the wedge product of spectra by $\oplus$.
We write $\Top^G$ for the $\infty$-category  of
$G$-spaces and $G$-equivariant maps; write $\Sp^G$ for the $\infty$-category of genuine
$G$-spectra and $G$-equivariant maps, and denote smash products of genuine $G$-spectra by $\otimes$.
Finally, we denote the smash product of $E\in \Sp^G$ and $X\in \Top^G$ by $E \otimes \Sigma^\infty_+ X$ or $E \otimes X$.

\paragraph{Acknowledgments.} We would like to thank Jeremy Hahn and Dylan Wilson for their contribution to this paper, as well as for many helpful conversations on these topics. We would also like to thank Peter Bonventre for sharing a draft of his paper, Mike Hill for helpful remarks on Snaith splittings, and Mona Merling for an illuminating discussion during an early part of this project. The authors would like to thank the Isaac Newton Institute for Mathematical Sciences for support and hospitality during the programme ``Homotopy harnessing higher structures'', when work on this project was started. 
This work was supported by EPSRC grant number EP/R014604/1.
The first author acknowledges support by ERC-2017-STG 75908 to D. Petersen, and 
by ISF grant 87590021 to I. Dan-Cohen.

\section{From topological functors to G-symmetric monoidal functors}
\label{sec:topo-func}
In \cref{sec:prelim}, we state some preliminaries on $G$-$\infty$-categories
  and $G$-symmetric monoidal structures, which give correct $\infty$-categorical
  language for working with $G$-spaces and genuine $G$-$\EE_{\infty}$ ring spectra.
In \cref{app:top_G_objs}, we review the genuine operadic nerve construction of
\cite{Bonventre}, for the case of a symmetric monoidal topological category
(considered as a topological operad with many objects). Applying the construction to
$\mfld_n$,  we get another equivalent definition of the $G$-symmetric monoidal $G$-$\infty$-category
$\GmfldD$. 
In \cref{sec:one-point}, we use the functoriality of the operadic nerve construction to produce a $G$-$\infty$-categorical version of the one-point compactification functor $(-)^+ \colon \mfld_n \to \Top_*$.
The resulting $G$-symmetric monoidal functor $  (-)^+ \colon \GmfldD \to (\ul{\Top}^{G,\vee}_*)^{vop}$ will play a central role in expressing the right hand side of
\cref{thm:NPD} as a $G$-symmetric monoidal functor.

\subsection{Preliminaries on $G$-$\infty$-categories}
\label{sec:prelim}
We use $\OG$ to denote the (discrete) $\infty$-category of $G$-orbits, i.e. transitive
$G$-sets and $G$-maps.
We have the $\infty$-category of $G$-spectra, denoted as $\Sp^G$. For the purpose of describing
the genuine $\EE_{\infty}$-ring structure, we must consider a $G$-spectrum with all
its restrictions to subgroups $H \subset G$. Therefore, it is convenient to consider
the $G$-$\infty$-category $\ul{\Sp}^G$.
\begin{mydef}
  A $G$-$\infty$-category is a coCartesian fibration (as in \cite[def. 2.4.2.1]{HTT}) over the $\infty$-category $\OGop$.
\end{mydef}
\noindent Roughly speaking, this is a diagram of $\infty$-categories $\OGop \to
\Cat_{\infty}$. It is a special case of a parametrized $T$-$\infty$-category as
in \cref{sec:Parametrized_prelim}. 
\begin{notation}
\label{notn:fibC}  For a $G$-$\infty$-category $\ul\C$ and $O\in \OGop$, let $\ul\C_{[O]}$ be the fiber of
  $\ul\C \fib \OGop$ over $O$. 
\end{notation}
\begin{ex}(\cite[Example 1.4]{Nardin_thesis})
  Let $\Top^G$ be the $\infty$-category of $G$-spaces and $G$-maps.
  The $G$-$\infty$-category $\ul{\Top}^G$ (which is denoted $\ul{\Top}_G$ in the reference)
  has the following datum:
\begin{itemize}
\item For each $G/H \in \OGop$, an $\infty$-category $\ul{\Top}^G_{[G/H]} \simeq \Top^H$;
\item For each $\varphi: G/K \to G/H$, coCartesian lifts along $\varphi$ induce maps
  $$\varphi^{*}: \ul{\Top}^G_{[G/H]} \simeq \Top^H  \to \ul{\Top}^G_{[G/K]} \simeq \Top^K$$ given by
  restrictions of group actions and conjugations.
\end{itemize} And similarly for the $G$-$\infty$-category of $G$-spectra, $\ul{\Sp}^G$.
\end{ex}

  \begin{rem}
    Explicitly, one takes $\ul{\Top}^G_{[G/H]} = \Top^G/(G/H)$. That is, an
    object is $(X,f)$ where $X$ is a $G$-space and $f: X \to G/H$ is a
    $G$-map; a morphism from $(X,f)$ to $(X',f')$ is a $G$-equivariant
      map $X \to X'$ over $G/H$.
 
\begin{itemize}
\item    Then $X_0 = f^{-1}(eH)$ is a $H$-space and $X \cong G \times_H X_0$;
\item  $\Map_{\ul{\Top}^G_{[G/H]}}(G\times_HX_0, G\times_HX_0') \cong
      \Map_H(X_0,X_0')$;
\item For a projection $\varphi: G/K \to G/H$, $\varphi^{*}(X,f)$
    is the pullback, and there is $\varphi^{*}X \cong G \times_K X_0$
    (via a shearing isomorphism),
    so $(\varphi^{*}f)^{-1}(eK) = \mathrm{res}^H_K X_0$.
\end{itemize}  \end{rem}

  \begin{rem}
    Suppose $X, Y \in \ul\Top^G_{[G/G]}$ are $G$-spaces. Then $\Map_{\ul{\Top}^G_{[G/G]}}(X,Y)$ is the
    space of $G$-equivariant maps. The projection $p:G/e \to G/G$ gives
    $p^{*}X = G/e \times X \cong G \times_e X
    \in\ul\Top^G_{[G/e]}$. Then $\Map_{\ul\Top^G_{[G/e]}}(p^{*}X, p^{*}Y) \cong
    \Map_{\mathrm{id}_{G/e}}(G\times_eX , G\times_e Y)$ is the space of
    non-equivariant maps from $X$ to $Y$. Conjugation by $\varphi: G/e \to G/e$ endows
    a $G$-action on this mapping space. This matches with the conjugation
    $G$-action on $\Map_e(X,Y)$ because of the shearing isomorphism.
  \end{rem}

We use $\GFin_*$ to denote the (discrete) $G$-$\infty$-category of finite pointed $G$-sets
(see \cite[def. 2.14]{Nardin_thesis}). Explicitly, we have that:
\begin{itemize}
\item An object of $\GFin_*$ is given by a map of finite $G$-sets $U \to O$ where
  $O$ is an orbit.
\item A morphism $\psi:I_0 \to I_1$ in $\GFin_*$ from $I_0=(U_0 \to O_0)$ to $I_1=(U_1 \to O_1)$ is given by a span of arrows (a diagram of $G$-sets) of the form 
\begin{align} \label{eq:psi_diagram}
  \xymatrix{
    U_0 \ar[d] & U_0' \ar[d] \ar[l]_{f} \ar[r]^{p} & U_1 \ar[d] \\
    O_0 & O_1 \ar[r]^{=} \ar[l]_{\varphi} \ar[r]^{=} & O_1
  }
\end{align}
where the induced map $U_0' \to \varphi^* U_0$ is injective (or equivalently
there exists another $G$-set $U''_0$ with a $G$-map $U''_0 \to \varphi^* U_0$
that induces an isomorphism $  U'_0 \coprod U''_0 \xto{\=} \varphi^* U_0$).
\item Compositions of morphisms are given by pullbacks, details of which we defer
  to  \cref{construction:topological_G_objects}.
\end{itemize}
Note that there is an inclusion

  $\OGop \to \GFin_*$ by sending $O$ to $(O = O)$ and sending $O_0
  \overset{\varphi}{ \longleftarrow } O_1$ to
  \begin{align*} 
  \xymatrix{
    O_0 \ar[d]^{=} & O_1 \ar[d]^{=} \ar[l]_{\varphi} \ar[r]^{=} & O_1 \ar[d]^{=} \\
    O_0 & O_1 \ar[r]^{=} \ar[l]_{\varphi} \ar[r]^{=} & O_1.
  }
\end{align*}
\begin{ex}
  For $G=e$, $\ul{\Fin}_{*}^e \cong \Fin_{*}$. The isomorphism sends $I = (U \to *)$
  to $U \amalg \{0\}$; a morphism given as \eqref{eq:psi_diagram} uniquely determines a
  morphism $U_0 \amalg \{0\} \to U_1  \amalg \{0\}$ by sending $U_0' \subset U_0$ 
  to $U_1$ via $f$ and sending $U_0'' \amalg \{0\}$  to $0 $. We write objects
  of $\Fin_{*}$ as $\langle n \rangle :=\{0,1,\cdots,n\}$ with base point $0$, where $n \geq 0$.
\end{ex}
\medskip
Non-equivariantly, a symmetric monoidal $\infty$-category is a coCartesian fibration
$$p: \C^{\otimes} \fib \Fin_{*}$$ satisfying the Segal condition that certain maps
$\C^{\otimes}_{\langle n \rangle} \to (\C^{\otimes}_{\langle 1 \rangle})^n$ are equivalences. The underlying
$\infty$-category is $\C : = \C^{\otimes}_{\langle 1 \rangle}$. The coCartesian lifts of
$\alpha_n: \langle n \rangle \to \langle 1 \rangle$ where $\alpha_n(i) = 1$ for $1 \leq i \leq n$ give  multiplications
$(\C^{\otimes}_{\langle 1  \rangle})^n \simeq \C^{\otimes}_{\langle n \rangle}  \to \C^{\otimes}_{\langle 1 \rangle}$.
Commutative algebra in $\C^{\otimes}$ are sections of
$p$ sending inert morphisms (i.e. $f: \langle m \rangle \to \langle n \rangle$ such that
$|f^{-1}\{i\}|=1$ for $1 \leq i \leq n$) to $p$-coCartesian edges. Intuitively, this
amounts to picking $X_{\langle n \rangle} \in \C^{\otimes}_{\langle n \rangle}$ and multiplications
$(X_{\langle 1 \rangle})^n \simeq X_{\langle n \rangle} \to X_{\langle 1 \rangle}$. 
Commutative algebras in the symmetric monoidal $\infty$-category $\Sp^{\otimes}$ are $\EE_{\infty}$-ring spectra.

Equivariantly, there is a hierarchy of commutativity for ring $G$-spectra.
Intuitively, a naive $\EE_{\infty}$-ring $G$-spectrum $E$ has a
$G$-equivariant multiplication $\otimes_2E \to E$ that is commutative up to coherent
homotopies. A genuine $\EE_{\infty}$, or $G$-$\EE_{\infty}$ ring spectrum , has
additional $G$-equivariant multiplications indexed by $G$-orbits $\otimes_{G/H} E \to E$,
where $\otimes_{G/H} E = \mathrm{N}^G_H \mathrm{res}_H^G E$ is given by the
Hill--Hopkins--Ravanel norm.

\begin{mydef}
  [\cite{Parametrized_algebra}, combination of def. 2.2.3. and prop. 2.2.6]
  \label{defn:g-symmetric-monoidal}
  A $G$-symmetric monoidal $G$-$\infty$-category is a coCartesian fibration $ p
  \colon \ul\C^\otimes \fib \GFin_*$ such that the following \textbf{$G$-Segal
    condition} is satisfied:
  \begin{center}
   For every $I=(U\to O) \in \GFin_*$,  the map
$ \ul\C^\otimes_I \to \prod_{W \in Orb(U)} \ul\C^\otimes_{[W]}$ induced by the morphisms
\begin{align*}
  \xymatrix{
    U \ar[d] & W \ar[d] \ar[l]_f \ar[r]^{=} \ar[d]^{=} & W \ar[d]^{=} \\
    O & W \ar[r]^{=} \ar[l] \ar[r]^{=} & W,
  }
\end{align*}
in $\GFin_*$ is an equivalence,
where $f:W \to U$ is an inclusion of the orbit $W$.
\end{center}
\end{mydef}
\begin{notation}
  \label{notn:fibCTensor}
  Let $ p  \colon \ul\C^\otimes \fib \GFin_*$ be a coCartesian fibration.
\begin{itemize}
\item For $I\in \GFin_*, \, I=(U\to O)$, let $\ul\C^\otimes_I$ be the fiber of $p$ over $I$. 
\item For $W$ a $G$-orbit, let $\ul\C^\otimes_{[W]}$ be the fiber of $p$ over $(W \xto{=} W)$.
\end{itemize}
\end{notation}
\noindent The underlying $G$-$\infty$-category of $ p \colon \ul\C^\otimes \fib \GFin_*$
is the coCartesian fibration
$$\ul\C:=\ul\C^\otimes \times_{\GFin_*} \OGop \fib \OGop$$
defined by pulling back $p$ along the inclusion  $\OGop \to \GFin_*$.
Note that $\ul\C_{[G/H]}$ as in \cref{notn:fibC} coincides with $\ul\C^{\otimes}_{[G/H]}$ as
in \cref{notn:fibCTensor}.

\begin{ex}
  Let $\ul\Sp^G \fib \OGop$ be the $G$-$\infty$-category of genuine $G$-spectra, see \cite{Nardin_thesis}.
  There is an essentially unique $G$-symmetric monoidal structure on $\ul\Sp^G$
  with the sphere spectrum as unit, see \cite[cor. 3.28]{Nardin_thesis}.
  By construction, $(\ul\Sp^G)^\otimes$ is a distributive $G$-symmetric monoidal $G$-$\infty$-category (in other words, parametrized smash products distribute over parametrized colimits).
  Informally, this $G$-symmetric monoidal structure encodes smash products and Hill-Hopkin-Ravenel norms. 
  In $(\ul\Sp^G)^\otimes \fib \ul\Fin_{*}^{G}$,
  there is $(\ul\Sp^G)^\otimes_{[G/H]} \simeq
  \Sp^H$ for $G/H \in \OG$, and moreover, 
  $(\ul\Sp^G)^\otimes_{(\nabla_n)} \simeq (\Sp^H)^n$ for $n\geq1$ and $\nabla_n: \amalg_n G/H \to G/H$ the fold map. More
  interestingly, for $K \subset H \subset G$ and projections $\varphi: G/K \to G/H$,
\begin{itemize}
\item The $G$-Segal map $(\ul\Sp^G)^\otimes_{(\varphi)} \xrightarrow{\sim} (\ul\Sp^G)^\otimes_{[G/K]} (\simeq \Sp^K)$ is
  given by a coCartesion lift along the following morphism in  $\ul{\Fin}_{*}^G$:
\begin{equation}
\label{eq:Segal}
   \xymatrix{
    G/K \ar[d]_{\varphi} & G/K \ar[d]_{=} \ar[l]_{=} \ar[r]^{=} & G/K \ar[d]_= \\
    G/H & G/K \ar[r]^{=} \ar[l]_{\varphi} \ar[r]^{=} & G/K.
  }
\end{equation}

\item The restriction $\mathrm{res}_{K}^{H}: \Sp^H \simeq (\ul\Sp^G)^\otimes_{[G/H]} \to
  (\ul\Sp^G)^\otimes_{[G/K]}\simeq \Sp^K$ is given by a coCartesion lift along the
  following morphism in  $\ul{\Fin}_{*}^G$:
\begin{equation}
\label{eq:res}
\xymatrix{
    G/H \ar[d]_{=} & G/K \ar[d]_{=} \ar[l]_{\varphi} \ar[r]^{=} & G/K \ar[d]_= \\
    G/H & G/K \ar[r]^{=} \ar[l]_{\varphi} \ar[r]^{=} & G/K.
  }
\end{equation}
 
\item The norm map $\mathrm{N}_{K}^{H}: \Sp^K \simeq (\ul\Sp^G)^\otimes_{(\varphi)} \to
  (\ul\Sp^G)^\otimes_{[G/H]}\simeq \Sp^H$ is given by a coCartesian lift along the
  following morphism in $\ul{\Fin}_{*}^G$:
\begin{equation}
\label{eq:norm}
  \xymatrix{
    G/K \ar[d]_{\varphi} & G/K \ar[d]_{\varphi} \ar[l]_{=} \ar[r]^{\varphi} & G/H \ar[d]_= \\
    G/H & G/H \ar[r]^{=} \ar[l]_{=} \ar[r]^{=} & G/H.
  }
\end{equation}
\end{itemize}
\end{ex}

  \begin{rem}
    In a general $G$-symmetric monoidal $G$-$\infty$-category $\ul\C^{\otimes}$ and
    $K \subset H \subset G$, we refer to the map  $\ul\C_{[G/K]} \to \ul\C_{[G/H]}$ induced by
    \eqref{eq:Segal} and \eqref{eq:norm} as the induction map.
 One can define a naive symmetric monoidal $G$-$\infty$-category as a
 $\OGop$-diagram of symmetric monoidal $\infty$-categories. Such gadgets will not
 have induction maps, and can be modeled by coCartesian fibrations over the
 subcategory of $\ul{\Fin}_{*}^G$ spanned by $I$ of the form $(\amalg_n O \to O)$ for
 $n \geq 0$.
\end{rem}

\subsection{The $G$-$\infty$-category of topological $G$-objects}
\label{app:top_G_objs}
Let $\C$ be a topological category.
Let $\otimes$ be an enriched symmetric monoidal structure\footnote{See the following mathoverflow post: https://mathoverflow.net/questions/51783/enriched-monoidal-categories \cite{Enriched_SM_categories}.
In Kelly's book as linked in the post, the tensor product of $\mathcal{V}$-enriched
categories is defined on page 12.} on $\C$ with unit $I$.
We refer to such an enriched symmetric monoidal category as a \myemph{symmetric
  monoidal topological category}.
It is natural to consider objects in $\C$ with $G$-action or action by a
  subgroup of $G$, which we refer
  to as \myemph{topological $G$-objects in $\C$.}
  Such gadgets should admit  restrictions and $G/H$-indexed tensor products.

The goal of this subsection is to make this rigorous using the mechanism of
action groupoids and monoidal pushforward functors.  We will construct a
topological category $\ul\C^\otimes$ over $\GFin_*$ of topological $G$-objects in $\C$ 
(\cref{construction:topological_G_objects}) and  will
prove:
\begin{thm}
  \label{thm:operaded-nerve-topological-cat}
  Let $N(\ul\C^\otimes)$ denote the coherent nerve of $\ul\C^\otimes$ (as in
\cite[def. 1.1.5.5]{HTT}). 
  Then 
  \[
    N(\ul\C^\otimes) \to \GFin_*
  \]
  is a $G$-symmetric monoidal $G$-$\infty$-category. 
\end{thm}

\medskip
We start with some preliminaries needed for constructing $\ul\C^{\otimes}$.

Let $I$ be a small category.
Make $\Fun(I,\C)$ into a topological category by endowing the set $\Nat(F,G)$ of
natural transformations between $F,G:I\to \C$ with the topology from being the
equalizer of \( \prod_{i\in I} \Map_{\C}(Fi, Gi) \rightrightarrows
\prod_{\phi:i\to i'} \Map_{\C}(Fi, Gi') \), where one map is induced by precomposition with $F(\phi)$ and the other by postcomposition with $G(\phi)$.

A covering map $p:I\to J$ induces a monoidal pushforward functor $p^\otimes_*:
\Fun(I,\C) \to \Fun(J,\C)$ (see \cite[A.3.2]{HHR} or \cite[3.2.4]{Rubin}).
\begin{prop}
  The monoidal pushforward functor $p^\otimes_*: \Fun(I,\C) \to \Fun(J,\C)$ is a topological functor.
\end{prop}
\begin{proof}
Let $F, G \in \Fun(I, \C)$ be functors. The mapping space $\Nat(p^\otimes_* F, p^\otimes_* G)$ is
obtained as the equalizer of $$ \prod_{j\in J} \Map_{\C}(\otimes_{pi=j}Fi, \otimes_{pi=j}Gi)
\rightrightarrows \prod_{\phi:j\to j'} \Map_{\C}(\otimes_{pi=j}Fi, \otimes_{pi' = j'} Gi'). $$
Because $\C$ is a symmetric monoidal topological category, the maps
$$\prod_{pi=j} \Map_{\C}(Fi,Gi) \to \Map_{\C}(\otimes_{pi=j}Fi, \otimes_{pi=j} Gi)$$
are continuous, and similarly with $i',j'$.
Thus $\Nat(F, G) \to \Nat(p^\otimes_* F, p^\otimes_* G)$ is continuous.
\end{proof}

Let $U$ be a $G$-set. The \textbf{action groupoid} of $U$, denoted $B_U G$, has
objects $x \in U$ and morphisms $\Hom(x,y)= \set{g\in G | gx =y}$. Write $BG$ for $B_{G/G}G$.
A map of $G$-sets $f:U \to V$ induces a functor on the action groupoids, which we
denote by $Bf:B_UG \to B_V G$. 

Note that the action groupoid of a pullback of $G$-sets is isomorphic to the strict pullback of action groupoids, 
\begin{align*}
  P \= X\times_Z Y \Rightarrow B_P G \= B_X G \times_{B_Z G} B_Y G.
\end{align*}

\paragraph{Constructing a topological category over $\GFin_*$.} 

\begin{construction} \label{construction:topological_G_objects}
  Let $\C$ be a symmetric monoidal topological category.
  We construct a topological category $\ul\C^\otimes$ over $\GFin_*$ as follows. 
  \begin{itemize}
    \item An object $x\in\ul\C^\otimes$ over $I\in\GFin_*, I=(U\to O)$ is a functor $x: B_U G \to \C$.
    \item Let $x\in \ul\C^\otimes$ be an object over $I_0=(U_0 \to O_0)$,
       $y\in \ul\C^\otimes$ be over $I_1=(U_1\to O_1)$, and $\psi:I_0 \to I_1$ be a morphism of $\GFin_*$ given by the diagram \eqref{eq:psi_diagram}.
      Define the space of morphisms of $\ul\C^\otimes$ from $x$ to $y$ over $\psi$ to be $\Map^\psi_{\ul\C^\otimes}(x,y) = \Nat(p^\otimes_*f^* x, y)$.
      Here,  $f^* x$ is the composition $B_{U'_0} G
      \xto{f} B_{U_0} G \xto{x} \C$, and  $p^\otimes_* f^* x: B_{U_1} G \to \C$
      is the monoidal pushforward of $f^{*}x$ along $Bp:B_{U'_0} G \to B_{U_1} G$. This is defined because $Bp: B_{U'_0} G \to B_{U_1} G$ is always a covering map.
    \item Define the mapping space in the topological category $\ul\C^\otimes$ as $\Map_{\ul\C^\otimes}(x,y) = \coprod_\psi \Map^\psi_{\ul\C^\otimes}(x,y)$, where the coproduct is indexed over all $\psi\in \Hom_{\GFin_*}(I_0,I_1)$.

    \item Let $x_0,x_1,x_2 \in \ul\C^\otimes$ be objects over $I_0,I_1,I_2\in \GFin_*$.
      In what follows we construct continuous maps 
      \begin{align*}
        \Map^{\psi_1}_{\ul\C^\otimes}(x_0,x_1) \times \Map^{\psi_2}_{\ul\C^\otimes}(x_1,x_2)  \to \Map^{\psi_2\psi_1}_{\ul\C^\otimes}(x_0,x_2) ,
      \end{align*}
      for each $\psi_1:I_0 \to I_1, \, \psi_2: I_1\to I_2$ in $\GFin_*$.
      This allows us to define the composition map
      $$\Map_{\ul\C^\otimes}( x_0,x_1) \times \Map_{\ul\C^\otimes}(x_1,x_2) \to \Map_{\ul\C^\otimes}(x_0,x_2) $$
      as the coproduct of these maps.
      In other words, we make sure that  \( \ul\C^\otimes \to \GFin_* \) respects compositions by definition. 

      We first choose an explicit description of the composition $\psi_2 \psi_1: I_0 \to I_2$.
      Let $I_0= (U_0 \to O_0), \, I_1 = (U_1 \to O_1), \, I_2 = (U_2 \to O_2) $ and 
      \begin{align}\label{eq:psi1psi2}
        \psi_1 = \left( \vcenter{\xymatrix{
            U_0 \ar[d] & U_0' \ar[d] \ar[l]_{f_1} \ar[r]^{p_1} & U_1 \ar[d] \\
            O_0 & O_1 \ar[r]^{=} \ar[l]_{\varphi_1} \ar[r]^{=} & O_1
          }}
        \right), \quad 
        \psi_2 = \left( \vcenter{\xymatrix{
            U_1 \ar[d] & U_1' \ar[d] \ar[l]_{f_2} \ar[r]^{p_2} & U_2 \ar[d] \\
            O_1 & O_2 \ar[r]^{=} \ar[l]_{\varphi_2} \ar[r]^{=} & O_2
          }}
        \right).
      \end{align}
      The composition $\psi_2 \psi_1$ is given by 
      \begin{align} \label{eq:GFin_composotion}
        \psi_2 \psi_1 = \left( \vcenter{\xymatrix{
            U_0 \ar[d] & U_2' \ar[d] \ar[l]_{f_1 \ol{f}_2} \ar[r]^{p_2 \ol{p}_1} & U_2 \ar[d] \\
            O_0 & O_2 \ar[r]^{=} \ar[l]_{\varphi_1 \varphi_2} \ar[r]^{=} & O_2
          }}
        \right), \quad 
        \vcenter{\xymatrix{
          U'_2 \ar[r]^{\ol{p}_1} \ar[d]_{\ol{f}_2} \pullbackcorner & U'_1 \ar[d]_{f_2} \ar[r]^{p_2} & U_2 \\
          U'_0 \ar[r]^{p_1} \ar[d]_{f_1} & U_1 \\
          U_0
        } }
      \end{align}
      where the maps $\overline{f}_2: U'_2 \to U'_0, \, \overline{p}_1 : U'_2 \to U'_1 $ are given by the pullback square of finite $G$-sets in \eqref{eq:GFin_composotion}. 
      Note that the pullback square of diagram \eqref{eq:GFin_composotion} induces a strict pullback square of action groupoids in the following diagram of groupoids
      \begin{align*}
        \xymatrix{
          \pullbackcorner B_{U'_2} G \ar[r]^{\ol{p}_1} \ar[d]_{\ol{f}_2} & B_{U'_1} G \ar[d]_{f_2} \ar[r]^{p_2} & B_{U_2} G \\
          B_{U'_0} G \ar[r]^{p_1} \ar[d]_{f_1} & B_{U_1} G \\
          B_{U_0} G.
        }
      \end{align*}
      By \cite[prop. A.31]{HHR} it follows that the following diagram commutes up to natural isomorphism (given by the symmetric monoidal structure of $\C$)
      \begin{align*}
        \xymatrix{
          \Fun(B_{U'_2} G, \C) \ar[r]^{(\ol{p}_1)^\otimes_*} & \Fun(B_{U'_1} G, \C) \ar[r]^{(p_2)^\otimes_*} & \Fun(B_{U_2} G, \C) \\
          \Fun(B_{U'_0} G, \C) \ar[r]^{(p_1)^\otimes_*} \ar[u]^{(\ol{f}_2)^*} & \Fun(B_{U_1} G, \C) \ar[u]^{(f_2)^*} \\
          \Fun(B_{U_0} G, \C) \ar[u]^{(f_1)^*}.
        } 
      \end{align*}
      In particular, for \( x_0\in \Fun(B_{U_0} G, \C) \) we get a natural isomorphism 
      \begin{align} \label{eq:parametrized_indexed_tensor_products}  
        (p_2)^\otimes_* (f_2)^* (p_1)^\otimes_* (f_1)^* x_0 \= 
        (p_2)^\otimes_* (\ol{p}_1)^\otimes_* (\ol{f}_2)^* (f_1)^* x_0 \=
        (p_2 \ol{p}_1)^\otimes_* (f_1 \ol{f}_2)^* x_0 ,
      \end{align}
      where the second isomorphism is given by \cite[prop. A.29]{HHR}.
      Note that the mapping spaces of the topological functor categories in \eqref{eq:parametrized_indexed_tensor_products} are the spaces of natural transformations, 
      so the functor \( (p_2)^\otimes_* (f_2)^* : \Fun(B_{U_1} G, \C) \to \Fun(B_{U_2} G, \C)\) induces a continuous map
      \begin{align} \label{eq:pushing_eta_1}
        \Nat \left( (p_1)^\otimes_* (f_1)^* x_0 ,x_1 \right) \to \Nat \left( (p_2)^\otimes_* (f_2)^* (p_1)^\otimes_* (f_1)^* x_0 , (p_2)^\otimes_* (f_2)^* x_1 \right)
      \end{align}
      We now define the map 
      \( \Map^{\psi_1}_{\ul\C^\otimes}(x_0,x_1) \times \Map^{\psi_2}_{\ul\C^\otimes}(x_1,x_2)  \to \Map^{\psi_2\psi_1}_{\ul\C^\otimes}(x_0,x_2) \)
      as the composition
      \begin{align} \label{eq:composition}
        \xymatrix{
          \Nat \left( (p_1)^\otimes_* (f_1)^* x_0, x_1 \right) \times \Nat \left( (p_2)^\otimes_* (f_2)^* x_1, x_2 \right) \ar[d] \\
          \Nat \left( (p_2)^\otimes_* (f_2)^* (p_1)^\otimes_* (f_1)^* x_0, (p_2)^\otimes_* (f_2)^* x_1 \right) \times \Nat \left( (p_2)^\otimes_* (f_2)^* x_1, x_2 \right) \ar[d]^-{\circ} \\
          \Nat \left( (p_2)^\otimes_* (f_2)^* (p_1)^\otimes_* (f_1)^* x_0, x_2 \right) \ar[d]^{\=} \\
          \Nat \left( (p_2 \ol{p}_1)^\otimes_* (f_1 \ol{f}_2)^* x_0, x_2 \right) ,
        }
      \end{align}
      where the first map is given by \eqref{eq:pushing_eta_1} on the first coordinate and the identity on the second, 
      the second map is the composition in $\Fun(B_{U_2} G,\C)$ and
      the last isomorphism is induced by \eqref{eq:parametrized_indexed_tensor_products}.
  \end{itemize}
  Associativity of the composition in $\ul\C^\otimes$ follows from \cite[prop. A.29]{HHR}.
\end{construction}
\begin{ex}
Applying \cref{construction:topological_G_objects} to $\Fin_{*}$ yields $\GFin_*$ (\cite[Example 6.16]{Bonventre}).
\end{ex}

\paragraph{Checking that $\ul\C^\otimes \to \GFin_*$ is a coCartesian fibration.}
\begin{lem}
  \label{lem:cocart}
  Let $\psi_1:I_0 \to I_1$ be a morphism of $\GFin_*$ given by \eqref{eq:psi1psi2} and $x \in \ul\C^\otimes$ over $I_0$, i.e., a functor \( x: B_{U_0} G \to \C \).
\begin{itemize}
\item Define $y: B_{U_1} G \to \C$ over $I_1$ by setting $y= (p_{1})^\otimes_* (f_{1})^* x$, and define
  $\ol{\psi_{1}} \in \Map^{\psi_1}_{\ul\C^{\otimes}}(x,y)$ as the identity natural
  transformation \( (p_{1})^\otimes_* (f_{1})^* x \xto{=}  (p_{1})^\otimes_*
  (f_1)^* x =y \).
\item Then for every $\psi_2 : I_1\to I_2$ in $\GFin_*$ given by~\eqref{eq:psi1psi2} and $t\in \ul\C^\otimes$ over $I_2$ the continuous
  map \( (\overline{\psi_{1}})^*:  \Map^{\psi_{2}}_{\ul\C^\otimes}(y,t)  \to  \Map^{\psi_{2}
    \psi_{1}}_{\ul\C^\otimes}(x,t) \) as defined in~\eqref{eq:composition} is an isomorphism.
\end{itemize} 
\end{lem}
\begin{proof}
 Explicitly, $(\ol{\psi_1})^{*}$ is the composite
    \begin{equation*}
   \Nat \left( (p_2)^\otimes_* (f_2)^* y, z \right) \overset{\ol{\psi_1}\circ}{\to}
        \Nat \left( (p_2)^\otimes_* (f_2)^* (p_1)^\otimes_* (f_1)^* x, t \right)  \overset{\cong}{
          \to } \Nat \left( (p_2 \ol{p}_1)^\otimes_* (f_1 \ol{f}_2)^* x, t \right).
    \end{equation*}
The first map is an isomorphism by the definition of $y$ and the second map is an isomorphism induced by~(\ref{eq:parametrized_indexed_tensor_products}) as before.
\end{proof}

\begin{cor} \label{top_G_objects_coCart_fib}
  The map $N(\ul\C^\otimes) \to \GFin_*$ is a coCartesian fibration.
\end{cor}
\begin{proof}
 \Cref{lem:cocart} implies that the square
  \begin{align*}
    \xymatrix{
      \Map_{\ul\C^\otimes}(y,t)  \ar@{->>}[d] \ar[r]^{\overline{\psi}^*} & \Map_{\ul\C^\otimes}(x,t) \ar@{->>}[d] \\
      \Hom_{\GFin_*}(I_1,I_2) \ar[r]^{\psi^*} & \Hom_{\GFin_*}(I_0,I_2) 
    }
  \end{align*}
  is homotopy Cartesian, since it induces weak equivalences on the homotopy fibers of the vertical maps.
  Therefore by \cite[prop. 2.4.1.10]{HTT} the morphism $\ol{\psi} : x\to y$ in $\ul\C^\otimes$ is a coCartesian lift of $\psi:I_0 \to I_1$ in  $\GFin_*$.
  We have shown that every $\psi:I_0 \to I_1$ and $x \in \ul\C^\otimes$ over $ I_0$ has a coCartesian lift $ \ol{\psi}$.
  Passing to the coherent nerve of $\ul\C$ (see \cite[def. 1.1.5.5]{HTT}) we know that the map $N(\ul\C^\otimes) \to \GFin_*$ is an inner fibration (again by \cite[prop. 2.4.1.10]{HTT}), so we have shown that it is coCartesian fibration by verifying \cite[def. 2.4.2.1]{HTT}. 
\end{proof}

\paragraph{Checking the $G$-Segal conditions.}
Let $I = (U \to O)$ be as in Definition \ref{defn:g-symmetric-monoidal}. Then $C^\otimes_I$ (\cref{notn:fibCTensor}) is the topological category with objects given by functors $x:B_U G \to \C$ and mapping spaces $\Map_{\ul\C^\otimes_I}(x,y)= \Map^{id_I}_{\ul\C^\otimes}(x,y) = \Nat(x,y)$.
\begin{rem}
  It is easy to see that if $W \= G/H$ then $B_W G \simeq BH$, hence $\ul\C^\otimes_{[W]}$ is equivalent to the topological category $\Fun(BH,\C)$ of $H$-objects in $\C$.
\end{rem}
Let $I=(U \to O)\in \GFin_*$ and $f:W \to U$ be an inclusion of orbit as in Definition \ref{defn:g-symmetric-monoidal}.
Our choice of coCartesian edges above implies that the $G$-Segal map is just the coherent nerve of
\begin{align*}
  f^*: \ul\C^\otimes_I \to \ul\C^\otimes_{[W]}, \quad (x: B_U G \to \C) \mapsto (f^* x: B_W G \xto{Bf} B_U G \xto{x} \C).
\end{align*}
Taking the product over all orbits $W\in \orb(U)$ we get a functor of topological categories \( \ul\C^\otimes_I \to \prod_{W\in \orb(U)} \ul\C^\otimes_{[W]} \)
whose coherent nerve is equivalent to the $G$-Segal map.  
Checking the $G$-Segal conditions amounts to proving
\begin{lem} \label{G_Segal_conds}
  The functor \( \ul\C^\otimes_I \to \prod_{W\in \orb(U)} \ul\C^\otimes_{[W]} \) is an equivalence.
\end{lem}
\begin{proof}
  This follows from the description of $\ul\C^{\otimes}_I$ and $f^{*}$ above.
\end{proof}

\begin{proof}[{Proof of \cref{thm:operaded-nerve-topological-cat} }]
  The map $N(\ul\C^\otimes) \to \GFin_*$ is a coCartesian fibration by \cref{top_G_objects_coCart_fib}, 
  and by \cref{G_Segal_conds} it satisfies the $G$-Segal conditions. 
\end{proof}

\subsection{The one point compactification functor}
\label{sec:one-point}
In this section, we upgrade the one point compactification functor to
a $G$-symmetric monoidal functor
\begin{align} \label{eq:OPC_as_GSM}
  (-)^+ \colon \GmfldD \to (\ul{\Top}^{G,\vee}_*)^{vop}.
\end{align}

We rely on the fact that one point compactification defines a functor of topological categories. 
\begin{construction} \label{const:one_point_compactification}
  Let $M\in \mfld_n$ be an $n$-dimensional manifold, and denote its one point compactification by $M^+ \in \Top_*$.
  Since morphisms in $\mfld_n$ are open embeddings of manifolds, one point compactification defines a functor \( (-)^+  \colon \mfld_n \to (\Top_*)^{op} \).
  Furthermore, $(-)^+$ takes disjoint unions to wedge sums, so it defines a symmetric monoidal
  functor $$ (-)^+ \colon (\mfld_n,\sqcup) \to ((\Top_*)^{op}, \vee). $$
  We consider both categories $\mfld_n$ and $\Top_*$ as \myemph{topological} symmetric monoidal categories using the compact open topology.
  By \cite{Continuity_of_collapse_maps}, this one point compactification is a functor of topological categories. 
\end{construction}

\begin{lem}
  Let $(\mfld_n,\sqcup)$ be the topological symmetric monoidal category of \cite[def. 2.1]{AF}.
  The $G$-symmetric monoidal $G$-$\infty$-category of topological $G$-objects in $(\mfld_n,\sqcup)$ is equivalent to the $G$-symmetric monoidal $G$-$\infty$-category $\GmfldD$ defined in \cite[sec. 3.4]{GFH}.
\end{lem}

\begin{proof}
  The proof follows from the definitions.
\end{proof}

\noindent Similarly, one can use the genuine operadic nerve construction to construct the $G$-$\infty$-category of pointed $G$-spaces with the $G$-coCartesian $G$-symmetric monoidal structure (see \cite[ex. 2.4.1]{Parametrized_algebra}).

The next statement uses the notion of an opposite $G$-$\infty$-category $\ul\C^{vop} \to \OGop$ of a $G$-$\infty$-category $\ul\C \to \OGop$, see \cite[def. 3.1]{Expose1}.
\begin{lem}
  The  $G$-symmetric monoidal $G$-$\infty$-category of topological $G$-objects in $((\Top_*)^{op},\vee)$ is equivalent to the $G$-$\infty$-category $(\ul{\Top}^{G, \vee}_*)^{vop}$, where the $G$-$\infty$-category of pointed $G$-spaces $\ul{\Top}^G_*$ is endowed with the $G$-coCartesian $G$-symmetric monoidal structure. 
\end{lem}

\begin{proof}
  The proof follows from the definitions.
\end{proof}

Since the construction of the $G$-symmetric monoidal $G$-$\infty$-category of topological $G$-objects is functorial, we can apply it to the one point compactification functor,
and get the $G$-symmetric monoidal functor \eqref{eq:OPC_as_GSM}.

\medskip
Next, we describe the $G$-symmetric monoidal functor $\ul{\Map}_*(-,X)$.
\begin{construction} \label{cons:Mapc_X_as_top_functor}
  Let \( X \in (\ul{\Top}^{G}_*)_{[G/G]} \) be a pointed topological $G$-space.
  Applying the parametrized Yoneda embedding of \cite[def. 10.2]{Expose1} to $X$ we get a $G$-functor
  \begin{align*}   
    \ul{\Map}_*(-,X) \colon (\ul{\Top}^{G}_*)^{vop} \to \ulTopG.
  \end{align*}
  By \cite[cor. 11.9]{Expose2} the functor $\ul{\Map}_*(-,X) \colon (\ul{\Top}^{G}_*)^{vop} \to \ulTopG$ preserves $G$-limits.
  Since the $G$-symmetric monoidal structures on $(\ul{\Top}^{G}_*)^{vop}$ and $\ulTopG$ are $G$-Cartesian it follows that $\ul{\Map}_*(-,X)$ extends to a $G$-symmetric monoidal functor 
  \begin{align} \label{eq:Map_X_as_GSM}
    \ul{\Map}_*(-,X) \colon (\ul{\Top}^{G,\vee}_*)^{vop} \to \ul{\Top}^{G,\times}.
  \end{align}
\end{construction}

Composing \eqref{eq:OPC_as_GSM} and \eqref{eq:Map_X_as_GSM}, 
  we get a $G$-symmetric monoidal functor 
  \begin{align}
    \label{eq:Mapc_X}
    \ul{\Map}_*\left( (-)^+,X \right) \colon \GmfldD \xto{(-)^+ } (\ul{\Top}^{G,\vee}_*)^{vop}  \xto{\ul{\Map}_*(-,X)} \ul{\Top}^{G,\times}.
  \end{align}

\begin{rem} \label{Mapc_X_as_top_functor}
  Restricting the $G$-symmetric monoidal functor of \eqref{eq:Mapc_X} to the fiber over $G/G\in \OGop$ defines a symmetric monoidal functor of $\infty$-categories 
  \begin{align*}
    \Map_{*} \left( (-)^+,X \right) \colon \mfldGD \xto{(-)^+ } (\Top^{G,\vee}_*)^{op}  \xto{\Map_*(-,X)} \Top^{G,\times}.
  \end{align*}
  Note that, to obtain this symmetric monoidal functor alone, we can simply apply the operadic nerve construction of \cite[2.1.1.23]{HA} to the corresponding symmetric monoidal \myemph{topological} categories and functors. Indeed, we can identify the $\infty$-categories and functors above as follows:
  \begin{itemize}
    \item The symmetric monoidal $\infty$-category $\mfldGD$ is equivalent to the operadic nerve of the topological category of smooth $n$-dimensional $G$-manifolds and $G$-equivariant smooth embeddings, with symmetric monoidal structure given by disjoint unions. 
    \item The symmetric monoidal $\infty$-category $\Top^{G,\vee}_*$ is equivalent to the operadic nerve of topological category of pointed $G$-CW spaces with the coCartesian monoidal structure (given by the wedge of pointed $G$-spaces). 
      The operadic nerve of the opposite topological category is equivalent to $\left( \Top^{G,\vee}_* \right)^{op}$. 
    \item The symmetric monoidal $\infty$-category $\Top^{G,\times}$ is equivalent to the operadic nerve of topological category of $G$-CW spaces with the Cartesian monoidal structure (given by the products of $G$-spaces). 
    \item The symmetric monoidal functor \( (-)^+ \colon \mfldGD \to (\Top^{G,\vee}_*)^{op} \) can be identified with the operadic nerve of the one point compactification functor, as in \cref{const:one_point_compactification}.
    \item The symmetric monoidal functor \(  \Map_*(-,X)  \colon (\Top^{G,\vee}_*)^{op} \to \Top^{G,\times} \) can be identified with the operadic nerve of the topological functor sending a pointed $G$-space $Y$ to the space of pointed maps $\Map_*(Y,X)$, with $G$-action given by conjugation.
  \end{itemize}
  We can therefore identify the $G$-space $\Map_{*}( M^+,X)$ with the space of pointed maps $M^+ \to X$, with $G$ acting by conjugation.
\end{rem}

\section{Genuine equivariant factorization homology}
\label{sec:GFH}
We give a quick introduction to genuine equivariant factorization homology and
its axiomatic characterization, and explain more details in the remainder of this section.

Let $A$ be in a $\EE_n$-ring spectrum and $M$ a framed manifold of dimension $n$.
The factorization homology of $A$ over $M$ is the spectrum $\int_M A$ given by the homotopy colimit of a diagram
\[
  (\disk^{fr}_n)_{/M} \to \Sp, \quad ( \sqcup_k \R^n \to M) \mapsto A^{\otimes k}
\]
sending an open embedding of $k$-disks in $M$ to the $k$-fold smash product of
$M$. The functor $\int_-A$ satisfies certain properties, which are taken to be
the the axiomatic definition of factorization homology theories in
\cite{AF}.

Let $G$ be a finite group and $V$ be a fixed $n$-dimensional
$G$-representation. There is a $G$-symmetric monoidal $G$-$\infty$-category
$\ul\mfld^{G,V-fr,\sqcup}$ of $V$-framed $G$-manifolds.
For a general $G$-symmetric monoidal $G$-$\infty$-category
$\ul\C^{\otimes}$, let $\ul{A}$ be a $\EE_V$-algebra in $\ul\C$, which is an equivariant
analog of a $\EE_n$-algebra.
The genuine equivariant factorization homology is a functor
\begin{equation*}
\ul\mfld^{G,V-fr} \to \ul\C, \, M \mapsto \int_M \ul{A},
\end{equation*}
which satisfies the following properties
\begin{enumerate}
\item It is $G$-symmetric monoidal.
  \item It sends $G$-collar decompositions to relative tensor products\footnote{here the relative tensor product is given by a two sided bar construction $|Bar(\int_{M_1} \ul{A}, \int_{M_0 \times \R} \ul{A}, \int_{M_2} \ul{A} )|$ in $\C_{[G/G]}$.}
    \[
      M_1 \cup_{M_0 \times \R} M_2 \mapsto 
      \int_{M_1} \ul{A} \otimes_{ \int_{M_0 \times \R} \ul{A} } \int_{M_2} \ul{A} 
    \]
  \item It sends sequential unions to colimits
    \[
      M = \cup_{i=0}^\infty M_i \mapsto \colim_i \int_{M_i} \ul{A} 
    \]
    where the $G$-submanifolds $M_i \subseteq M$ satisfy $\ol{M_i} \subset M_{i+1}$.
  \end{enumerate}
Let $\Fun^{G-\otimes}(\ul\mfld^{G,V-fr}, \ul\C)$ denote the full subcategory of
$G$-symmetric monoidal functors and $ \mathcal{H}(\ul\mfld^{G,V-fr}, \ul\C)$
denote the full subcategory of genuine factorization homologies.

\medskip
Throughout this paper we use several contexts of this construction:
\begin{enumerate}
  \item In \cref{sec:NAPD}, $\ul{\C}$ is the category of $G$-spaces.
  \item In \cref{sec:Thom_GFH}, $\ul{\C}$ is the category of $G$-spaces with a $G$-map to
   $\Pic(\ul\Sp^G)$, where $\Pic(\ul\Sp^G)$ is the $G$-space classifying equivariant stable spherical fibrations.
  \item In \cref{sec:computation}, $\ul{\C}$ is the category of genuine $G$-spectra.
  \end{enumerate}

\subsection{$V$-framed $G$-manifolds and $G$-disks}
 A smooth $G$-manifold is a smooth manifold with a $G$-action by diffeomorphisms. 
  We will assume all $G$-manifolds to be smooth and drop the word in this paper.
\begin{mydef}
  Let  $\mfld^G$ be the topological category of $n$-dimensional $G$-manifolds and
  $G$-equivariant open embeddings, endowed with the open compact topology.
  
  Applying the coherent nerve construction to $\mfld^G$ defines an $\infty$-category.
  By abuse of notation we will denote this $\infty$-category by $\mfld^G$.

  Disjoint union of $G$-manifolds defines a symmetric monoidal structure on $\mfld^G$.
\end{mydef}

\begin{mydef}
  A $V$-framing of a $G$-manifold $M$ is an equivariant isomorphism of vector bundles 
  \[
    \phi_M \colon TM \xto{\=} M \times V.
  \]
\end{mydef}

The tangent bundle of a $G$-manifold $M \in \mfld^G$ is a $G$-vector bundle, classified by a $G$-map $\tau_M \colon M \to BO_n(G)$ to the classifying space of $G$-vector bundles of rank $n$.
We can alternatively describe a $V$-framing on $M$ as a homotopy lift
\begin{equation}
  \begin{tikzcd}\label{diag:V-framing}
      & \ast \ar[d,"V"] \\
    M \ar[r,"\tau_M"] \ar[ur,"\phi_M"] & BO_n(G),
  \end{tikzcd}
\end{equation}
where $V \colon \ast \to BO_n(G)$ is the $G$-map classifying $V$, considered as a $G$-vector bundle over a point.
Note that the diagram above is precisely a morphism in $\Top^G/BO_n(G)$, the $\infty$-category of $G$-spaces over $BO_n(G)$.

Since tangent bundles pull back along open embeddings, the assignment $M \mapsto \tau_M$ can be assembled into a functor of $\infty$-categories $\tau \colon \mfld^G \to \Top^G$, see \cite[sec. 3.2]{GFH} for details.

\begin{mydef}
  The $\infty$-category of $V$-framed $G$-manifolds $\mfld^{G,V-fr}$ is given by the pullback 
  \begin{equation*}
    \begin{tikzcd}
      \mfld^{G,V-fr} \ar[d] \ar[r] \pbcorner & \Top^G / \ast \ar[d,"V"] \\
      \mfld^G \ar[r,"\tau"] & \Top^G / BO_n(G).
    \end{tikzcd}
  \end{equation*}
  One can extend disjoint union of $G$-manifolds to a symmetric monoidal structure on $\mfld^{G,V-fr}$. 
\end{mydef}
One can use the universal property of $BO_n(G)$ to get an explicit description
of a morphism in $\mfld^{G,V-fr}$ of $V$-framed 
$G$-manifolds in terms of trivialization's of $G$-vector bundles.
(See \cite[Sec 3.1]{Zou}. The spaces of morphisms here are the $G$-fixed points of the $G$-space of
  morphisms there.)

\begin{mydef}
  The $G$-manifold $V$ has a canonical $V$-framing, $can_V \colon TV \= V \times V$.
  For $U$ a finite $G$-set endow the $G$-manifold $ U \times V$ with the canonical $V$-framing
  \[
    T (U \times V) = U \times TV \xto{id_U \times can_V} U \times V \times V.
  \]
  Let $\disk^{G,V-fr} \subset \mfld^{G,V-fr}$ be the full subcategory spanned by
  $U \times V$ for all finite $G$-sets $U$.
 The disjoint union restricts to a symmetric monoidal structure on $\disk^{G,V-fr}$.
\end{mydef}

  \begin{mydef}(\cite[sec. 3.3]{GFH})
    The  $G$-$\infty$-category of $V$-framed $G$-manifolds is
    $$\ul{\mfld}^{G,V-fr} = \mfld^{G,V-fr} \times_{\Top^G} \ul\Top^G \fib \OGop.$$ 
  \end{mydef}
  This comes with a functor to $\ul{\mfld}^G$, which can be thought of as forgetting the framing. For a subgroup $H \subset G$, we can consider $V$ as an $H$-representation by restricting the action.
\begin{lem}
  There is an equivalence of $\infty$-categories
  \[
    \ul\mfld^{G,V-fr}_{[G/H]} \iso \mfld^{H,V-fr}, 
  \]
  which restricts to an equivalence  $ \ul\disk^{G,V-fr}_{[G/H]} \iso \disk^{H,V-fr} $.
\end{lem}
\begin{proof}
An object $(M \to G/H,\phi_M) \in \ul\mfld^{G,V-fr}_{[G/H]}$ is a $V$-framed $G$-manifold $(M,\phi_M)$ with a $G$-map $f \colon M \to G/H$.
Then the fiber $M_0 = f^{-1}(eH)$ is a $H$-manifold,
and  $\phi_M|_{M_0}$ is a $V$-framing. On the other hand, let $(M, \phi_M)$ be a $V$-framed $H$-manifold. Then $G \times_H M$
  is a $V$-framed $G$-manifold given by
\begin{equation*}
  T(G \times_H M) \cong G \times_H TM \overset{id \times \phi_M}{
    \longrightarrow } G \times_H(M \times V) \cong (G \times_HM) \times V.
\end{equation*}
And similarly for mapping spaces.
So
$$\ul{\mfld}^{G,V-fr}_{[G/H]} \simeq\mfld^{G,V-fr} \times_{\Top^G} \ul\Top^G_{[G/H]} \simeq \mfld^{G,V-fr} \times_{\Top^G} \Top^G/(G/H) \simeq \mfld^{H,V-fr}.$$
\end{proof}

The disjoint union gives a $G$-symmetric monoidal structure to the $G$-$\infty$-categories $\ul\disk^{G,V-fr}$ and
$\ul\mfld^{G,V-fr}$. (For details, see \cite[Sec 3.4]{GFH}.) The functor which forgets the framing extends to a $G$-symmetric monoidal functor $\ul\mfld^{G,V-fr,\sqcup} \to \ul\mfld^{G,\sqcup}$. We have $\ul\mfld^{G,V-fr,\sqcup}_{[G/H]} \simeq
\mfld^{H,V-fr}$. For $K \subset H \subset G$, the restriction maps $\mfld^{H,V-fr,\sqcup} \to
\mfld^{K,V-fr,\sqcup}$ are restrictions of actions; the induction maps
$\mfld^{K,V-fr,\sqcup} \to \mfld^{H,V-fr,\sqcup}$ are $K\times_H -$.

\subsection{$\EE_V$-algebras}
The notion of an $\EE_n$-algebra in symmetric monoidal $\infty$-category admits
an equivariant refinement, called an $\EE_V$-algebra. 
In \cref{sec:v-fold-loop} we will see two examples of $\EE_V$-algebras, $G$-commutative algebras and $V$-fold loop spaces.
\begin{mydef}[$\EE_V$-algebras] \label{def:Vdisk_alg}
  Let $V$ be a real $n$-dimensional representation of $G$ and $\ul\C$ a $G$-symmetric monoidal $G$-$\infty$-category.
  An {$\EE_V$-algebra} in $\ul\C$ is a $G$-symmetric monoidal functor 
  \begin{align*}
    \ul{A} \colon    \ul\disk^{G,V-fr,\sqcup} \to \ul\C^\otimes.
  \end{align*}
 Let $\Alg_{\EE_V}(\C)$ denote the $\infty$-category of $G$-symmetric
  monoidal functors $\Fun^{G-\otimes}(\ul\disk^{G,V-fr}, \ul\C)$. 
\end{mydef}
\begin{rem}
  In \cite{GFH} the definition above is referred to as a \myemph{$V$-framed $G$-disk algebra}, while the notion of an $\EE_V$-algebra is reserved for algebras over the $G$-$\infty$-operad $\EE_V$, constructed from the $G$-operad of little $V$-disks using a genuine operadic nerve construction of \cite[ex. 6.5]{Bonventre}.
  In this paper we chose not to distinguish between $\EE_V$-algebras and $V$-framed $G$-algebras, since these notions coincide. 
  Indeed, \cite[lem. 3.7.2, prop. 3.9.8]{GFH} shows that $\ul\disk^{G,V-fr}$ is the $G$-symmetric monoidal envelope of $\EE_V$.
\end{rem}

\subsection{Genuine equivariant factorization homology}
We define genuine equivariant factorization homology as a $G$-left Kan
  extension. (For details, see \cite[sec. 4.1]{GFH} and \cite[sec. 10]{Expose2}.)
\begin{mydef}
  [Genuine equivariant factorization homology]
  \label{defn:genu-equiv-fh}
  Let $\ul{A} \in \Alg_{\EE_V}(\ul{\C})$ be an $\EE_V$-algebra in a presentably $G$-symmetric monoidal $\infty$-category $\ul\C$.
  Define
  \[
    \int \ul{A} \colon \ul\mfld^{G,V-fr} \to \ul{\C}
  \]
  as the $G$-left Kan extension of $\ul{A} \colon \ul\mfld^{G,V-fr} \to \ul\Sp^G$ along the inclusion $\iota \colon \ul\disk^{G,V-fr} \subset \ul\mfld^{G,V-fr} $.
\end{mydef}
If $M \in \mfld^{G,V-fr}$ then 
\begin{equation}
\label{eq:FHdefn}
\int_M \ul{A}  \simeq G-\colim \left(
  \ul\disk^{G,V-fr}_{/\ul{M}} \to \ul\disk^{G,V-fr} \xto{\ul{A}} \ul\C
  \right)
\end{equation}
is a $G$-colimit in $\ul\C$.

We can express $\int_M \ul{A}$ as a colimit in the $\infty$-category $\ul\C_{[G/G]}$, which requires some preparation.
The functor
\[
  \Cat_{\infty,G} \to \Cat_\infty, \quad \ul\E \mapsto \ul\E_{[G/G]}
\]
has a left adjoint given by $\E \mapsto ( \E \times \OGop \to \OGop)$.
The counit of this adjunction is a natural $G$-functor
\[
  \ul\E_{[G/G]} \times \OGop \to \ul\E, \quad (x,G/H) \mapsto x(G/H)
\]
where $x(G/H)$ is the restriction of $x$ along the forgetful functor $\ul\E_{[G/G]} \to \ul\E_{[G/H]}$.
 We show that this $G$-functor preserves $G$-colimits for $\ul{\E}=\ul\disk^{G,V-fr}_{/\ul{M}}$.
\begin{lem}  \label{cofinal_disks_over_M_restrictions}
  The restriction map
\begin{equation}
\label{eq:diskCofinal}
    (\disk_G)_{/M} \to  (\disk_H)_{/M}
\end{equation}  is cofinal for every $H<G$.
\end{lem}
\begin{proof}
  Let $\D \subset \left(\disk_H\right)_{/M} $ denote the full subcategory spanned those $H$-embeddings $\varphi \colon D \to M$ for which the induced $G$-map
  \[
    G \times_H D \to M
  \]
  is an embedding.
  Observe that the inclusion $\D \subset \left(\disk_H\right)_{/M} $ is an equivalence.
  Indeed, every $\varphi \colon D \to M$ in $\left(\disk_H\right)_{/M}$ is
  equivalent to an object of $\D$ by decresing the radius of the disks.
  
  The restriction functor \eqref{eq:diskCofinal} clearly factors through $\D$, and 
  \[
    \left(\disk_G\right)_{/M} \to \D 
  \]
  is right adjoint, hence cofinal; its left adjoint is given by 
  \[
    \D \to \left(\disk_G\right)_{/M}, \quad ( \varphi \colon D \to M) \mapsto (G \times_H D \to M).
  \]
  The cofinality of \eqref{eq:diskCofinal} follows as it is a composition of a cofinal functor and an equivalence.
\end{proof}

\begin{prop}
  Let $M \in \mfld^{V-fr}_G$ be a $V$-framed $G$-manifold and $\ul{A} \in
  \Alg_{\EE_V}(\ul{\C})$ be an $\EE_V$-algebra.   Then the genuine equivariant factorization homology $\int_M \ul{A}$ is equivalent to the colimit of 
  \[
     (\disk^{V-fr}_G)_{/M} \to\disk^{V-fr}_G \xto{\ul{A}_{[G/G]}} \C_{[G/G]}. 
    \]
\end{prop}

\begin{proof}
   We first show that the $G$-functor
  \[
    (\disk^{V-fr}_G)_{/M} \times \OGop \to \ul\disk^{G,V-fr}_{/\ul{M}}
  \]
  is $G$-cofinal.  
  By \cite[prop. 4.2.6]{GFH} it is enough to prove that 
  \[
    (\disk_G)_{/M} \times \OGop \to \ul\disk^{G}_{/\ul{M}}
  \]
  is $G$-cofinal.
  However, $G$-cofinality is equivalent to fiberwise cofinality by
  \cite[thm. 6.7]{Expose2}, which is proved in
  \cref{cofinal_disks_over_M_restrictions}.

  So, the $G$-colimit in \eqref{eq:FHdefn} is equivalent to the $G$-colimit
  of a diagram 
\begin{equation*}
  (\disk^{V-fr}_G)_{/M} \times \OGop \to \ul{\C};
  \end{equation*} the conclusion then
  follows from \cite[prop. 5.8]{Expose2}.

\end{proof}

\subsection{Equivariant Morse theory }
Equivariant Morse theory of $G$-manifolds was developed by Wasserman in \cite{Wasserman}, for a compact Lie group $G$.
We will only use this theory for $G$ a finite group, which simplifies the resulting equivariant handlebody decompositions \footnote{in a nutshell, handlebundles are just disjoint unions of handles}.
Wasserman uses a $G$-invariant Morse function to construct finite equivariant handle decompositions for compact $G$-manifolds.  
The same techniques can be used to decompose a non-compact $G$-manifold $M$ (constructing $M$ inductively by attaching a finite number of handles at each stage), providing that the Morse function satisfies some reasonable assumptions.

Recall that a $G$-invariant Morse function on a $G$-manifold $M$ is simply a Morse function $f \colon M \to \R$ which is $G$-invariant.
It will be useful to assume that for every $a\in \R$ the submanifold $f^{-1}(-\infty,a]$ is compact.
Such a function exists by the results of~\cite{Hepworth}.
\begin{thm}[Hepworth] \label{Hepworth_Morse}
  Every $G$-manifold $M$ admits a $G$-invariant Morse function $M \to \R$ such that $f^{-1}(-\infty,a]$ is compact for each $a\in \R$.
\end{thm}
\begin{proof}
  By \cite[ex. 3.2]{Hepworth} $M$ defines a differentiable Deligne-Mumford stack $\mathfrak{M}$.
  Apply \cite[cor. 6.12]{Hepworth} to $\mathfrak{M}$.
\end{proof}
Since $G$ is finite, the Morse lemma implies that the critical points of $f$ are isolated and decompose into $G$-orbits (see \cite[cor. 4.8]{Hepworth}).
We will always assume that our $G$-manifold is equipped with a complete $G$-invariant Riemannian metric.
\begin{lem}
  Every a $G$-manifold $M$ admits a complete $G$-invariant Riemannian metric.
\end{lem}
\begin{proof}
  This follows from examining the proof of theorem 1 of~\cite{nomizu1961existence}.
  Let $(,)$ be a $G$-invariant Riemannian metric on $M$.
  If $(,)$ is not complete, then~\cite{nomizu1961existence} construct modified metric $(-,-)' = \frac{1}{r(x)}(-,-)$.
  By inspection the function $r(x) = \sup\set{ r>0 \mid \set{y|d(x,y)<r} \text{ is relatively compact} }$ is $G$-invariant, hence the complete metric $(-,-)'$ is also $G$-invariant.
\end{proof}

We can use $G$-Morse functions as above to express any $G$-manifold as a sequential union.
\begin{construction} \label{G_Morse_decomp}
  Let $M$ be a $G$-manifold and $f \colon M \to \R$ a $G$-invariant Morse function as in \cref{Hepworth_Morse}.
  Choose a sequence of regular values $a_0 < a_1 < \cdots$ such that $f^{-1}(-\infty,a_0] = \emptyset$ and the interval $[a_{i-1},a_i]$ contains a single critical value for every $i>0$.
  Define $G$-submanifolds 
  \[
    M_i = \set{ x\in M | f(x) < a_i}  \subseteq M.
  \]
  Then each $M_i$ is the interior of a compact $G$-manifold with boundary $\overline{M_i} = \set{ x\in M | f(x) \leq a_i}$ and $M= \cup_i M_i$ is a sequential union of $G$-manifolds.
  By compactness of the $G$-manifolds $M_i$, the critical points $U$ of a critical value $c$ form a finite $G$-set.
\end{construction}
The gap between $M_{i-1}$ and $M_i$ can be described using equivariant handle attachment.

\paragraph{$G$-handles and the attaching lemma}
We will use a slight variant of Wasserman's treatment of equivariant handlebody decomposition.
Namely, we work with open $G$-manifolds and $G$-handles, which are the interiors of Wasserman's $G$-manifolds and $G$-handles.
\begin{mydef}\label{def:open_G_handles}
  Let $U$ be a finite $G$-set and $E \to U$ be a $G$-vector bundle with a $G$-invariant metric.
  Define $D(E)$ to be the $G$-manifold
  \[
    D(E) = \set{v \in E \mid |v| <1}.
  \]
  An \myemph{(open) $G$-handle bundle} is a $G$-manifold of the form $D(E) \times_U D(E')$.
\end{mydef}
\begin{rem}
  Note that as a manifold $D(E) \times_U D(E')$ is simply the disjoint union of handles
  \[
    D(E) \times_U D(E') = \bigsqcup_{x\in U} D(E_x) \times D(E'_x).
  \]
\end{rem}
\begin{rem}
  Wasserman defines a $G$-handle bundle of type $(E,E')$ as a $G$-manifold with corners
  \[
    E(1) \times_U E'(1) = \set{ (x,y) \in E \oplus E' \mid |x|\leq 1, |y|\leq1}.
  \]
   The boundary of such a $G$-handle bundle is 
   \( \big(\partial E(1) \times_U E'(1)\big) \cup \big(E(1) \times_U \partial E'(1)\big) \),
   and the $G$-handle bundle is attached along $\partial E(1) \times_U E'(1)$.
  Our definition of an open $G$-handle agrees with the \myemph{interior} of Wasserman's handle-bundle.
\end{rem}
\begin{ex}[$G$-Handle-bundle of critical orbits]\label{G_handle_crit_orb}
  Let $M$ be a $G$-manifold equipped with a $G$-invariant Morse function $f \colon M \to \R$.
  Assume $f$ has a critical value $c\in \R$ with a isolated critical points.
Write $U$ for the $G$-set of critical points of value $c$.
  For each $x\in U$ the tangent space at $x$ decomposes into negative definite and positive definite subspaces, $T_x M = E_x \oplus E'_x$. 
  These decompositions are preserved by the action of $G$ because $f$ has a $G$-invariant Hessian, so we get a decomposition of $G$-vector bundles over $U$
  \[
    TM |_U = E \oplus E',
  \] 
  to which we associated an open $G$-handle-bundle $D(E) \times_U D(E')$.
  The exponential map identifies $D(E) \times_U D(E')$ (suitably normalized) to
  an open $G$-sub-manifold of $M$, see \cref{fig:G_Morse}. 
  Let $\delta > 0$ be a small number. Write
  \[
    M_0 = \set{ x\in D(E) \times_U D(E') \mid f(x)=c-\delta}.
  \]
  Then flowing along the (suitably normalized) gradient of $f$ defines equivariant diffeomorphisms 
\begin{align*}
  \set{ x\in D(E) \times_U D(E') \mid c-2\delta < f(x) < c-\delta}
  & \= M_0 \times (c-2\delta,c-\delta)  \= M_{0} \times \R \\
  \set{ x\in D(E) \times_U D(E') \mid f(x) > c-2\delta}
  & \= D(E) \times_U D(E') 
\end{align*}
Via the diffeomorphisms, $M_0 \times \R $ is an open $G$-sub-manifold of the
$G$-handle-bundle $D(E) \times_U D(E')$.
\end{ex}

\paragraph{Equivariant handle-body decomposition}
Let $f \colon M \to \R$ be a $G$-invariant Morse function with regular values $a<b$ and a unique critical value in the interval $[a,b]$.
Wasserman uses a $G$-invariant version of the attaching lemma to show that $\overline{M_b}$ is $G$-diffeomorphic to $\overline{M_a}$ with a $G$-handle bundle attached.
Passing to the interior of these $G$-manifolds, we get the following variant.
\begin{figure}[h]
  \centering
  \includegraphics[width=0.75\linewidth]{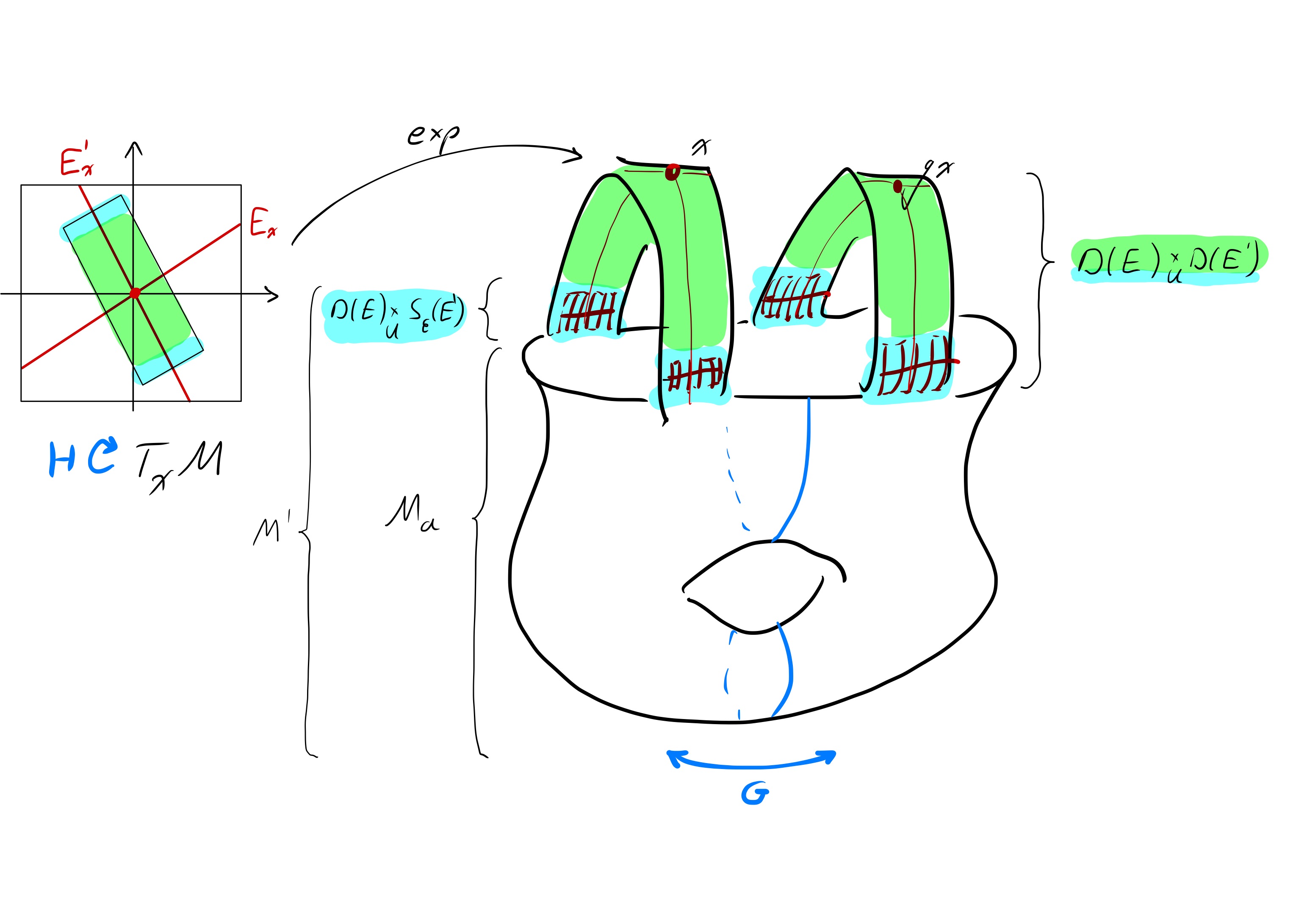}
  \caption{Attaching a $G$-handle $D(E) \times_U D(E')$ of a critical point $x \in U$ with stabilizer $H$ and orbit $U=\set{x,gx}$}
  \label{fig:G_Morse}
\end{figure}
\begin{prop} \label{G_Morse_handle}
  Let $f \colon M \to \R$ be a $G$-invariant Morse function with regular values $a<b$.
  Assume $f$ has a single critical value $a<c<b$ with finite isolated critical points $U$ of value $c$, and let $TM|_U = E \oplus E'$ be the decomposition of the tangent bundle into positive definite and negative definite subspace.
 
  Then $M_b= \set{x \in M \mid f(x)<b}$ is equivariantly diffeomorphic to a $G$-collar decomposition 
  \[
    M_a \bigcup_{M_{0} \times \R} D(E) \times_U D(E'),
  \]
  where $M_a = \set{x \in M \mid f(x)<a}$, $D(E) \times_U D(E')$ is an open $G$-handle bundle (see
  \cref{def:open_G_handles}) and $M_{0} \times \R \subset D(E) \times_U D(E')$ is an open
  $G$-sub-manifold (see \cref{G_handle_crit_orb}).
\end{prop}
\begin{proof}
  The proof follows from \cite[thm. 4.6]{Wasserman} after choosing a complete $G$-invariant Riemannian metric on $M$.
Since $f$ has no critical values in $[a,c)$, without loss of
    generality we assume that $a$ is sufficiently close to $c$, or $a =
    c-\delta$. Wasserman shows that $\overline{M_b}$  is equivariantly
  diffeomorphic to the disjoint union of $\overline{M_a}$ and
  $E(1) \times_U E'(1)$ along  $\partial E(1) \times_U E'(1)$. 
  Passing to the interior, we see that $M_b$ is equivariantly diffeomorphic to the union $M' \bigcup_{M_0 \times \R} (D(E) \times_U D(E'))$, where
  \[
      M' = \set{ x\in M \mid (f(x) < c- 2\delta) \text{ or } ( x \in D(E) \times_U
        D(E') \text{ and } f(x) < c-\delta ). }
  \]
  See \cref{fig:G_Morse}.
  Observe that $M'$ is equivariantly diffeomorphic to $M_a$. This completes the proof.
\end{proof}

\subsection{Axiomatic characterization}
In this section, we prove the following result, which is for later use  of genuine factorization homology
\begin{thm} \label{cor:GFHaxiom}  
  Let $F \colon \ul\mfld^{G,V-fr} \to \ul\C$ be a $G$-symmetric monoidal functor.
  In order to show that $F$ is a genuine equivariant factorization homology theory, it is enough to check that 
  \begin{enumerate}
    \item $F$ respects sequential unions of the form $M=\cup_i M_i$ for a
      $V$-framed $G$-manifold $M$ and 
      $M_i = \set{ x\in M | f(x) < a_i}$,
      where $f \colon M \to \R$ is a $G$-invariant, bounded below proper Morse
      function with isolated critical values, and $a_0<a_1< \cdots$ is a sequence of regular values.
    \item $F$ satisfies $\otimes$-excision for $G$-collar decompositions of the form 
      \[ 
        M = M' \cup_{M_0 \times \R} M''
      \]
      where $M$ is the interior of a compact $V$-framed $G$-manifold with boundary (possibly with empty boundary), 
      $M'$ and $M''$ are open $G$-submanifolds and $M_0$ is a closed $G$-submanifold of codimension 1.
  \end{enumerate}
  Indeed, the open $G$-submanifolds $M', M'' \subseteq M$ can be taken to be
\begin{equation*}
    M'= M_{i-1}, \quad M''= D(E) \times_U D(E'), \quad M_0 \times \R \text{ as
      in \cref{G_handle_crit_orb}}.
\end{equation*}
\end{thm}

To give the proof, we review the proof of the following theorem.
\begin{thm}
  (\cite[thm. 6.0.2]{GFH})
  \label{thm:GFHaxiom}
  Let $\ul\C$ be a presentably $G$-symmetric monoidal $G$-$\infty$-category.
  The functor
  \begin{align*}
    & \int \colon \EE_V-\Alg(\ul{\C}) \to \Fun^{G-\otimes}(\ul\mfld_G^{V-fr}, \ul\C), \\
    & \ul{A} \mapsto \left( \int \ul{A} \colon \ul\mfld^{G,V-fr} \to \ul\C \right)
  \end{align*}
  is fully faithful with essential image $\mathcal{H} (\ul\mfld^{G,V-fr},\ul\C)$ and homotopy inverse given by the restriction along $\ul\disk^{G,V-fr} \to \ul\mfld^{G,V-fr}$.
\end{thm}
\begin{proof}[Idea of proof]
  First, observe that $\int$ is fully faithful.
  Indeed, it is an operadic $G$-left Kan extension along the fully faithful inclusion $\ul\disk^{G,V-fr} \to \ul\mfld^{G,V-fr}$.

  The hard part of the proof is showing that 
  $\mathcal{H} (\ul\mfld^{G,V-fr},\ul\C)$ is a $G$-factorization homology theory.
  This is done in \cite[sec. 5.2, 5.3]{GFH} in the framework of parametrized $\infty$-category theory.

  It remains to show that the counit of the adjunction
  \[
    \int \colon \EE_V-\Alg(\ul\C) \adj \mathcal{H}(\ul\mfld^{G,V-fr}, \ul\C) 
  \]
  is a natural equivalence.
  Let $F \in \mathcal{H}(\ul\mfld^{G,V-fr}, \ul\C)$ be a genuine equivariant factorization homology theory, and $\ul{A} \colon \ul\disk^{G,V-fr} \to \ul \C$ its restriction to an $\EE_V$-algebra.
  We have to show that the counit map
  \[
    \eta \colon \int_M \ul{A} \to F(M)
  \]
  is an equivalence for every $M \in \ul\mfld^{G,V-fr}_{[G/H]} \simeq \mfld_H^{V-fr}$.
  Since the subgroup $H\leq G$ will remained fixed we can assume $M \in \mfld_G^{V-fr}$ without loss of generality.
  Note that at this point we know that $\eta$ is a natural transformation of genuine equivariant factorization homology theories.
  The idea is to use equivariant Morse theory to get a geometric decomposition of $M$, and then inductively prove that the counit map is an equivalence.

  Let $M = \cup_i M_i$ be the equivariant Morse decomposition of \cref{G_Morse_decomp}.
  Naturality of $\eta$ and the fact that $F$ and $\int \ul{A}$ both respect sequential unions imply that it is enough to show that 
  \[
    \int_{M_i} \ul{A} \to F(M_i)
  \]
  is a weak equivalence, where $M_i$ is the interior of the compact $G$-manifold $\overline{M_i}$.
  Next, by \cref{G_Morse_handle}, $M_i$ is equivariantly diffeomorphic to a $G$-collar decomposition 
  \[
    M_{i-1} \bigcup_{ M_0 \times \R } D(E) \times_U D(E').
  \]
  Since $\eta$ is a natural transformation between functors that satisfy $\otimes$-excision we get a commutative diagram
  \begin{equation}
    \begin{tikzcd}
      \int_{M_{i-1}} \ul{A} 
      \otimes_{ \int_{M_0 \times \R} \ul{A} }
      \int_{D(E) \times_U D(E')} \ul{A} \ar[d,"\simeq"] \ar[r] &
      F(M_{i-1})
      \otimes_{F(M_0 \times \R)}
      F(D(E) \times_U D(E') ) \ar[d,"\simeq"] \\
      \int_{M_i} \ul{A} \ar[r,"\eta_{M_i}"] & F(M_i).
    \end{tikzcd}
  \end{equation}
  Consider the maps
  \begin{align*}
    \eta_{M_{i-1}} \colon & \int_{M_{i-1}} \ul{A} \to F(M_{i-1}), \quad 
    \eta_{D(E) \times_U D(E')} \colon \int_{D(E) \times_U D(E')} \ul{A} \to F(D(E) \times_U D(E')) \\
                          & \eta_{M_0 \times \R} \colon \int_{M_0 \times \R} \ul{A} \to F(M_0 \times \R )
  \end{align*}
  inducing the top morphism in the diagram above.
  The first map is an equivalence by induction and the second is an equivalence since the $G$-handle $D(E) \times_U D(E') \subset M$ is a $V$-framed $G$-disk and $\ul{A} = F|_{\ul\disk^{G,V-fr}}$ agrees on $V$-framed $G$-disks by construction.

  It remains to show that 
  \begin{align} \label{eq:eta_G_collar}
    \eta_{M_0 \times \R} \colon \int_{M_0 \times \R} \ul{A} \to F(M_0 \times \R)
  \end{align}
  is an equivalence.
The insight here is that $G$ acts trivially on $\R$ in the $G$-collar $M_0 \times \R$.\footnote{
    Note that $M_0 \times \R$ has a $V$-framing induced by the inclusion into $M$.
    However we do not (and cannot) assume that $ M_0 \times \R$ is a product of framed $G$-manifolds.
  }

  We can therefore finish the proof by wrapping the above argument in yet another induction.
  Consider $V$-framed $G$-manifolds of the form $N = N' \times\R^k$, with $G$ acting trivially on $\R^k$.
  Go back to the beginning of the argument, and prove that 
  \[
    \eta \colon \int_{N} \ul{A} \to F(N) 
  \]
  is an equivalence by induction on $\dim N'$.
  When $\dim N'=0$ the $G$-manifold $ N = D(E) \simeq E$ is a $G$-vector bundle over a finite $G$-set $U$, hence a $V$-framed $G$-disk, so $\eta_N$ is an equivalence.
  For $\dim N'>0$, we use $G$-Morse theory use a $G$-invariant Morse function on $N'$ to decompose $N$ as above, and use the new inductive hypothesis to ensure that~\eqref{eq:eta_G_collar} is an equivalence.
\end{proof}

\begin{proof}(Proof of \cref{cor:GFHaxiom})
 Note that $F$ is a genuine $G$-equivariant factorization homology theory in the sense of \cref{defn:genu-equiv-fh} if for every $H \leq G$ the symmetric monoidal functor
 \[
   F_{[G/H]} \colon \mfld_H^{V-fr} \simeq \ul\mfld^{G,V-fr}_{[G/H]} \to \ul\C_{[G/H]}
  \]
  satisfies $\otimes$-excision and respects sequential unions.
Examining the role of $\otimes$-excision and sequential unions in the above
proof, we see that the axioms are only invoked for decompositions of the form
that arise in equivariant Morse theory, which is exactly the content of \cref{cor:GFHaxiom}.
\end{proof}

\section{Equivariant nonabelian Poincar\'e duality}
\label{sec:NAPD}
\newcommand{\mapcX}[1]{\ensuremath{\mathrm{Map}_{*}((#1)^+,X)}}
In this section we prove the equivariant version of nonabelian Poincar\'e duality
regarding equivariant factorization homology (\cref{thm:NPD}).
Our approach is motivated by  \cite[sec. 4]{AF}. 

\begin{thm}[Equivariant nonabelian Poincar\'e duality] \label{thm:NPD}
   For a $V$-framed $G$-manifold  $M$ and  $X \in \mathbf{Top}^G_{*}$  such that
$\pi_k(X^H) = 0$ for all subgroups $H \subset G$ and $k < \mathrm{dim}(V^H)$, there is a natural equivalence
of $G$-spaces
$$\int_M \Omega^{V}X \simeq \mathrm{Map}_{*}(M^+, X).$$
Here, $M^+$ is the one-point-compactification of $M$.
\end{thm}
 We first explain the functors in the statement. In
 \cref{cons:Mapc_X_as_top_functor} and \cref{Mapc_X_as_top_functor}, we construct the $\infty$-functor
  $\Map_*((-)^+, X)$ as the underlying functor of the $G$-symmetric monoidal $\infty$-functor
  $$\ul{\mathrm{Map}}_*((-)^+,X) \in \mathrm{Fun}^{\otimes}_{G}(\ulmfld^{G,\sqcup},\ulTopG_{*})$$
  at the fiber over $G/G$.
  Precomposing with the forgetful $G$-functor $ \ulmfld^{G,V-fr,\sqcup} \to \GmfldD $,
  we get
\begin{equation}\label{eq:map-to-X-framed}
\ul{\mathrm{Map}}_*((-)^+,X) \in
 \mathrm{Fun}^{\otimes}_{G}(\ulmfld^{G,V-fr,\sqcup},\ulTopG_{*}),
\end{equation}
whose underlying functor over $G/G$ we still denote by $\Map_{*}((-)^{+},X)$.
The functor $\Omega^VX$ is the restriction of $\ul{\mathrm{Map}}_*((-)^+,X)$ to
$V$-framed $V$-disks and thus an $\EE_V$-algebra in $\ulTopG$. (For more
details, see \cref{EValg_of_V_loop_space}.)

\subsection{Equivariant Atiyah duality and equivariant Poincar\'e duality}
Assuming \cref{thm:NPD}, we deduce a version of the equivariant Atiyah duality theorem (\cref{cor:AD}) and the equivariant Poincar\'e duality theorem (\cref{cor:PD}) for $V$-framed $G$-manifolds.
Atiyah duality for $G$-manifolds has also previously been studied in \cite[III.5]{lewis1986equivariant}.

\begin{prop}\label{cor:AD}
  Suppose $M$ is a $V$-framed $G$-manifold and $E$ is $G$-spectrum such that
  $\pi_k ^H(\Sigma^V E) = 0$ for $k < \mathrm{dim}(V^H)$.
  Then there is a $G$-equivalence:
\begin{equation*}
\Omega^\infty (\Sigma^\infty_+ M \otimes E) \simeq \mathrm{Map}_{*}(M^{+}, \Omega^{\infty}\Sigma^VE).
\end{equation*}
In particular, taking $E = \Sigma^{\infty} S^0$ to be the $G$-sphere spectrum, we recover Atiyah duality for $M$:
\begin{equation*}
\Omega^\infty (\Sigma^\infty_+ M) \simeq \mathrm{Map}_{*}(M^{+}, \Omega^{\infty}\Sigma^\infty S^V).
\end{equation*}
\end{prop}

\begin{proof}
  We consider $\Omega^\infty E$, the $G$- infinite loop space of $E$. 
Denote $X = \Omega^{\infty} \Sigma^VE$.
Then we have $\Omega^\infty E \simeq \Omega^V X$ as $\EE_V$-algebras.
Moreover,  $\pi_k ^H (X) \cong \pi_k ^H( \Sigma^V E)$
for $k \geq 0$. Therefore by assumption, $X$ satisfies the connectivity hypotheses in \cref{thm:NPD},
and we obtain
\begin{equation}
\label{eq:pd}
\int_M \Omega^\infty E \simeq \mathrm{Map}_{*}(M^{+}, X).
\end{equation}

We claim that there is a $G$-equivalence:
\begin{equation}
\label{eq:compareFH}
\int_M \Omega^\infty E \simeq \Omega^\infty (\Sigma^\infty_+ M \otimes E).
\end{equation}
First, $\Omega^\infty ( \Sigma^\infty_+(-) \otimes E)$ is a $G$-factorization
homology on $V$-framed $G$-manifolds, as we show later in \cref{lem:HB}. Second,
the factorization homology theories $\int_{-} \Omega^\infty E$ and $\Omega^\infty( \Sigma^\infty_+ (-) \otimes
E) $ have the same coefficients: their coefficients are the $\EE_V$-algebras
$\Omega^\infty E$ and
$\Omega^\infty( \Sigma^{\infty}_+V \otimes E) $. The map which contracts $V$ to a point,
$\Omega^\infty( \Sigma^{\infty}_+V \otimes E)  \to \Omega^\infty E$, is a map of $\EE_V$
-algebras, as it is a map of $G$-infinite loop spaces. It is an equivalence of $G$-spaces, and
therefore an equivalence of $\EE_V$-algebras.
So the two factorization homology theories in~\eqref{eq:compareFH} agree and we obtain the equivalence.

The claim follows from combining \eqref{eq:pd} and
\eqref{eq:compareFH}.
To apply it to
$E=\Sigma^{\infty}S^0$, we check when $k < \mathrm{dim}(V^H)$, the tom Dieck
splitting and $k < \mathrm{dim}(V^K)$ for $K \subset H$ gives
\begin{align*}
  \pi_k^H(\Sigma^\infty S^V)
  \cong & \oplus_{[K \subset H]} \pi_k(\Sigma^{\infty} EW_{H}K_+ \wedge_{W_HK}
          S^{V^K}) = 0.
          \qedhere
\end{align*}
\end{proof}

\begin{prop}
  \label{cor:PD}
  Suppose $M$ is a $V$-framed $G$-manifold and $\underline{B}$ is a Mackey functor.
Then
$$\mathrm{H}_{\star} (M; \underline{B}) \cong \widetilde{\mathrm{H}}^{V-\star} (M^+;
\underline{B}).$$
In particular, if $M$ is closed, then
$\mathrm{H}_{\star} (M; \underline{B}) \cong \mathrm{H}^{V-\star} (M;
\underline{B})$.
\end{prop}
\begin{proof}
  We can give $S^V$ an $H$-CW decomposition with the lowest cells
  other than the base point in dimension $\mathrm{dim}(V^H)$. So 
  we have $\pi_k ^H( \Sigma^V \mathrm{H}\underline{B}) \cong \tilde{H}^H _k (S^V ; \underline{B}) = 0$
  when $k < \mathrm{dim}(V^H)$. Therefore we can take $E$  in \cref{cor:AD} to be the Eilenberg--MacLane spectrum $\mathrm{H}\underline{B}$ and get
\begin{equation*}
\Omega^\infty (\Sigma^\infty_+ M \otimes \mathrm{H}\ul{B}) \simeq \mathrm{Map}_{*}(M^{+}, K(\underline{B},V)).
\end{equation*}
The desired Poincar\'e duality follows from taking homotopy groups on both sides and identifying:
\begin{align*}
  \pi_{\star}\Omega^\infty (\Sigma^\infty_+ M \otimes \mathrm{H}\ul{B})
  & \cong \mathrm{H}_{\star}(M, \underline{B}); \\
  \pi_{\star}(\mathrm{Map}_{*}(M^{+}, K(\underline{B}, V)))
  & \cong \widetilde{\mathrm{H}}^{V-\star}(M^+; \underline{B}). \qedhere
\end{align*}
\end{proof}

\begin{lem}\label{lem:HB}
  Let $E$ be a $G$-spectrum.
  Then
  $\Omega^\infty( \Sigma^\infty_+ (-) \otimes E)$ is a $G$-factorization homology theory on $V$-framed $G$-manifolds.
\end{lem}
\begin{proof}
One can express $\Omega^\infty( \Sigma^\infty_+ (-) \otimes E)$ as the composition of $G$-functors
\[
  \Omega^\infty( \Sigma^\infty_+ (-) \otimes E) \colon \ulmfld^{G,V-fr}  \to \ulmfld^{G} \xto{fgt} \ulTopG \xto{\Sigma^\infty_+} \ul\Sp^{G} \xto{- \otimes E} \ul\Sp^G \xto{\Omega^\infty} \ulTopG.
\]
Each $G$-functor in the composition extends to a $G$-symmetric monoidal functor:
\begin{enumerate}
  \item The $G$-functor \( \ulmfld^{G,V-fr} \to \ulmfld^G \), forgetting the $V$-framing, is $G$-symmetric monoidal by construction. 
  \item The functor \( fgt \colon \ulmfld^G \to \ulTopG \) is $G$-symmetric monoidal, as it can be defined by the following construction.
   As a functor of topological categories, the forgetful functor  \( \mfld_n \to \Top \) is symmetric monoidal (takes disjoint unions to coproducts).
  Construct the forgetful functor \( fgt \colon \ulmfld^G \to \ulTopG \) by applying the genuine operadic nerve construction (see also \cref{app:top_G_objs}).
  \item The $G$-functor \( \Sigma^\infty_+ \colon \ulTopG \to \ul\Sp^G \) is a $G$-left adjoint, hence strongly preserves $G$-colimits.
  In particular, it extends to a $G$-symmetric monoidal functor with respect to the $G$-coCartesian monoidal structure on both categories. 
  \item Similarly, \( \ul\Sp^{G} \xto{- \otimes E} \ul\Sp^G \) strongly preserves $G$-colimits, and therefore extends to a $G$-symmetric monoidal functor.
  \item The $G$-functor \( \ul\Sp^G \xto{\Omega^\infty} \ulTopG \) is a $G$-right adjoint, and therefore extends to a $G$-symmetric monoidal functor with respect to the \myemph{$G$-Cartesian} monoidal structures. 
    \item Finally, since $\ul\Sp^G$ is a $G$-semi-additive $G$-$\infty$-category, the $G$-Cartesian and $G$-coCartesian monoidal structure are canonically equivalent.
\end{enumerate}
In fact, this decomposition also shows that $\Omega^\infty( \Sigma^\infty_+ (-)
\otimes E):\ulTopG \to \ulTopG$ preserves sifted colimits fiberwise.

We now verify the $\otimes$-excision axiom.
Let $M= M' \cup_{M_0 \times \R} M''$ be a $G$-collar decomposition of $V$-framed $G$-manifolds.
After applying the forgetful functor 
\[
  \ulmfld^{G,V-fr}  \to \ulmfld^{G} \xto{fgt} \ulTopG
\]
the $G$-space $M = M' \coprod_{M_0 \times \R} M'' $ is equivalent to the geometric realization 
\[
  \big\lvert 
  \begin{tikzcd}
    M' \coprod M''  &
    M' \coprod M_0 \times \R \coprod M'' \ar[l, shift left] \ar[l, shift right] &
    M' \coprod M_0 \times \R \coprod M_0 \times \R \coprod M'' \ar[l]  \ar[l, shift left=2]  \ar[l, shift right=2]  &
    \cdots \ar[l, shift left=1]  \ar[l, shift left=3]  \ar[l, shift right=1]  \ar[l, shift right=3] 
  \end{tikzcd} 
  \big\rvert.
\]
Since $\Omega^\infty( \Sigma^\infty_+ (-) \otimes E)$ preserves geometric realizations, we have an equivalence
\[
  \Omega^\infty( \Sigma^\infty_+ (M) \otimes E) \simeq 
  \mathbf{B}\big(\Omega^\infty( \Sigma^\infty_+ (M') \otimes E), 
    \Omega^\infty( \Sigma^\infty_+ (M_0 \times \R) \otimes E),
    \Omega^\infty( \Sigma^\infty_+ (M'') \otimes E)\big) ,
\]
hence $  \Omega^\infty( \Sigma^\infty_+ (-) \otimes E)$ satisfies $\otimes$-excision.
A similar approach verifies $\Omega^\infty( \Sigma^\infty_+ (-) \otimes E) $ respects $G$-sequential unions, hence it is a $G$-factorization homology theory. 
\end{proof}

\subsection{Proof of \cref{thm:NPD}}
We prepare for the proof of  by showing that
the functor $\Map_{*}((-)^+,X)$ 
respects sequential unions and satisfies $\otimes$-excision $H$-equivariantly  for any subgroup
$H \subset G$ in the sense of \cref{cor:GFHaxiom}. Without loss of generality we may work with $H=G$.
We work in topological categories and 
$\Map_{*}((-)^+,X)$ is a topological functor between topological categories.

We start with proving the $G$-sequential union property.
We find it convenient to replace the space $\Map_*(M^+,X)$ with the space of compactly supported maps.
\begin{mydef}
  For a $G$-manifold $M$ and a based $G$-space $X$, let
  \[
    \Map_c(M,X) = \{f \in \Map(M,X)\, | \, \mathrm{supp}(f) \text{ is compact}\}
  \]
  be the space of compactly supported maps. 
  Here, the support of a map $f$ is the closure of the preimage of the complement of the base point.
\end{mydef}

\begin{prop}
  \label{Map_c_vs_Map_plus}
 There is a $G$-equivalence $\Map_c(M,X) \iso
  \Map_*(M^+,X)$, which is natural in open embeddings in the variable $M$.\footnote{See the first five bullet points of section 2.5 of https://etale.site/teaching/s16-lcfa/locally-constant.pdf for the non-equivariant version of this.} \qed
\end{prop}

\medskip 
 Let $M_1 \subset M_2$ be an open inclusion of $G$-manifolds. Extending $f \in
 \Map_c(M_1,X)$ by mapping to
  the base point on $M_2 - M_1$
  gives a $G$-map $\Map_c(M_1,X) \to \Map_c(M_2,X) $.  
The following lemma is a geometric observation.
\begin{lem}
\label{prop:cofibration}
  Let $f \colon M \to \R$ be a $G$-invariant Morse function with isolated critical points, and let $t < s$ be two regular
  values. Denote $M_0  = f^{-1}(-\infty,t)$, $M_1  = f^{-1}(-\infty,s)$, and let $X$ be
  any based $G$-space. Then $\Map_{c}(M_0, X) \to \Map_{c}(M_1,X)$  is a $G$-cofibration.
  \end{lem}
\begin{proof}
  We show that the inclusion  $\Map_{c}(M_0, X) \to \Map_{c}(M_1,X)$  is an $G$-NDR (neighborhood deformation retract) pair, thus a $G$-cofibration. 
The idea is that for a map in $\Map_{c}(M_1,X)$, we can use the gradient flow to
``drag its support towards $M_0$''.

  More precisely, let us assume $t=0$ without loss of generality and that $\epsilon>0$ is small enough such that $(-\epsilon, \epsilon)$ are all regular values for $f$. 
  We prepare several functions to construct the $G$-NDR data $(h,u)$.

  Let $u' \colon \Map_c(M_1,X) \to [0,s] $ be
  $$u'(-) = \sup\{f(\mathrm{supp}(-)),0\}$$
  and
  $u \colon \Map_c(M_1,X) \to [0,1] $ be
  
  \[
    u(-) = \min\{u'(-)/\epsilon,1\}.
  \]
  The map $u$ is equivariant because $f$ is, and it is easy to check that
  
\begin{equation*}
u^{-1}(0) = \Map_c(M_0,X) \text{ and } u^{-1}(t) = \Map_c(f^{-1}(-\infty,\epsilon t),X) \text{ for } t \in (0,1).
\end{equation*}
     
  Since $(-\epsilon,\epsilon)$ are all regular values of the Morse function, we can construct
  an equivariant flow $F \in \Map(I ,\mathrm{Diff}(M_1 , M_1))$ such that 
  \begin{equation*}
    F(0) = \mathrm{id}_{M_1} \text{ and } F(t)(M_0) = f^{-1}(-\infty, t\epsilon).
  \end{equation*}
  Now take $h \colon \Map_c(M_1,X) \times I \to \Map_c(M_1,X)$ to be 
  \begin{equation*}
    h(-,t) = - \circ F(u(-)t).
  \end{equation*}
  It is easy to verify that $(h,u)$ represents $\Map_{c}(M_0, X) \to \Map_{c}(M_1,X)$ as a $G$-NDR pair.
\end{proof}

\begin{prop}(Sequential union)
  \label{lem:NAPD-G-seq-union}
Let $f \colon M \to \R$ be a $G$-invariant Morse function with isolated critical points.  Let
 $s_1< s_2 < \cdots $ be regular values such that for each $s_i
\in \R$, $f^{-1}(-\infty,s_i]$ is compact, and write $M_i =f^{-1}(-\infty,s_i)$.
Then there is a $G$- equivalence
\begin{equation*}
\mathrm{Map}_*(M^+,X) \simeq \mathrm{hocolim}_{i}\mathrm{Map}_*((M_i)^{+},X).
\end{equation*}
\end{prop}
\begin{proof}
As $\overline{M_i} \subset f^{-1}(-\infty,s_i] \subset M_{i+1}$, we have
\begin{equation*}
 \mathrm{Map}_c(M,X) \simeq \mathrm{colim}_{i}\mathrm{Map}_c(M_i,X).
\end{equation*}
Since $\mathrm{Map}_{c}(M_i, X) \to
\mathrm{Map}_{c}(M_{i+1},X)$ is a $G$-cofibration (\cref{prop:cofibration}),
\begin{equation*}
 \mathrm{colim}_{i}\mathrm{Map}_c(M_i,X) \simeq  \mathrm{hocolim}_i \mathrm{Map}_c(M_i,X).
\end{equation*}
Via the functorial identification $\mathrm{Map}_c(M,X) \simeq \mathrm{Map}_*(M^+,X)$, we have
\begin{equation*}
\mathrm{Map}_*(M^+,X) \simeq \mathrm{hocolim}_{i}\mathrm{Map}_*((M_i)^{+},X). \qedhere
\end{equation*}
\end{proof}

Next, we prove the $G$-$\otimes$-excision property.
We begin by fixing a $G$-collar decomposition.
\medskip\\
\begin{minipage}{0.75\textwidth}
  \begin{notation}
  \label{nota:collar}
  In the rest of this section, $X$ is a pointed $G$-space and $M$ is an
  $n$-dimensional $G$-manifold, either closed or the interior of a
  compact manifold with boundary.
  We fix a $G$-collar decomposition $M = M' \bigcup_{M_0 \times \mathbb{R}}
  M''$, where $M'$, $M''$ are open $G$-submanifolds and $M_0$ is a closed
  $G$-submanifold of codimension 1.
 We abuse notation to write $M-M'$ for $M-(M'-M_0 \times [0,+\infty))$ and
 $M-M''$ for $M-(M''-M_0 \times (-\infty,0])$. (See \cref{fig:pic-collar}.)
 They are diffeomorphic as manifolds with boundary, but we gain better control of the boundary and collar gluing.
 Both $M-M'$ and $M-M''$ have $M_0 \times \{0\} \cong M_0$ in their  boundaries.
\end{notation}
\end{minipage}
\begin{minipage}{0.25\textwidth}
   \centering
   \includegraphics[width=\textwidth]{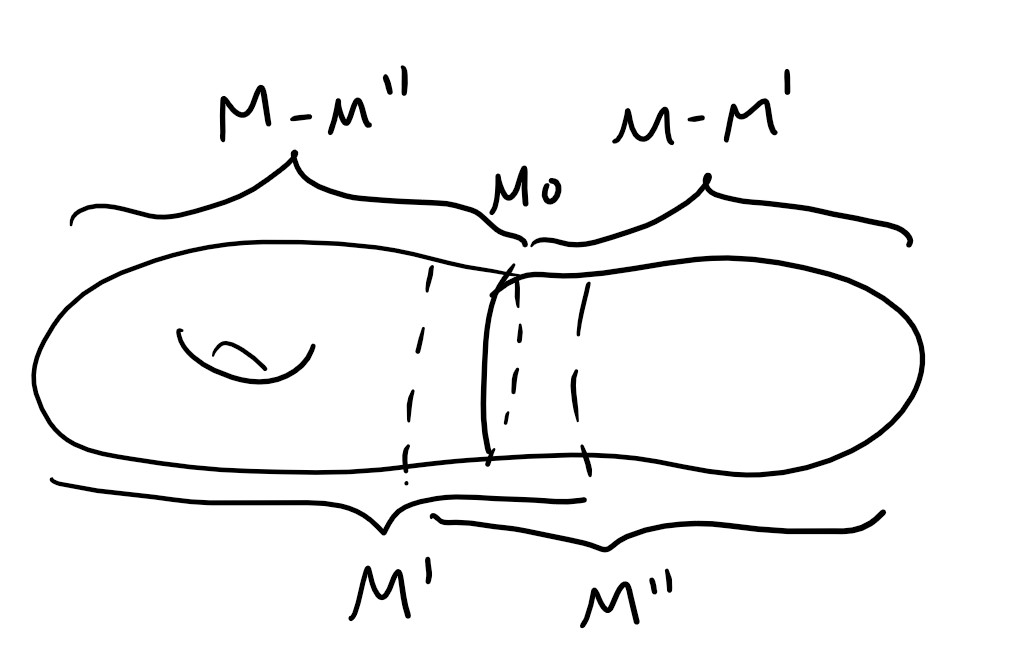} 
\captionof{figure}{Notation}
\label{fig:pic-collar}
\end{minipage}

\begin{lem}\label{prop-G-fibration}
  The restriction map $\mathrm{Map}_*((M-M')^+, X) \to \mathrm{Map}_*((M_0)^{+},X)$ is a $G$-fibration with fiber $\mathrm{Map}_*((M'')^{+},X)$. 
  The corresponding statement in which $M'$ and $M''$ are switched also holds.
\end{lem}

\begin{proof}
  Note that the embedding $M_0 \hookrightarrow M - M'$ is proper, and therefore indeed induces a map $M_0 ^+ \to (M-M')^+$. For the first part, it suffices to show that $M_0^+ \to (M-M')^+$ is a $G$-cofibration in the Hurewicz sense, as mapping out of it would give a Hurewicz fibration, which is in particular a Serre fibration.

It is known that closed embeddings of compact $G$-manifolds are $G$-cofibrations. To
prove it, suppose we have a closed embedding $A \hookrightarrow B$. By
\cite{mostow1957equivariant}, a compact $G$-manifold $B$ can be embedded in some
orthogonal $G$-representation $W$.
Being a closed submanifold, $A$ is equivariantly embedded as a retract of an
open subspace of $W$ \cite[Theorem  1.4]{Illman2000}. Both $A$ and $B$ are then
$G$-ENRs (Euclidean neighborhood retract) in the sense of
\cite{lewis1986equivariant}, and by \cite[III.4]{lewis1986equivariant}, an inclusion of $G$-ENRs is a $G$-cofibration.

If $M$ is compact, so is $M_0$ and $M-M'$, as they are closed submanifolds of
$M$. Then $M_0 \to M-M'$ is a $G$-cofibration and so is  $M_0^+ \to
(M-M')^+$. 

  \begin{minipage}{0.75\textwidth}
    Now suppose that $M$ is the interior of a compact manifold $\overline{M}$
    with boundary $\partial \overline{M}$. (The previous case can also be
    treated as $\partial \overline{M} = \varnothing$.)
    Since both $M_0 $ and $M-M'$ are submanifolds of $M$ which are closed,
    $\partial \overline{M_0}$ and $D=\partial (\overline{M-M'}) \cap
    \partial\overline{M}$ are both submanifolds of $\partial \overline{M}$ which
    are closed. (See \cref{fig:pic1} for illustration.) We can identify 
\begin{align*}
  M_0^+& = \overline{M_0}/\partial\overline{M_0} \\
  (M-M')^+ & = (\overline{M-M'}) /D
\end{align*}
\end{minipage}
\begin{minipage}{0.22\textwidth}
    \centering
   \includegraphics[width=\textwidth]{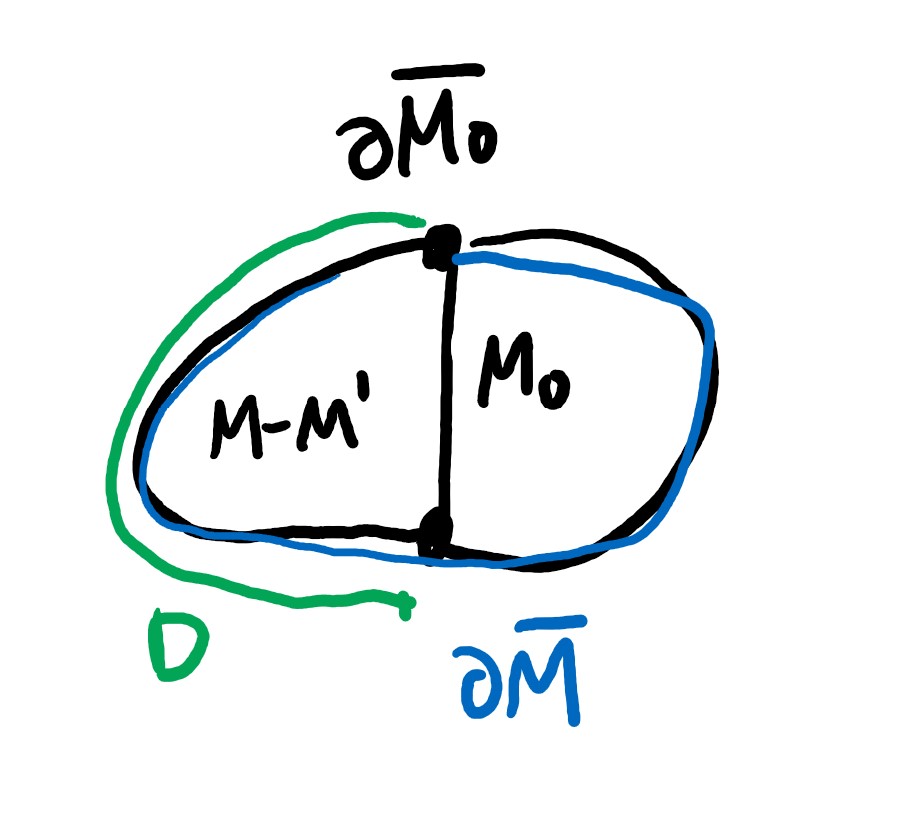} 
\captionof{figure}{Illustration}
\label{fig:pic1}
\end{minipage}

The following is a pushout square all maps of which are closed embedding of compact
$G$-manifolds and thus $G$-cofibrations:
\begin{equation*}
  \begin{tikzcd}
   \partial \overline{M_{0}} \ar[r] \ar[d] & \overline{M_0}
    \ar[d] & \\
    D \ar[r] & \partial (\overline{M-M'}) \ar[r] & \overline{M-M'}
  \end{tikzcd}
\end{equation*}
Therefore, $M_0^+ = \overline{M_0}/\partial \overline{M_0} \cong
\partial(\overline{M-M'})/D  \longrightarrow \overline{M-M'}/D = (M-M')^{+}$ is
a $G$-cofibration.  Furthermore, the cofiber of this map is
$\overline{M-M'}/\partial(\overline{M-M'}) \cong \overline{M''}/\partial
\overline{M''} = (M'')^+$. So the fiber in the claim can be identified with $\mathrm{Map}_{*}((M'')^+, X)$.
\end{proof}

\begin{lem}\label{prop-maps-pb}
 The diagram

$$\xymatrix{
\mathrm{Map}_*(M^{+},X) \ar[r] \ar[d] & \mathrm{Map}_*((M-M')^{+}, X) \ar[d] \\
\mathrm{Map}_*((M-M'')^{+}, X) \ar[r] & \mathrm{Map}_*(M_0^+,X)
}$$

is a homotopy pullback diagram of $G$-spaces.
\end{lem}

\begin{proof}
  The embeddings $M_0 \hookrightarrow M-M'$ and
  $M - M' \hookrightarrow M$ are proper, therefore the natural maps  
\begin{equation*}
M_0^+ \to (M-M')^+, (M-M')^+ \to M^+ \text{ (and similarly with $M''$)}
\end{equation*}
 are defined and are pointed maps. The maps in the diagram are
restrictions along those. $M_0$ is a closed submanifold, so $M_0^+$ is a closed subspace of
$(M-M')^+$ and of $(M-M'')^+$, and is the intersection of the two.
Therefore $M^+$ is the pushout $(M-M')^+ \cup_{(M_0)^+} (M-M'')^+$. 
This shows that the square is a pullback square. It is also a homotopy pullback since the
restriction maps to $ \mathrm{Map}_*((M_0)^{+},X)$ are $G$-fibrations, as shown
in \cref{prop-G-fibration}.
\end{proof}

\begin{lem}\label{prop-coeq-pullback}
  Suppose $\mapcX{M_0}$ is $G$-connected. Then there is an equivalence between
  the homotopy coherent quotient  
\begin{equation*}
\mapcX{M'} \otimes_{\mapcX{M_0 \times \R}}\mapcX{M''}
\end{equation*}
and the homotopy pullback 
\begin{equation*}
 \mapcX{M-M''} \times_{\mapcX{M_0}} \mapcX{M-M'}.
\end{equation*}
Here, the homotopy coherent action is given by open embeddings via the collar as
  \begin{equation}
    \label{eq:glue}
M' \amalg (M_0 \times \mathbb{R}) \hookrightarrow M', \, 
(M_0 \times \mathbb{R}) \amalg M'' \hookrightarrow M''
\end{equation}
and the homotopy pullback is via the restriction maps in \cref{prop-maps-pb}.
\end{lem}
\begin{proof}
  
We first identify the homotopy pullback with a bar construction. For brevity, we write $B_0$ for $\mapcX{M_0}$, $B''$ for $\mapcX{M-M''}$ and $B'$ for $\mapcX{M-M'}$.
  Then $B_0$ is a based $G$-space with base
  point the constant map to the base point of $X$ and $G$ acts by conjugation.

  Denote by $\times_{B_0}$ the homotopy
  pullback of spaces over ${B_0}$. Consider the restriction
\begin{equation*}
B'= \mapcX{M-M'} \to  \mapcX{M_0} = {B_0}.
\end{equation*}
By \cref{prop-maps-pb} and $M'' \simeq (M-M')\cup_{M_0}(M_0 \times
(-\infty,0])$, there is
\begin{equation*}
  \mapcX{M''} \simeq B' \times_{{B_0}} \mapcX{M_0 \times (-\infty,0]},
\end{equation*}
and similarly,
\begin{equation*}
\mapcX{M'} \simeq B'' \times_{{B_0}} \mapcX{M_0 \times [0,+\infty)}.
\end{equation*}
Writing $B_{I}$ for $\mapcX{M_0 \times I}$ for $I = \R$, $[0, +\infty)$ and
$(-\infty,0]$,
we have a right homotopy action of $B_{\R}$ on $B_{(-\infty,0]}$ and a left
homotopy action of $B_{\R}$ on $B_{[0,+\infty)}$, induce geometrically by
$[0, +\infty) \coprod \R \hookrightarrow [0,+\infty)$ and $\R \coprod (-\infty,
0] \hookrightarrow (-\infty, 0]$. As these geometric maps induces the maps
\eqref{eq:glue}, the homotopy actions of $B_{\R}$ on $\mapcX{M'}$ and
$\mapcX{M''}$  can be identified through these equivalences.

So the homotopy coherent quotient 
is equivalent to the geometric
realization of the simplicial $G$-space
$$ \mathbf{B}(B'' \times_{{B_0}}  B_{[0,+\infty)}, B_\R, B'
\times_{{B_0}} B_{(-\infty,0]}).$$

 Geometric realization commutes with pullbacks of simplicial spaces over a $G$-connected based (\cite[Corollary 11.6]{GILS} or \cite[5.5.6.17]{HA}, noting that the category of $G$-spaces is a topos), so we have $G$-equivalence:
\begin{align*}
  \mathbf{B}(B'' \times_{{B_0}}  B_{[0,+\infty)}, B_\R, B'
\times_{{B_0}} B_{(-\infty,0]})
   &  \simeq B'' \times_{{B_0}} \mathbf{B}( B_{[0,+\infty)}, B_\R, B_{(-\infty,0]})  \times_{{B_0}} B'.
\end{align*}
There is an equivalence $\mathbf{B}( B_{[0,+\infty)}, B_\R, B_{(-\infty,0]})\simeq {B_0}$
as $G$-spaces when ${B_0}$ is $G$-connected.  
Moreover, it is an equivalence as $G$-spaces over ${B_0} \times {B_0}$.
Here, the two augmentation maps of ${B_0}$ are identity on both sides;
the two augmentation maps of $\mathbf{B}( B_{[0,+\infty)}, B_\R, B_{(-\infty,0]})$ over ${B_0}$ are induced
by evaluation at 0 on both sides in the bar construction.
We then have
\begin{align*}
 B'' \times_{{B_0}} \mathbf{B}( B_{[0,+\infty)}, B_\R, B_{(-\infty,0]})
  \times_{{B_0}} B'
  \simeq B'' \times_{{B_0}} {B_0} \times_{{B_0}} B'
  \simeq B'' \times_{{B_0}} B',
\end{align*}
which finishes the proof.
\end{proof}

\begin{lem} \label{cor:G-connected}
  Let $M$ be a smooth $G$-manifold and $N$ be a closed $G$-submanifold.
  Let $X$ be a based $G$-space such that $X^H$ is
  $\mathrm{dim}(M^H)$-connected for all subgroups $H \subset G$.
  Then $\mathrm{Map}_{*}(M/N,X)$ is $G$-connected.
\end{lem}
 Here, since $M$ is a smooth $G$-manifold, $M^H$
 is also a manifold, but possibly empty or with components of different dimensions. We define
  $\mathrm{dim}(\varnothing)=-1$ and $\mathrm{dim}(M^H)$ to be the largest dimension of the components.
\begin{proof}
  Take an equivariant triangulation of $(M,N)$, which exists by \cite[Theorem 3.6]{Illman78}.
  It gives $(M,N)$ a relative $G$-CW structure. Denote $M^{-1}=N$ and
  $S^{-1}=\varnothing$. Then $\mathrm{Map}_{*}(M^{-1}/N,X) = *$ is $G$-connected.
  We induct on the $G$-CW skeleton of $M$. For $k \geq 0$, we have the pushout:
\begin{equation*}
  \begin{tikzcd}
    \coprod_{H_i} G/H_i \times S^{k-1} \ar[d,hook] \ar[r,"f"] & M^{k-1} \ar[d,hook] \\
    \coprod_{H_i} G/H_i \times D^{k} \ar[r] & M^{k} \\
  \end{tikzcd}
\end{equation*}

It gives a cofiber sequence:
\begin{equation*}
M^{k-1}/N \to M^k/N \to \bigvee_{H_i} (G/H_i)_+ \wedge S^k.
\end{equation*}

Mapping into $X$ gives a fiber sequence:
\begin{equation*}
 \prod_{H_i} \mathrm{Map}_{*}((G/H_i)_{+} \wedge S^{k}, X) 
 \to \mathrm{Map}_{*}(M^k/N, X) \to \mathrm{Map}_{*}(M^{k-1}/N,X).
\end{equation*}

For any subgroups $H$ and $H_i$, by the double coset formula, $G/H_i \cong \coprod_{j} H/K_{ij} $ as
$H$-sets, where each of the $K_{ij}$ is the intersection of $H$ with some conjugate of $H_i$. Therefore, 
\begin{align*}
\mathrm{Map}_{*}((G/H_i)_{+} \wedge S^k, X)^{H} & \cong \mathrm{Map}_{*}((\bigvee_j
                                                  (H/K_{ij})_{+}) \wedge
S^{k}, X)^H \\
& \cong \prod_j \mathrm{Map}_{*}(S^{k}, X^{K_{ij}}).
\end{align*}
Since $k \leq dim(M^{H_i}) \leq dim(M^{K_{ij}})$, $X^{K_{ij}}$ is $k$-connected by assumption. So
$\mathrm{Map}_{*}(S^{k}, X^{K_{ij}})$ is connected.
By the long exact sequence of homotopy groups and induction, $\mathrm{Map}_{*}(M^k/N, X)^H$ is connected for all $k$.

This implies that $\mathrm{Map}_{*}(M/N,X)$ is $G$-connected.
\end{proof}

\begin{lem}
  \label{lem:dimM0} Suppose that $M_0 \times \R$ is $V$-framed.
  Then $M_0^H$ is either empty or a manifold of dimension less than $\mathrm{dim}(V^H)$.  \end{lem}
\begin{proof}
 There is an isomorphism of $G$-vector bundles $\mathrm{T}(M_0 \times \R) \cong
 (M_0 \times \R) \times V$. Suppose that $M_0^H \neq \varnothing$ and take $x
 \in M_0^H$. The exponential map $V \cong T_{(x,0)} \to M_0 \times \R$ is a local
 $H$-homeomorphism. Taking $H$-fixed points gives a local chart that establishes
 $(M_0 \times \R)^H \cong M_0^H \times \R$ as a
manifold of dimension $V^H$.
So $\mathrm{dim}(M_0^H) < \mathrm{dim}(V^H)$.
\end{proof}

\begin{prop}($\otimes$-excision)
  \label{lem:NAPD-G-excision}
  Let $M$ be a $V$-framed $G$-manifold with the collar decomposition in \cref{nota:collar}, and let $X$
  be a $G$-space as in \cref{thm:NPD}. Then there is a $G$-equivalence:
   $$\mathrm{Map}_{*}((M')^+,X) \otimes_{\mathrm{Map}_{*}((
     \R \times M_0)^+,X)}
   \mathrm{Map}_{*}((M'')^+,X) \to \mathrm{Map}_{*}(M^+,X).$$
\end{prop}
\begin{proof}
By \cref{lem:dimM0}, $\mathrm{dim}(M_0^H) < \mathrm{dim}(V^H)$. Then by our assumption
on the connectivity of $X$, \cref{cor:G-connected} applied to the pair $(M,N) =
(\overline{M_0},\partial \overline{M_0})$ shows that $\mathrm{Map}_{*}(M_0^+,X)$ is $G$-connected. So we can use
\cref{prop-coeq-pullback} to identify the homotopy coherent quotient with the homotopy pullback
$\mapcX{M-M''} \times_{\mapcX{M_0}} \mapcX{M-M'}$, then use \cref{prop-maps-pb} to identify it with
$\mathrm{Map}_{*}(M^+, X)$.
\end{proof}

\begin{proof}[Proof of \cref{thm:NPD}]
We claim that the functor in \eqref{eq:map-to-X-framed} is a $G$-factorization homology theory of $V$-framed $G$-manifolds. 
By \cref{cor:GFHaxiom},
  it suffices to show that the functor 
  $\mathrm{Map}_*((-)^+,X)$ satisfies $\otimes$-excision,
  which is shown in \cref{lem:NAPD-G-excision},
and respects sequential unions, which is shown in \cref{lem:NAPD-G-seq-union}.
Furthermore, we can identify the coefficients of the $G$-factorization homology theory
$\ul{\mathrm{Map}}_*((-)^{+},X)$ 
with $ \Omega^VX $ by evaluating at $V$.
So we have an equivalence of $G$-functors
\begin{equation*}
\ul{\mathrm{Map}}_*((-)^{+},X) \simeq \int_{-} \Omega^V X.
\end{equation*}
 The conclusion follows from taking the underlying functor over the orbit $G/G$.
\end{proof}

\section{Equivariant Thom spectra} \label{sec:G_Thom_spectra}

In this section, we will define the $G$-Thom spectrum functor, and show that it strongly preserves $G$-colimits and is $G$-symmetric monoidal. This
will allow us to conclude that it respects equivariant factorization homology.

\medskip

We first recall the construction of Thom spectra according to \cite[def. 2.20]{ABGHR}, an approach that leverages the equivalence of spaces and $\infty$-groupoids. 
The Thom spectrum of a stable spherical fibration $E$ over $X$ is defined as the colimit of 
\begin{align} \label{eq:ABGHR_Thom}
  X \xto{E} \operatorname{Pic}(\Sp) \to \Sp.
\end{align}
Here, $\operatorname{Pic}(\Sp)$ is the Picard space of the $\infty$-category of spectra,
that is, the classifying space of local systems of invertible spectra;
$\operatorname{Pic}(\Sp) \to \Sp$ is the inclusion of a sub-$\infty$-groupoid;
the map $X \xto{E} \operatorname{Pic}(\Sp)$ is the classification map of the spherical fibration $E$,
and the colimit of \eqref{eq:ABGHR_Thom} is indexed by $X$, considered as an $\infty$-groupid.

Together these Thom spectra assemble to a colimit preserving functor from spaces over $\operatorname{Pic}(\Sp)$ to spectra:
\[
  \Ss_{/ \operatorname{Pic}(\Sp)} \simeq \operatorname{Psh}( \operatorname{Pic}(\Sp) ) \to \Sp.
\]
By the universal property of the presheaf category, this functor is characterized by its restriction along the Yoneda embedding $ \operatorname{Pic}(\Sp) \to  \operatorname{Psh}(\operatorname{Pic}(\Sp))$, see \cite[cor. 3.13]{ABGHR},
and therefore given as a left Kan extension 
\[
  \begin{tikzcd}
    & \operatorname{Pic}(\Sp) \ar[d, hook] \ar[rd] \\
     \Ss_{/ \operatorname{Pic}(\Sp)} \ar[r,"\simeq"]  & \operatorname{Psh}( \operatorname{Pic}(\Sp) ) \ar[r, dashrightarrow]  & \Sp.
  \end{tikzcd}
\]
We apply a similar approach to construct a $G$-equivariant Thom spectrum.
The goal of this section is the following theorem.
\begin{thm} \label{G_Th_omnibus_thm}
  There exists a $G$-symmetric monoidal functor 
  \[
    \Th \colon \ulTopG_{/\B} \simeq  \Psh(\B) \to \ul\Sp^G 
  \]
  that strongly preserves $G$-colimits.
  Moreover, let  $E\in \Sp^G$ be an invertible genuine $G$-spectrum and  $e \colon X \to \B$ be a
  $G$-map  from a $G$-space $X$ such that $e$ is $G$-homotopic to the constant map with value $E$.
  Then the functor $\Th$ takes $e\in \left(\ulTopG_{/\B}\right)_{[G/G]}$
  to the genuine $G$-spectrum $E\otimes \Sigma^{\infty}_+X\in \Sp^G$. 
\end{thm}
Let us briefly explain the notation used in \cref{G_Th_omnibus_thm}.
Since $\Sp$ is the fiber of the $G$-$\infty$-category $\ul\Sp^G$ over $G/e$
\footnote{One should think of $\ul\Sp^G$ as defining a nontrivial $G$-action on $\Sp$.}, 
we can endow $\operatorname{Pic}(\Sp)$ with a $G$-action whose $H$-fixed points
are equivalent to the Picard space of genuine $H$-spectra.

We call the resulting $G$-space the \myemph{Picard $G$-space of $\ul\Sp^G$}, and
think of it as an object in the category of $G$-spaces, $\B \in \ulTopG$ (see
\cref{G_Picard} below for details).

Finally, the $G$-$\infty$-category $\ulTopG_{/\B}$ is the parametrized slice category of $\ulTopG$:
its fiber $\left( \ulTopG_{/\B} \right)_{[G/H]}$ is equivalent to the slice
$\infty$-category $\mathbf{Top}^H_{/\B}$ of $H$-spaces over $\B$.
In particular, the fiber $\left(\ulTopG_{/\B}\right)_{[G/G]}$ is equivalent to the category of $G$-spaces over $\B$, and its objects are given by maps of $G$-spaces $e \colon X\to \B$.

\begin{rem} \label{rem:p-local_Thom}
  For \cref{G_Th_omnibus_thm} and its applications,
  one could also work with $p$-local genuine $G$-spectra. 
  Let $ \ul\Sp^G_{(p)} \subset \ul\Sp^G$ be the $G$-subcategory of (fiberwise) $p$-local spectra.
  Note that the $G$-symmetric monoidal structure of $\ul\Sp^G$ induces a $G$-symmetric monoidal structure on $\ul\Sp^G_{(p)} $, 
  as the essential image of the $G$-localization
  \[
    - \otimes \mathbb{S}_{(p)} \colon \ul\Sp^G \to \ul\Sp^G ,
  \] 
  see \cite[thm. 2.9.2 and rem. 2.9.3]{Parametrized_algebra}.
  Replacing $\ul\Sp^G$ with $ \ul\Sp^G_{(p)}$, we obtain a $p$-local Thom spectrum functor
  \[
    \Th \colon \ulTopG_{/\Pic(\ul\Sp^G_{(p)})} \simeq  \Psh(\Pic(\ul\Sp^G_{(p)})) \xto{\Th'} \ul\Sp^G_{(p)}. 
  \]
  The entire section as well as \cref{sec:Thom_GFH} and \cref{sec:v-fold-loop} hold mutatis mutandis.
  In particular, the $p$-local $G$-Thom spectrum of a $G$-map $X \to \Pic(\mathbb{S}_{(p)})$ is given by 
  \[
    \Th \left( X \to \Pic(\mathbb{S}_{(p)}) \to \Pic(\ul\Sp^G_{(p)}) \right) ,
  \]
  where the second map is the inclusion $\Pic(\mathbb{S}_{(p)}) \subset \Pic(\ul\Sp^G_{(p)})$.
\end{rem}

\subsection{The Picard $G$-space} \label{G_Picard}
In this subsection we construct the Picard $G$-space $\Pic(\ul\Sp^G)$.
In order to define a $G$-symmetric monoidal structure on the parametrized slice category $\ul\Top^G_{/\Pic(\ul\Sp^G)}$ we will need to endow $\Pic(\ul\Sp^G)$ with a $G$-symmetric monoidal structure, encoding the fact that the Hill-Hopkins-Ravenel norm $N_H^G$ sends invertible $H$-spectra to invertible $G$-spectra (since $N_H^G$ is symmetric monoidal).
More generally, we will define the Picard $G$-space of a $G$-symmetric monoidal $G$-$\infty$-category and show that it inherits a $G$-symmetric monoidal structure. 

Let $p \colon\ul\C^\otimes \fib \GFin_*$ be a $G$-symmetric monoidal
$G$-$\infty$-category and let $\ul\C \fib \OGop$ be its underlying
$G$-$\infty$-category. (See \cref{sec:prelim} for definitions.)
Each fiber $\ul\C_{[G/H]}$ of the underlying $G$-$\infty$-category is endowed with a symmetric
monoidal structure, defined by the pull back of $p$ along 
\[
  \Fin_* \to \GFin_* ,\quad I \mapsto (I \times G/H \to G/H).
\]
\begin{mydef}
  An object $x$ in the $G$-$\infty$-category $\ul\C$ which is over $G/H$ is \myemph{invertible} if $x$
  is an invertible object in the $\infty$-category $\ul\C_{[G/H]}$, that is, the object $x\in
  \ul\C_{[G/H]}$ is dualizable  and the evaluation map $ x \otimes x^\wedge \to \mathbbm{1}$ is an
  equivalence.
  The Picard $G$-space $\Pic(\ul\C) \to \OGop$ is the maximal $G$-$\infty$-groupoid of $\ul\C$ spanned by invertible objects, i.e. $\Pic(\ul\C) \subseteq \ul\C$ is the subcategory spanned by invertible objects and coCartesian morphisms.
\end{mydef}

\begin{rem}
  \label{rem:ElmendorfPic}
  By construction, $\Pic(\ul\C)$ is a $G$-$\infty$-groupoid, that is,
  the map $\Pic(\ul\C) \to \OGop$ is a left fibration.
  Since the $\infty$-category of left fibrations over $\OGop$ is equivalent to the $\infty$-category of $G$-spaces, 
 there is a $G$-space corresponding to the $G$-$\infty$-groupoid $\Pic(\ul\C)$.
  By abuse of notation, we write $\Pic(\ul\C) \in \GTop$ for this $G$-space, with $H$-fixed points given by $\Pic(\ul\C)_{[G/H]} \simeq \operatorname{Pic}(\ul\C_{[G/H]})$.
\end{rem}
 
\paragraph{The $G$-symmetric monoidal structure of the Picard $G$-space.}
In \cref{GSM_str_on_max_G_subgpd}, it is shown that the left fibration 
\[ 
  \ul\C^\otimes_{coCart} \fib \GFin_*
\]
endows the maximal $G$-$\infty$-groupoid $\ul\C^\simeq \subseteq \ul\C$ with a $G$-symmetric monoidal structure.
This $G$-symmetric monoidal structure further restricts to a $G$-symmetric monoidal structure
$$\Pic(\ul\C)^\otimes \fib \GFin_*$$ on $\Pic(\ul\C) \subseteq \ul\C^\simeq$, as we now explain.

Let $I=(U \to G/H) \in \GFin_*$ be an object. As shown in \cref{GSM_str_on_max_G_subgpd},
the $G$-Segal map for $\ul\C^\otimes_{coCart}$ gives an equivalence
\( (\ul\C^\otimes_{coCart})_{[I]} \iso \prod_{W\in \orb(U)} \ul\C^\simeq_{[W]} \).
\begin{construction}
  Let $\Pic(\ul\C)^\otimes \subseteq \ul\C^\otimes_{coCart}$ be the full subcategory whose fiber 
  over the object $I=(U \to G/H) \in \GFin_*$,
  \[
     \left( \Pic(\ul\C)^\otimes \right)_{[I]} \subseteq \left( \ul\C^\otimes_{coCart} \right)_{[I]} \simeq \prod_{W \in \orb{U}} \ul\C^\simeq_{[W]},
  \]
  is spanned by tuples of invertible objects 
  \[
    (x_{[W]}) \in \prod_{W \in \orb{U}}  \operatorname{Pic}(\ul\C_{[W]}).
  \]
  Here, $ \operatorname{Pic}(\ul\C_{[W]}) \subseteq (\ul\C_{[W]})^\simeq = \ul\C^\simeq_{[W]}$ is the Picard space of the fiber $\ul\C_{[W]}$.
\end{construction}
\begin{lem}
  The restriction of \( \ul\C^\otimes_{coCart} \fib \GFin_* \) to \( \Pic(\ul\C)^\otimes \subseteq \ul\C^\otimes_{coCart} \) defines a $G$-symmetric monoidal structure on $\Pic(\ul \C)$.
\end{lem}
\begin{proof}
  Let $I=(U \to G/H) \in \GFin_*$ be an object.
  By construction, we have
  $$(\Pic(\ul\C)^\otimes)_{[I]} \simeq  \prod_{W \in \orb{U}} \operatorname{Pic}(\ul\C_{[W]}).$$
  
  We show that the restriction
  \[
    p \colon \Pic(\ul\C)^\otimes  \subseteq \ul\C^\otimes_{coCart} \overset{p'}{\fib} \GFin_*
  \]
  is a left fibration. In the language of \cite[defn. 4.4]{Expose1},
  this amounts to proving that $\Pic(\ul\C)^\otimes$ is a
  $\GFin_*$-subcategory. 
  By \cite[lem. 4.5]{Expose1},
  it suffices to check the following: a $p'$-coCartesian edge $x \to y$ lies in $
  \Pic(\ul\C)^\otimes$ just in case $x$ does.
  These are true since in each fiber $\Pic(\ul\C)_{[G/H]} \subset \ul\C^\simeq_{[G/H]}$,
  invertible objects are closed under tensor products,
  and for every $\varphi \colon G/K \to G/H$, the norm functor
  $\otimes_\varphi \colon \ul\C_{[G/K]} \to \ul\C_{[G/H]}$ 
  is symmetric monoidal, and in particular preserves invertible objects.
  
  The $G$-Segal conditions for $\ul\C^\otimes_{coCart}$ restrict to give the $G$-Segal conditions for $\Pic(\ul\C)^\otimes$. 
\end{proof}

  By \cref{G_Picard_as_G_comm_alg}, we can think of $\Pic(\ul\C)$  as a $G$-commutative algebra in $(\ulTopG)^\times$.

\subsection{The Thom spectrum $G$-functor}

We are now ready to construct the $G$-functor of \cref{G_Th_omnibus_thm}.
\begin{mydef} \label{Th_construction}
  Let $\Th' \colon \Psh(\B) \to \ul\Sp^G$ be the $G$-left Kan extension of the inclusion $\B \to \ul\Sp^G$ along the parametrized Yoneda embedding $ j \colon \B \to \Psh(\B)$.
  Define 
  \[
    \Th \colon \ulTopG_{/\B} \simeq  \Psh(\B) \xto{\Th'} \ul\Sp^G 
  \]
  to be $\Th'$ precomposed with an inverse of the natural equivalence of \cref{parametrized_st_unst_over_gpd}.
\end{mydef}

\begin{prop} \label{Th_pres_colim}
  A $G$-functor $\ulTopG_{/\B}  \to \ul\Sp^G$ is equivalent to $\Th$ if and only
  if it strongly preserves $G$-colimits and its restriction along the parametrized Yoneda embedding 
  \[ 
   \B \overset{j}{\cof} \Psh(\B) \simeq \ulTopG_{/\B}
  \] 
  is equivalent to the inclusion $\B \to \ul\Sp^G$.
\end{prop}
\begin{proof}
  This is an application of \cite[thm. 11.5]{Expose2}.
\end{proof}
\begin{cor}
  The restriction of the $G$-functor $\Th \colon \ulTopG_{/\B} \to \ul\Sp^G$ to the fiber over $G/e\in \OGop$ is equivalent to the Thom spectrum functor 
  \[
    \Ss_{/ \operatorname{Pic}(\Sp)} \simeq \operatorname{Psh}( \operatorname{Pic}(\Sp) ) \to \Sp 
  \]
 of \cite{ABGHR}.
\end{cor}
\begin{proof}
  It suffices to show the restriction of the Thom spectrum $G$-functor $\Th$ satisfies the universal
  property of the Thom spectrum functor as in \cite[cor. 3.13]{ABGHR}.

  We can identify the fibers over the orbit $G/e$ of the corresponding $G$-$\infty$-categories as
  \begin{align*}
    \left( \ulTopG_{/\B} \right)_{[G/e]} & \simeq \left( \ulTopG_{[G/e]} \right)_{/\B_{[G/e]}}  \simeq \Ss_{/ \operatorname{Pic}(\Sp)} , \quad
    \left( \ul\Sp^G \right)_{[G/e]} \simeq \Sp.
  \end{align*}
  By construction, the functor 
  \[
    \Ss_{/ \operatorname{Pic}(\Sp)} \simeq \left( \ulTopG_{/\B} \right)_{[G/e]} \xto{\Th_{[G/e]}} \left( \ul\Sp^G \right)_{[G/e]} \simeq \Sp
  \]
  preserves colimits and its restriction along 
  \[
    \operatorname{Pic}(\Sp) \simeq \B_{[G/e]} \to \Psh(\B)_{[G/e]} \simeq \Ss_{ /\operatorname{Pic}(\Sp)}
  \]
  is equivalent to the canonical inclusion \( \operatorname{Pic}(\Sp) \simeq \Pic(\ul\Sp^G)_{[G/e]} \to (\ul\Sp^G)_{[G/e]} \simeq \Sp \).
\end{proof}

\medskip

The Thom spectrum of a stably trivial sphere bundle over $X$ is given by a smash product $ S^n \otimes \Sigma^\infty_+ X$ with an invertible spectrum $S^n \in \Sp$.
Our next goal is an equivariant version of this fact. 

We will use the following notation when discussing $G$-Thom spectra. 
\begin{notation}
  Consider $\B$ as a $G$-space by \cref{rem:ElmendorfPic}. 
  An invertible $H$-spectrum $E\in \left(\B \right)_{[G/H]} \simeq \operatorname{Pic}(\Sp^H)$ is
  then an $H$-fixed point of the $G$-space $\B$, which corresponds to an $H$-equivariant map $ \ast \to \B$.
  We denote by
  \(
    f_E \colon \ast \to \B 
  \)
  the $H$-map corresponding to $E$.
\end{notation}

\begin{prop} \label{Th_of_representables}
  The $G$-functor $\Th$ sends the map of $H$-spaces $f_E \colon\ast \to \B$ to
  its corresponding invertible $H$-spectrum $E \in \B_{[G/H]} {\subset \ul\Sp^G_{[G/H]}}$.
\end{prop}
\begin{proof}
  The parametrized Yoneda embedding $ j \colon \B \to \Psh(\B)$ is fully faithful (\cite[thm. 10.4]{Expose1}).
  By \cite[prop. 10.5]{Expose2} the composition $\Th' \circ j$ is equivalent to the inclusion $\B \to \ul\Sp^G$, so $\Th'(j(E)) = E $.
  We will finish the proof by showing that $f_E \colon\ast \to \B$ corresponds to $j(E) \in
  \ulTopG_{/\B}$ under the equivalence
  $$\ul{\operatorname{Psh}}_G(\B) \simeq \ulTopG_{/\B}$$ in \cref{parametrized_st_unst_over_gpd}.

  The equivalence of \cref{parametrized_st_unst_over_gpd} sends the representable presheaf 
  \[
    j(E) \in \Psh(\B)_{[G/H]}
  \]
  to a $\ul{G/H}$-functor of $\ul{G/H}$-$\infty$-groupoids 
  \[
    \left( \B_{/\ul{E}} \fib \B \ultimes \ul{G/H} \right) \in
    \left( \Top_{\ul{G/H}} \right)_{/ (\B \ultimes \ul{G/H}) } 
    \simeq 
    \left( \ulTopG_{/\B} \right)_{[G/H]}.
  \]
  Note that the fibers of $(\B_{/\ul{E}})$ are all contractible as slices of $\infty$-groupoids,
  so the natural $\ul{G/H}$-functor 
  \(
    \sigma_{E} \colon \ul{G/H} \to (\B)_{/\ul{E}}
  \)
  is a $\ul{G/H}$-equivalence.
  It follows that the $\ul{G/H}$-functor $\B_{/\ul{E}} \fib \B \ultimes \ul{G/H}$ is equivalent to the composition
  \[
    \ul{G/H} \xto{\sigma_{E}} (\B)_{/\ul{E}} \fib \B \ultimes \ul{G/H}.
  \]
  This composition is precisely the $\ul{G/H}$-object $E \colon \ul{G/H} \to \B \ultimes \ul{G/H}$ associated to $E$, which corresponds to the $H$-map $f_E \colon \ast \to \B$ under the isomorphism 
  \[
    \Fun_{\ul{G/H}}(\ul{G/H}, \B \ultimes \ul{G/H}) \= \Fun_G(\ul{G/H}, \B ) \= \Map_H(\ast, \B). \qedhere
  \] 
\end{proof}

Using \cref{Th_of_representables} we can calculate the equivariant Thom spectrum of $G$-nullhomotopic maps.
\begin{prop} \label{prop:G_Thom_of_nullhomotpic_map}
  Let $X\in \GTop \simeq \ulTopG_{[G/G]}$ be a $G$-space and $E \in \operatorname{Pic}(\Sp^G)\simeq
  \B_{[G/G]}$ be an invertible $G$-spectrum, then 
  \[
    \Th( X \to \ast \xto{f_E} \B) \simeq E \otimes \Sigma^\infty_+ X.
  \]
\end{prop}
\begin{proof}
  The point $\ast \in \GTop \simeq \ulTopG_{[G/G]}$ is the terminal $G$-space.
  Express the $G$-space $X$ as $\Gcolim_X(\ast)$, the $G$-colimit of the constant $G$-diagram 
  \[
    X \to \ul{G/G} \to \ulTopG,
  \]
  where the second functor corresponds to the terminal $G$-space. 
  Postcomposition with $f_E$ induces a $G$-functor  
\begin{equation*}
  \begin{array}{cccccc}
    (f_E)_* \colon & \ulTopG & \simeq & \ulTopG_{/\ul{\ast}} & \to & \ulTopG_{/\B} , \\
     & X & \mapsto & (X\to \ast) & \mapsto & ( X \to \ast \xto{f_E} \B ) ,
  \end{array}
\end{equation*}
  and this $G$-functor strongly preserves $G$-colimits.
  Since $(f_E)_*(\ast) = \left( f_E \colon \ast \to \B \right)$ we have  
  \[
    (f_E)_*(X) =  (f_E)_* \left( \Gcolim_X(\ast) \right) \simeq \Gcolim_X \left( f_E \right) \in \left(\ulTopG_{/\B} \right)_{[G/G]}.
  \]
  We can now apply the $G$-functor $\Th$ of \cref{Th_construction} to $(f_E)_*(X) $.
  By \cref{Th_pres_colim} and \cref{Th_of_representables} we have 
  \[
    \Th \left( (f_E)_*(X) \right) \simeq \Th \left( \Gcolim_X \left( f_E \right) \right) \simeq  \Gcolim_X \left( \Th \left(f_E \right) \right) \simeq \Gcolim_X (E).
  \]

  On the other hand, 
  \[
    E \otimes \Sigma^\infty_+ X \simeq  E \otimes \Sigma^\infty_+ \left( \Gcolim_X (\ast) \right) \simeq \Gcolim_X \left( E \otimes \Sigma^\infty_+ \ast \right) \simeq \Gcolim_X (E).
  \]
  Together we have  
  \[
    \Th \left( X \to \ast \xto{f_E} \B \right) \simeq E \otimes \Sigma^\infty_+ X , 
  \]
  as claimed.
\end{proof}

\medskip

We end this section by extending $\Th$ to a $G$-symmetric monoidal functor.

We first describe the $G$-symmetric monoidal structure on $(\ulTopG)_{/\B}$.
We have seen in \cref{G_Picard} that 
the $G$-symmetric monoidal structure of $\ul\Sp^G$ induces a $G$-symmetric monoidal structure on the $G$-$\infty$-groupoid $\B$,
and that we can consider $\Pic(\ul\Sp^G)$ as a $G$-commutative algebra in $(\ulTopG)^\times$. 
Therefore, we can endow the parametrized slice category $(\ulTopG)_{/\B}$ with the parametzied slice $G$-symmetric monoidal structure of \cref{G_slice_GSM_str}.
\begin{prop} \label{thm:G_Thom_as_GSM_functor}
  The $G$-functor $\Th$ of \cref{Th_construction} extends to a $G$-symmetric monoidal functor
  \[
    \Th^\otimes \colon \left( \ulTopG_{/\B} \right)^\otimes \to (\ul\Sp^G)^\otimes ,
  \]
  where $\left( \ulTopG_{/\B} \right)^\otimes$ is the slice $G$-symmetric monoidal structure of \cref{G_slice_GSM_str}.
\end{prop}
\begin{proof}
  By \cref{prop:GSM_G_LKE}, the $G$-left Kan extension $\Th'$ extends to a $G$-symmetric monoidal functor.
  The result follows from the fact that the equivalence of \cref{parametrized_st_unst_over_gpd} extends to a $G$-symmetric monoidal equivalence, see \cref{GSM_straightening_unstraightening}.
\end{proof}

\begin{cor} \label{prop:G_Thom_of_projection_map}
  Let $X$ be a $G$-space, let $f \colon Y \to \B$ be a $G$-map and 
  \[ 
    A = \Th(Y \xto{f} \B) 
  \] 
  be its $G$-Thom spectrum.
  Then we have an equivalence of genuine $G$-spectra
  \[
    \Th (X \times Y \xto{pr} Y \xto{f} \B) \simeq A \otimes \Sigma^\infty_+ X.
  \]
\end{cor}
\begin{proof}
  The tensor product in $(\ulTopG)_{/\B}$ admits the following description
  \begin{align*}
    \left( Y \xto{f} \B \right) & \otimes \left( X \to \ast \xto{f_{\mathbb{S}}} \B \right) \\
    & \simeq \left( X \times Y \to \ast \times Y \xto{f_{\mathbb{S}} \times f} \B \times \B \xto{\otimes} \B \right) ,
  \end{align*}
  where $\mathbb{S}\in \Sp^G$ is the $G$-equivariant sphere spectrum.
  Since  
  the map $f_\mathbb{S} \colon \ast \to \B$ is constant on $\mathbb{S}$, the unit of $\operatorname{Pic}(\Sp^G) = \B_{[G/G]} $,
  the composition 
  \[
    Y = \ast \times Y \xto{f_{\mathbb{S}} \times f} \B \times \B \xto{\otimes} \B  
  \]
  is equivalent to $ f \colon Y \to \B$.
  Therefore 
  \begin{align*}
    \left( Y \xto{f} \B \right) \otimes \left( X \to \ast \xto{\mathbb{S}} \B \right) 
    \simeq \left( X \times Y \xto{pr}  Y \xto{f}  \B \right).
  \end{align*}
  The $G$-Thom functor is $G$-symmetric monoidal so
  \begin{align*}
    \Th   \left( X \times Y \xto{pr}  Y \xto{f}  \B \right) 
    & \simeq 
    \Th \left( Y \xto{f} \B \right) \otimes \Th \left( X \to \ast \xto{\mathbb{S}} \B \right) \\
    & \simeq 
    A \otimes \left( \mathbb{S} \otimes \Sigma^\infty_+ X \right) 
    \simeq 
    A \otimes \Sigma^\infty_+ X ,
  \end{align*}
  where the second equivalence follows from \cref{prop:G_Thom_of_nullhomotpic_map}.
\end{proof}

\subsection{Genuine equivariant factorization homology of  Thom spectra} \label{sec:Thom_GFH}
We prove that our $G$-Thom spectrum functor respects equivariant factorization homology.

We will use the following equivariant version of \cite[lem. 3.25]{AF}.
\begin{lem} \label{GFH_coeffs_change}
  Let $\ul\C^\otimes, \ul\D^\otimes$ be presentable $G$-symmetric monoidal $G$-$\infty$-categories, and let 
  \[
    F \colon \ul\C^\otimes \to \ul\D^\otimes
  \]
  be a $G$-symmetric monoidal $G$-functor whose restriction to the underlying $G$-$\infty$-categories strongly preserves $G$-colimits (see \cite[def. 11.2]{Expose2}).
  Let \( A\in \Alg_{\EE_V}(\ul\C) \) be an $\EE_V$-algebra in $\C$.
  Then the composition 
  \begin{align*}
    F \circ \int_- A \colon \ulmfld^{G,V-fr} \to \ul\C \to \ul\D
  \end{align*}
  is equivalent to $\int_- F(A) \colon \ulmfld^{G,V-fr} \to \ul\D$, where $F(A)$ is the
  $\EE_V$-algebra in $\D$ defined as
  \[
    \ul\disk^{G,V-fr,\sqcup} \xto{A} \ul\C^\otimes \xto{F} \ul\D^\otimes.
  \]
\end{lem}
\begin{proof}
  The $G$-functor $F\circ \int_- A$ extends to a $G$-symmetric monoidal $G$-functor 
  \begin{align*}
    \ulmfld^{G,V-fr,\sqcup} \xto{\int_- A} \ul\C^\otimes \xto{F^\otimes} \ul\D^\otimes.
  \end{align*}
  The $G$-symmetric monoidal functor $\int_- A$ satisfies $G$-$\otimes$-excision and respects $G$-sequential unions (see \cite[prop 5.2.3, prop. 5.3.3]{GFH}).
  Since $F$ strongly preserves $G$-colimits,  
  $F \circ \int_- A$ also satisfies $G$-$\otimes$-excision and respects $G$-sequential unions.
  By \cref{thm:GFHaxiom}, 
 \( F \circ \int_- A \) is equivalent to a genuine $G$-factorization homology 
  \( 
     \ulmfld^{G,V-fr,\sqcup} \to  \ul\D^\otimes 
  \)
  with coefficients given by the restriction of $F \circ \int_- A$ to
  $\ul\disk^{G,V-fr,\sqcup}$, which is exactly $F(A)$.
\end{proof}

\begin{prop} \label{eq:G_Thom_commutes_with_GFH}
  If $A$ is 
  an $\EE_V$-algebra
  in $\ulTopG_{/\B}$, then we have a natural equivalence 
  \begin{align*}
    \int_- \Th(A) \simeq \Th \left( \int_- A \right).
  \end{align*}
\end{prop}
\begin{proof}
  By \cref{Th_construction} and \cref{thm:G_Thom_as_GSM_functor},
  the $G$-Thom spectrum $$\Th  \colon \ulTopG_{/\B} \to \ul\Sp^G$$ strongly preserves
  $G$-colimits, and extends to a $G$-symmetric monoidal $G$-functor.
  The claim now follows from \cref{GFH_coeffs_change} with $\ul\C = \ulTopG_{/\B}$ and $\ul\D =\ul\Sp^G$.
\end{proof}

\section{$V$-fold loop spaces and $\EE_V$-algebras}
\label{sec:v-fold-loop}
Recall that the coefficients for $\displaystyle\int_M -$, where $M$ is a $V$-framed $G$-manifold, are
 $\EE_V$-algebras as in \cref{def:Vdisk_alg}.
The $\EE_V$-algebras  we consider in this paper typically arise in two ways: as Thom spectra of $V$-fold loop maps, or as $G$-commutative algebras. 
In this section, we will explain how we consider each of those as an $\EE_V$-algebra, and give a description of the equivariant factorization homology of the $G$-Thom spectrum of a $V$-fold loop map. 

\subsection{$G$-commutative algebras as $\EE_V$-algebras}
\label{sec:g-comm-alg}
The first examples of $\EE_V$-algebras are given by $G$-commutative algebras. A $G$-commutative algebra is a map of
$G$-$\infty$-operads from the terminal $G$-$\infty$-operad to a $G$-symmetric monoidal
$G$-$\infty$-category $\ul\C^\otimes \to \GFin_*$ (\cite[ex. 3.3]{Nardin_thesis}).
Since the terminal $G$-$\infty$-operad $\GFin_* \to \GFin_*$ is itself a $G$-symmetric monoidal $G$-$\infty$-category, a $G$-commutative algebra $A \in \CAlg_G(\ul\C)$ is a \myemph{lax} $G$-symmetric monoidal functor $A \colon \GFin_* \to \ul\C^\otimes$.

Note that the structure map \( \ul\disk^{G,V-fr,\sqcup} \to \GFin_* \) can itself be considered as a $G$-symmetric monoidal functor.
Therefore we can consider any $G$-commutative algebra $A\in \CAlg_G(\ul\C)$ as an $\EE_V$-algebra by precomposition with the structure map 
\begin{align*}
  \ul\disk^{G,V-fr,\sqcup} \to \GFin_* \xto{A} \ul\C^\otimes.
\end{align*}

For $\ul\C^\otimes = \ul\Top^{G,\times}$ the notion of $G$-commutative algebra agrees with a $G$-symmetric monoidal structure on a $G$-space, considered as a $G$-$\infty$-groupoid.

\subsection{$V$-fold loop spaces as $\EE_V$-algebras} 
\label{sec:v-fold-loop-space}
In this subsection, we explain how to establish $V$-fold loop spaces as $G$-symmetric monoidal functors 
\begin{equation*}
  \Omega^VX: \ul\disk^{G,V-fr,\sqcup} \to \ul\Top^{G,\times}.
\end{equation*}

\begin{construction} \label{EValg_of_V_loop_space}
Fixing an $n$-dimensional representation $V$, we can precompose the $G$-$\infty$-functor
$\ul{\Map}_*\left( (-)^+,X \right)$ defined in \eqref{eq:Mapc_X} with the forgetful
  map from $V$-framed $G$-manifolds to $G$-manifolds. We can further restrict to $V$-framed $G$-disks
  and obtain
  \begin{align}
    \label{OmegaVX}
    \ul\disk^{G,V-fr,\sqcup} \subset \ulmfld^{G,V-fr,\sqcup} \to \GmfldD \xto{(-)^+ } (\ul{\Top}^{G,\vee}_*)^{vop}  \xto{\ul{\Map}_*(-,X)} \ul{\Top}^{G,\times}.
  \end{align}
  We denote the composite by $\Omega^VX$, and it is an $\EE_V$-algebra in $\ulTopG$. 
  The underlying $G$-space of $\Omega^V X$ is given by evaluating the functor $\Omega^V X$ at $V\in \ul\disk^{G,V-fr}_{[G/G]}$, 
  which is the $G$-space $\Map_{*}( S^V, X)$ of pointed maps from the representation sphere $S^V =V^+$ to $X$. 
  Note that the $G$-space $\Map_{*}( S^V, X)$ is equivalent to the $V$-fold loop space of $X$, which justifies the name  $\Omega^V X$. 
\end{construction}

\begin{rem}
  Note that \cref{EValg_of_V_loop_space} is functorial in $X\in \Top^G_{*} \simeq (\ulTopG_*)_{[G/G]}$.
  The parametrized Yoneda embedding $j \colon \ulTopG_* \to \Psh(\ulTopG)$ of \cite[def. 10.2]{Expose1} is a $G$-functor,
  and in particular defines a map between the fibers over $G/G\in \OGop$, which is a functor of $\infty$-categories
  \begin{align*}
    \Top^G_{*} \to \Fun_G((\ulTopG_*)^{vop} , \ulTopG) ,\quad X \mapsto \Map_*(-,X) \colon (\ulTopG_*)^{vop} \to \ulTopG.
  \end{align*}
  This functor factors through the full subcategory of $\Fun_G((\ulTopG_*)^{vop} , \ulTopG)$ spanned by $G$-functors preserving finite $G$-products.
  Let $\Fun^\times_G((\ulTopG_*)^{vop} , \ulTopG)$ denote the $\infty$-category of $G$-symmetric monoidal $G$-functors with respect to the $G$-Cartesian $G$-symmetric monoidal structures.
  Since $\Fun^\times_G((\ulTopG_*)^{vop} , \ulTopG) \subset \Fun_G((\ulTopG_*)^{vop} , \ulTopG)$ is equivalent to the full subcategory described above, the functor $\Map_*(-,X)$ lifts to 
  \begin{align*}
    \Top^G_{*}  \to \Fun^\times_G((\ulTopG_*)^{vop} , \ulTopG) ,\quad X \mapsto \Map_*(-,X) \colon (\ulTopG_*)^{vop} \to \ulTopG.
  \end{align*}

  Precomposition with the $G$-symmetric monoidal functor 
  \begin{align*}
    \ul\disk^{G,V-fr,\sqcup} \subset \ulmfld^{G,V-fr,\sqcup} \to \GmfldD \xto{(-)^+ } (\ul\Top^{G,\vee}_*)^{vop} 
  \end{align*}
  defines a functor of $\infty$-categories
  \begin{align*}
    \Top^G_{*} \to \Fun_G^\times ((\ulTopG_*)^{vop} , \ulTopG) \to  \Fun_G^\otimes ( \ul\disk^{G,V-fr}, \ulTopG) ,
    \quad X \mapsto \Omega^V X.
  \end{align*}
\end{rem}

\subsection{$V$-fold loop maps as $\EE_V$-algebras}
\label{sec:v-fold-loop-map}
Let $f \colon  X\to Y$ be a map of pointed $G$-spaces.
From the functoriality of \cref{EValg_of_V_loop_space} we get $\Omega^V f \colon \Omega^V X \to
\Omega^V Y$, which is a map of $\EE_V$-algebras in $\ulTopG$.
 Suppose $\Omega^V Y$ is a $G$-commutative algebra. That is, there is a $G$-commutative algebra $A$ such that the following construction in \cref{sec:g-comm-alg}
  $$ \ul\disk^{G,V-fr,\sqcup} \to \GFin_* \xto{A} \ulTopG$$
  is equivalent as a $G$-symmetric monoidal functor to
  $$\Omega^V Y  \colon  \ul\disk^{G,V-fr,\sqcup} \to \ulTopG.$$
  Taking $\ul\C^{\otimes} = (\ulTopG)^{\times}$ in 
  \cref{O_algs_in_G_slice}, we have an equivalence of $\infty$-categories 
\[  
  \Alg_{\EE_V} \left( (\ulTopG)_{/\ul{A}} \right) \simeq \Alg_{\EE_V}(\ulTopG)_{/A} ,
\] So, the map $\Omega^V f$ can be considered as an $\EE_V$-algebra in
  $\ulTopG_{/A}$.

  \begin{ex} \label{eq:GFH_G_Thom_V_loop_map}
Consider an $\EE_V$-algebra in $\ulTopG_{/\B}$ of the form $\Omega^V f: \Omega^VX \to \B$.
Applying \cref{eq:G_Thom_commutes_with_GFH} to coefficients $A = \Omega^Vf$,
 we get a natural equivalence of $G$-functors
\begin{align*} 
  \int_- \Th(\Omega^V f) \simeq \Th\left( \int_- \Omega^V f \right):  \ulmfld^{G,V-fr} \to \ul\Sp^G.
\end{align*}
  \end{ex}

\subsection{Equivariant factorization homology of equivariant Thom spectra}
In this subsection we describes the interaction between equivariant Thom spectra and equivariant factorization homology.
Our main result is \cref{thm-fh-thom}, which describes the genuine $G$-factorization homology theory 
\[
  \int_- \Omega^V f  \colon \ulmfld^{G,V-fr} \to \ulTopG_{/\B} 
\]
appearing in \cref{eq:GFH_G_Thom_V_loop_map}.

In fact, this description works when $\B$ is replaced by any $G$-commutative algebra $B$ in $G$-spaces.
For the next propositions we fix $B\in \CAlg_G(\ulTopG)$ to be a $G$-commutative algebra in
$G$-spaces and $\Omega^Vf: \Omega^VX \to B$ to be a map of $\EE_V$-algebras.
By \cref{sec:v-fold-loop-map}, $\Omega^Vf$ can be considered as an $\EE_V$-algebra in
$\ulTopG_{/\ul{B}}$. 

\begin{prop} \label{prop:fgt-fh-loopf}
  Let $\Omega^V f  \colon \Omega^V X \to B$ be a $V$-fold loop map.
  Then we have a natural equivalence 
  \begin{align*}
     \fgt\left( \int_- \Omega^V f \right) \simeq  \int_- \Omega^V X 
  \end{align*}
  of $G$-functors \( \ulmfld^{G,V-fr} \to \ulTopG \).
\end{prop}
\begin{proof}
We have the following properties of the forgetful $G$-functor \(  \ulTopG_{/\ul{B}}  \to \ulTopG \):
\begin{enumerate}
\item The forgetful $G$-functor \( \fgt \colon \ulTopG_{/\ul{B}}  \to \ulTopG \)
  preserves $G$-colimits. 
\item Let $(\ulTopG)^\times$ denote the $G$-Cartesian $G$-symmetric monoidal structure on $\ulTopG$.
  The forgetful $G$-functor \( \fgt \colon \ulTopG_{/\ul{B}}  \to \ulTopG \) extends to a $G$-symmetric monoidal functor 
  \[ 
    \fgt \colon(\ulTopG)^\times_{/\ul{B}}  \to (\ulTopG)^\times.
  \]
\end{enumerate}
Note that the composition \( \ul\disk^{G,V-fr,\sqcup} \xto{\Omega^V f}  (\ulTopG)^\times_{/\ul{B}}  \xto{\fgt} (\ulTopG)^\times \) is equivalent to $\Omega^V X$.
The claim follows from \cref{GFH_coeffs_change}.
\end{proof}

\begin{thm}\label{thm-fh-thom}
  Let $X$ be a pointed $G$-space and $\Omega^V f \colon \Omega^V X \to \B$ be a map of $\EE_V$-algebras.
  Then for every $V$-framed $G$-manifold $M$, there is an equivalence of genuine $G$-spectra
  \begin{equation*}
    \int_M \Th(\Omega^V f) \simeq \Th \left( \int_M \Omega^V X \xto{(\Omega^Vf)_{*}} \int_M \B \to \B \right).
  \end{equation*}
  Here, $\Th \colon \ulTopG_{/\B} \to \ul\Sp^G$ is the parametrized Thom $G$-functor in \cref{Th_construction}.
\end{thm}
\begin{proof}
  We write $B=\Pic(\ul\Sp^G)$.
By \cref{prop:fgt-fh-loopf}, 
$$\int_M \Omega^V f \in (\ulTopG_{/\ul{B}})_{[G/G]} \simeq (\Top^G)_{/B}$$
is given by a map of $G$-spaces
\begin{equation}
  \label{eq:loop-V-f}
  \displaystyle\int_M \Omega^V X \to B.
\end{equation}

Our next task is to describe this map.
Consider $id_B$ as an object of $\Alg_{\EE_V}(\ulTopG)_{/B}$, and observe that the map $\Omega^V f
\colon \Omega^V X \to B$ can be considered as a map $\epsilon: \Omega^V f \to id_B$ in $\Alg_{\EE_V}(\ulTopG)_{/B}\simeq \Alg_{\EE_V}( \ulTopG_{/\ul{B}})$. 
This map of $\EE_V$-algebras induces a natural transformation 
\begin{equation*}
\epsilon_{*} \colon \displaystyle\int_- \Omega^V f \to \displaystyle\int_- id_B.
\end{equation*}
The forgetful functor \( \fgt \colon \ulTopG_{/\ul{B}} \to \ulTopG \) gives
a map \( fgt(\epsilon_{*}) = (\Omega^Vf)_{*} \colon \displaystyle\int_- \Omega^V X \to
\displaystyle\int_- B \) of $G$-spaces over $B$.
It follows that the $G$-map of
\eqref{eq:loop-V-f} factors as 
\begin{align*}
  \int_M \Omega^V X \xto{(\Omega^Vf)_{*}} \int_M B \to B,
\end{align*}
where \(\displaystyle\int_MB \to B \) is given by $\displaystyle\int_M
id_{B}$. The claim then follows from \cref{eq:GFH_G_Thom_V_loop_map}.
\end{proof}

\section{Computations of equivariant factorization homology}
\label{sec:computation}
Assume that $G$ is a finite group and $V$ is a finite dimensional $G$-representation.
In this section, we prove \cref{thm-comp}, which deals with factorization homology when the algebra $A$ is a Thom spectrum of a more highly commutative map than $\EE_V$; it is as commutative as a representation that $M$ embeds in. We apply \cref{thm-comp} to compute the genuine equivariant factorization homology of certain Thom spectra. In \cref{cor-MU}, we compute the factorization homology of the Real bordism
spectrum, $MU_\mathbb{R}$. In \cref{cor-em}, we treat Eilenberg--MacLane spectra; see the appendix
by Hahn--Wilson for a computation of $\THR(\HZ)$. In \cref{cor-relative}, we compute $C_2$-relative
$\THH$ (see \cite{TC_via_the_norm}), $\THH_{C_2}(\HF_2)$. This is an $S^1$-spectrum, but we compute its underlying spectrum.

\subsection{A computation theorem}
\begin{thm}\label{thm-comp} Let $A$ be the $G$-Thom spectrum of an $\EE_{V \oplus W}$-map,
  $$\Omega^{V \oplus W} f: \Omega^{V \oplus W}X \to \Pic(\ul\Sp^G),$$
  with $\pi_k(X^H) = 0$ for all subgroups $H \subset G$ and $k < \mathrm{dim}((V \oplus W)^H)$.
  Let $M$ be a $G$-manifold of the same dimension as the representation $V$.
  Suppose that $M \times W$ embeds equivariantly in $V \times W$, and that there is an equivariant embedding from the unit disk $D(V) \hookrightarrow M$ (call its image $D$). Then

$$\int_{M \times W} A \simeq A \otimes \Sigma^\infty_+ \mathrm{Map}_*(M^+ - D, \Omega^W X).$$
\end{thm}

\begin{rem}
  Recall that we use $\otimes$ to denote the smash product of ($G$-)spectra, and
  $\mathrm{Map}_{*}$ to denote the $G$-space of non-equivariant based maps.
\end{rem}

A particularly useful corollary is obtained by setting $W = \mathbb{R}$ and $M = S^V$.

\begin{cor}\label{cor-comp}
  Let $A$ be the $G$-Thom spectrum of an $\EE_{V \oplus \mathbb{R}}$-map
  $\Omega^{V \oplus \mathbb{R}}X \to \Pic(\ul\Sp^G)$ with $\pi_k(X^H) = 0$ for
  all subgroups $H \subset G$ and $k < \mathrm{dim}(V^H)+1$. 
  Then
  $$\int_{S^V \times \mathbb{R}} A \simeq A \otimes \Sigma^\infty_+ (\Omega X).$$
\end{cor}

\begin{proof}[Proof of \cref{thm-comp}]
  Denote the equivariant embedding by $emb: M \times W \hookrightarrow V \times W$.
  Let $M \times W$ be $(V \oplus W)$-framed as a submanifold.
  Denote by $f: X \to B^{V\oplus W} \Pic(\ul\Sp^G)$ the map whose $\Omega^{V \oplus W}$-looping is
  $\Omega^{V \oplus W} f: \Omega^{V \oplus W}X \to \Pic(\ul\Sp^G)$. 

  Consider the following commutative diagram, where the first horizontal map is an equivalence by
  \cref{thm:NPD}. The right hand column uses the homeomorphism $(M \times W)^+ \cong \Sigma^W(M^+)$.
 
$$\xymatrix{
  \int_{M \times W} \Omega^{V \oplus W} X \ar[r]^-{\sim} \ar[d]^-{(\Omega^{V \oplus W} f )_*} &  \mathrm{Map}_*(\Sigma^W(M^+), X) \ar[d]^-{f_*} \\
  \int_{M \times W} \Pic(\ul\Sp^G) \ar[r] \ar[d]^-{emb_*} & \mathrm{Map}_*(\Sigma^W(M^+), B^{V \oplus W} \Pic(\ul\Sp^G)) \ar[d]^-{emb_*} \\
  \int_{V \times W} \Pic(\ul\Sp^G) \ar[r] \ar[d]^-\sim & \mathrm{Map}_*(S^{V \oplus W}, B^{V \oplus W} \Pic(\ul\Sp^G)) \ar[d]^-\sim \\
  \Pic(\ul\Sp^G) \ar[r]^-= & \Pic(\ul\Sp^G)
}$$

By \cref{thm-fh-thom}, $\displaystyle\int_{M \times W} A$ is the equivariant Thom spectrum of the left hand vertical composite; thus it is equivalent to the equivariant Thom spectrum of the right hand vertical composite. Note that this composite is also equal to
$$\xymatrix{
\mathrm{Map}_*(\Sigma^W(M^+), X) \ar[r]^-{emb_*} & \mathrm{Map}_*(S^{V \oplus W}, X) \ar[r]^-{f_*} & \mathrm{Map}_*(S^{V \oplus W}, B^{V \oplus W} \Pic(\ul\Sp^G)) \simeq \Pic(\ul\Sp^G) 
}$$

The map $emb_*$ above is induced by the embedding $emb: M \times W \hookrightarrow V \times W$, equivalently by the Pontryagin-Thom collapse map associated to it, $S^{V \oplus W} \to \Sigma^W(M^+)$. We have an inclusion of a small disk $D \cong V$ in $M$, and the cofiber sequence
$$\Sigma^W(M^+ -D) \overset{\Sigma^Wi}{\longrightarrow}  \Sigma^W(M^+) \longrightarrow \Sigma^W S^V \cong S^{V \oplus W}$$
is split (up to homotopy) by this Pontryagin-Thom collapse map, as the composite collapse $(V \times W)^+ \to (M \times W)^+ \to (D \times W)^+ \cong (V \times W)^+$ is homotopic to the identity. Thus we have an equivalence

$$\xymatrix{
  (emb_*, i^*): \mathrm{Map}_*(\Sigma^W(M^+), X) \ar[r]^-{\sim}
  & \mathrm{Map}_*(S^{V \oplus W}, X) \times \mathrm{Map}_*(\Sigma^W(M^+ -D), X) \ar[d]^-\sim \\
  & \mathrm{Map}_*(S^{V \oplus W}, X) \times \mathrm{Map}_*(M^+ - D, \Omega^W X)
}$$

Furthermore, this equivalence fits in the following commutative diagram:
$$\xymatrix{
\mathrm{Map}_*(\Sigma^W(M^+), X) \ar[r]^-{(emb_*, i^*)} \ar[r]_-{\sim} \ar[d]^{emb_{*}} & \mathrm{Map}_*(S^{V \oplus W}, X) \times \mathrm{Map}_*(M^+ - D, \Omega^W X) \ar[d]^{pr_1} \\
\mathrm{Map}_*(S^{V \oplus W}, X) \ar[r]^-= \ar[d]^{f_*} & \mathrm{Map}_*(S^{V \oplus W}, X) \ar[d]^-{f_*} \\
\mathrm{Map}_*(S^{V \oplus W}, B^{V \oplus W} \Pic(\ul\Sp^G)) \ar[r]^-= & \mathrm{Map}_*(S^{V \oplus W}, B^{V \oplus W} \Pic(\ul\Sp^G))
}$$

We have shown that $\displaystyle\int_{M \times W}A$ is equivalent to the equivariant Thom spectrum of the left hand vertical composite, thus it is also equivalent to the equivariant Thom spectrum of the right hand vertical composite, which, by \cref{prop:G_Thom_of_projection_map}, is equivalent to $A \otimes \Sigma^\infty_+ \mathrm{Map}_*(M^+ - D, \Omega^W X)$.
\end{proof}

\subsection{Some computational corollaries}
Our first application computes the factorization homology of $MU_\R$.
The Real bordism spectrum $MU_\R$ is the Thom spectrum of a map of $C_2$-$\EE_\infty$ spaces $BU_\mathbb{R} \to \Pic(\ul\Sp^{C_2})$ (for example, as in Remark 13 of \cite{HL}).
Since $(BU_\mathbb{R})^e \simeq BU$ and $(BU_\mathbb{R})^{C_2}\simeq BO$, the $C_2$-space $BU_\mathbb{R}$ is
$C_2$-connected.

\begin{lem}\label{lem-connect}
  If $X$ is a $G$-connected $\EE_V$-algebra, then $\pi_k(B^V X)=0$ for $k \leq \mathrm{dim}(V^H)$.
  Thus, the connectivity condition in \cref{thm-comp} or \cref{cor-comp} is satisfied when $X$ is
  $G$-connected.
\end{lem}
\begin{proof}
  We say that a $G$-space $X$ is $V$-connected if $\pi_k(X)=0$ for $k \leq \mathrm{dim}(V^H)$.
  The $V$-fold delooping can be computed by the monadic bar construction $B^VX = B(\Sigma^V, \mathrm{D}_V,
  X)$, where $\mathrm{D}_V$ is the monad associated to the little $V$-disk operad.
  Since fixed points commutes with geometric realization,
  it suffices to show that each $\Sigma^V \mathrm{D}_VX$ is $V$-connected. This follows from that
  $(\Sigma^V \mathrm{D}_VX)^H \cong \Sigma^{V^H}(\mathrm{D}_VX)^H$ and that $\mathrm{D}_VX$ is
  $G$-connected for a $G$-connected $X$ (for the proof, see \cite[Lemma 4.11]{Zou}).
\end{proof}

 \cref{cor-comp}  and \cref{lem-connect} combine to give
\begin{cor}\label{cor-MU}
We compute the equivariant factorization homology of representation spheres with coefficients in $MU_{\mathbb{R}}$.
$$\int_{S^V \times \mathbb{R}} MU_\mathbb{R} \simeq MU_\mathbb{R} \otimes \Sigma^\infty_+ (B^V BU_\mathbb{R})$$

In particular,
$$\THR(MU_\mathbb{R}) \simeq MU_\mathbb{R} \otimes \Sigma^\infty_+ (B^\sigma BU_\mathbb{R})$$
\end{cor}

\medskip

We now use \cref{thm-comp} and \cref{cor-comp}, along with theorems of Behrens--Wilson \cite{BehrensWilson} and Hahn--Wilson \cite{HahnWilson} which show that certain equivariant Eilenberg--MacLane spectra are Thom spectra, to compute  equivariant factorization homology with coefficients in these spectra. 

Take $G= C_2$. Let $\sigma$ be its sign representation, $\rho \cong \sigma + 1$ its 2-dimensional
regular representation, and $\lambda \cong 2\sigma$ its two-dimensional rotation representation. Let
$\THR$ denote Real topological Hochschild homology \cite{Real_THH}, which is equivalent to
$\int_{S^\sigma}$ by \cite[remark 7.1.2]{GFH}.
\begin{cor}\label{cor-em} 
   We compute $\THR$ and $\int_{S^\lambda}$ of certain Eilenberg--MacLane spectra.
  \begin{enumerate}[(1)]
    \item $\THR(\HF_2) \simeq \HF_2 \otimes \Sigma^\infty_+ (\Omega S^{\rho +1}) \simeq \HF_2 \otimes \Sigma^\infty_+ (\Omega^{\sigma} S^{\lambda +1}) $
    \item $\THR(\HZ_{(2)}) \simeq \HZ_{(2)} \otimes \Sigma^\infty_+ (\Omega^\sigma (S^{\lambda +1}\langle \lambda +1 \rangle))$ 
    \item $\int_{S^\lambda} \HF_2 \simeq \HF_2 \otimes \Sigma^\infty_+ S^{\lambda +1}$
    \item $\int_{S^\lambda} \HZ_{(2)} \simeq \HZ_{(2)} \otimes \Sigma^\infty_+ (S^{\lambda +1}\langle \lambda +1 \rangle)$
  \end{enumerate}
  Here, $S^{\lambda +1}\langle \lambda +1 \rangle$ is the fiber of the unit map
  $S^{\lambda + 1} \to K(\ul\Z, \lambda+1) = \Omega^{\infty}\Sigma^{\lambda+1}\HZ$.
\end{cor}

\begin{proof}
By Theorem 1.2 of \cite{BehrensWilson}, the Eilenberg--MacLane spectrum $\HF_2$ is equivariantly the
Thom spectrum of a $\rho$-fold loop map $\Omega^\rho S^{\rho +1} \to  BO_{C_2}$. As the inclusion $BO_{C_2} \to \Pic(\ul\Sp^{C_2}) $ is a map of $G$-symmetric monoidal $G$-spaces, $\HF_2$ is also the Thom spectrum of a $\rho$-fold loop map $\Omega^\rho S^{\rho +1} \to \Pic(\ul\Sp^{C_2})$.
Thus \cref{cor-comp} yields the first equivalence of (1), with $V = \sigma$ and $W = \mathbb{R}$. 
Furthermore, Hahn and Wilson \cite{HahnWilson} have shown that $\HF_2$ is equivariantly the Thom
spectrum of a $(\lambda+1)$-fold loop map $\Omega^\lambda S^{\lambda + 1} \to \Pic(\mathbb{S}_{(2)})$, and
that $\HZ_{(2)}$ is equivariantly the Thom spectrum of a $(\lambda+1)$-fold loop map $\Omega^\lambda
(S^{\lambda + 1}\langle \lambda +1 \rangle) \to \Pic(\mathbb{S}_{(p)})$.
\cref{cor-comp} with \cref{rem:p-local_Thom} yields (3) and (4), with $V = \lambda$ and $W = \mathbb{R}$.

\medskip
For the second equivalence of (1) and for (2),
there is an isomorphism $\lambda + 1 \cong 2\sigma +1$ and an
equivariant embedding $S^\sigma \times \mathbb{R} \hookrightarrow \sigma +1$, thus an equivariant
embedding $S^\sigma \times \rho \hookrightarrow \lambda + 1$. We intend to use \cref{thm-comp} with
$M=S^{\sigma}$, $V=\sigma$, $W= \rho$ and
$$X=B^{\lambda+1}\Omega^{\lambda}S^{\lambda+1} \text{ or }
X=B^{\lambda+1}\Omega^{\lambda}(S^{\lambda+1}\langle \lambda+1\rangle) \text{ respectively. }$$
To check the assumptions, by \cref{lem-connect}
it suffices to show that $\Omega^\lambda S^{\lambda+1}$ and
$\Omega^{\lambda}(S^{\lambda+1}\langle \lambda+1\rangle)$ are $C_2$-connected. This is true by
\cref{cor:G-connected}, since $\mathrm{dim}((S^{\lambda})^e)=2$ and $\mathrm{dim}((S^{\lambda})^{C_2})=0$;
it can also be verified that $S^{\lambda+1}$ and $S^{\lambda+1}\langle \lambda+1 \rangle$ are
$C_2$-connected and underlying 2-connected.
So, from \cref{thm-comp} we obtain
$$\int_{S^\sigma \times \rho} \HF_{2} \simeq \HF_{2} \otimes \Sigma^\infty_+ \mathrm{Map}_*(\sigma_+, \Omega^\rho B^{\lambda +1} \Omega^\lambda S^{\lambda +1});$$
$$\int_{S^\sigma \times \rho} \HZ_{(2)} \simeq \HZ_{(2)} \otimes \Sigma^\infty_+ \mathrm{Map}_*(\sigma_+, \Omega^\rho B^{\lambda +1} \Omega^\lambda (S^{\lambda +1}\langle \lambda +1 \rangle)).$$
To simplify, we have $\Omega^\rho B^{\lambda +1}
\Omega^\lambda S^{\lambda +1} \simeq  B^{\sigma} \Omega^\lambda
S^{\lambda +1} \simeq \Omega^{\sigma} S^{\lambda +1}$, since $\Omega^{\sigma} S^{\lambda +1}$ is $C_2$-connected.
As $\sigma$ is equivariantly contractible, $ \mathrm{Map}_*(\sigma_+, \Omega^{\sigma}S^{\lambda +1})
\simeq \Omega^{\sigma} S^{\lambda +1}$.
The second equivalence is similar.
\end{proof}

\begin{rem}\label{rem-2^n}
By Theorem B and Theorem E of \cite{HahnWilson}, the same proof can be used to show that (3) and (4) also hold for $G = C_{2^n}$.
\end{rem}

From either of the two descriptions of $\THR(\HF_2)$ in \cref{cor-em}, one can use the Snaith splitting to compute $\THR(\HF_2)$ as an $\HF_2$-module.
From either  $\Sigma^{\infty}_+ \Omega \Sigma S^{\rho} \simeq \oplus_{k \geq 0} \mathbb{S}^{k\rho}$ or $\Sigma^{\infty}_+ \Omega^{\sigma} \Sigma^{\sigma} S^{\rho} \simeq \oplus_{k \geq 0} \mathbb{S}^{k\rho}$, we have
\[
 \THR(\HF_2) \simeq \HF_2 \otimes \Sigma^\infty_+ \Sigma^{\infty}_+ \Omega \Sigma S^{\rho} \simeq \oplus_{k \geq0} \Sigma^{k\rho}\HF_2.
\]
This recovers the additive part of $\THR(\HF_2)$ in \cite[Theorem 5.15]{Real_THH}. 
They use this module structure and the fact that $\THR(\HF_2)$ is an associative $\HF_2$-algebra to promote this to an equivalence of $C_2$-ring spectra. 
In particular,
\[
  \pi_\bigstar \THR(\HF_2) \cong \pi_\bigstar (\HF_2)[x_\rho].
\]

\medskip

Finally, we apply our theory to $C_2$-relative topological Hochschild homology.

\begin{cor}\label{cor-relative}
We compute $\THH_{C_2}(\HF_2)$.
$$\THH_{C_2}(\HF_2) \simeq H\mathbb{F}_2 \otimes \Sigma^\infty_+ (\Omega S^3)$$
\end{cor}

Note that this only computes the underlying (non-equivariant) spectrum of $\THH_{C_2}(\HF_2)$. Section 5 of \cite{AGHKK} uses \cref{G_Th_omnibus_thm} from our paper in a somewhat different approach to compute $\THH_{C_2}(\HF_2)$ as a $C_2$-spectrum.

\begin{proof}
 Let $g$ denote the generator of $C_2$, and for a $C_2$-space $X$, let $L_g X$ denote the twisted
 free loop space $\{ \gamma \colon I \to X \, \lvert \, \gamma(1) = g\gamma (0) \}$. Let $S^1_{rot}$
 denote the circle, with $C_2$ acting by rotation. Note that $S^1_{rot}$ is a $\R$-framed $C_2$-manifold.
 
  By \cite[proposition 7.2.2]{GFH}, for $A$ a $C_n$-ring spectrum, $\THH_{C_n}$ is given by the
  $C_n$-geometric fixed points of $\int_{S^1_{rot}} A$.
  Using the description of $\HF_2$ in the proof of \cref{cor-em} and \cref{thm-fh-thom},
    we have
    $$\displaystyle\int_{S^1_{rot}} \HF_2 \simeq \Th\left(\displaystyle\int_{S^1_{rot}}
      \Omega^{\rho}S^{\rho+1} \to \Pic(\ul\Sp^{C_2})\right).$$
    By \cref{thm:NPD}, we can identify the $G$-spaces:    
\begin{equation}
\label{eq:base}
\displaystyle\int_{S^1_{rot}}\Omega^{\rho}S^{\rho+1} \simeq \Map(S^1_{rot}, \Omega^\sigma S^{\rho + 1})
\end{equation}
Moreover, the Thom spectrum $\int_{S^1_{rot}} \HF_2$ has an $\HF_2$-orientation given by the composite
$$\xymatrix{
\HF_2 \otimes \displaystyle\int_{S^1_{rot}} \HF_2 \ar[r]^-{id \otimes i} & \HF_2 \otimes \HF_2 \ar[r]^-{mult} & \HF_2
}$$
Here, the map $i \colon \int_{S^1_{rot}} \HF_2 \to \HF_2$ exists because $\HF_2$ is commutative. For example, we can take $i$ to be induced on factorization homology by the embedding $S^1_{rot} \times \mathbb{R} \to \lambda$. 
  By the Thom isomorphism and \eqref{eq:base}, we have

$$ \HF_2 \otimes \int_{S^1_{rot}} \HF_2 \simeq \HF_2 \otimes \Sigma^{\infty}_+\mathrm{Map}(S^1_{rot}, \Omega^\sigma S^{\rho + 1})$$

Upon passage to geometric fixed points, we obtain
$$\Phi^{C_2}(\HF_2) \otimes \THH_{C_2}(\HF_2) \simeq \Phi^{C_2}(\HF_2) \otimes \Sigma^\infty_+ \mathrm{Map}_{C_2}(S^1_{rot}, \Omega^\sigma S^{\rho + 1})$$

Because $\Phi^{C_2}(\HF_2) \simeq H\mathbb{F}_2[t]$, where $t$ is in degree 1, we have that

$$H\mathbb{F}_2 \otimes \THH_{C_2}(\HF_2) \simeq H\mathbb{F}_2 \otimes \Sigma^{\infty}_+(L_g \Omega^\sigma S^{\rho + 1})$$

By Corollary 16 of \cite{KK},

$$ H\mathbb{F}_2 \otimes \Sigma^\infty_+ (L_g \Omega^\sigma S^{\rho + 1}) \simeq H\mathbb{F}_2 \otimes \Sigma^\infty_+ (L\Omega S^3) \simeq H\mathbb{F}_2 \otimes H\mathbb{F}_2 \otimes \Sigma^\infty_+ (\Omega S^3)$$

Note that $\THH_{C_2}(\HF_2)$ and $H\mathbb{F}_2 \otimes \Sigma^\infty_+ (\Omega S^3)$ are both $H\mathbb{F}_2$-modules (the former is in fact an algebra over $H\mathbb{F}_2$, as $H\underline{\mathbb{F}}_2$ is commutative.) They are equivalent after smashing with $H\mathbb{F}_2$, therefore they are equivalent.
\end{proof}

\begin{appendices}
  \appendix
\section{Some results in parametrized $\infty$-category theory}
\label{sec:Parametrized_prelim}
In this section we gather the results used in \cref{sec:G_Thom_spectra}, with partial proofs. 
Much of this section has to do with parametrized symmetric monoidal structures.
However, a complete treatment of this subject is beyond the scope of this paper. 
We will therefore consider only $G$-$\infty$-categories and $G$-symmetric monoidal structures (with the exception of \cref{sec:Parametrized_stunst}).

\subsection{Parametrized straightening/unstraightening} \label{sec:Parametrized_stunst}

In this subsection, we state the results that we need about parametrized straightening/unstraightening.
The results are stated for $S$-$\infty$-categories, i.e., coCartesian fibrations over a fixed $\infty$-category $S$.
Taking $S=\OGop$ recovers the notion of $G$-$\infty$-categories, used throughout this paper.

\begin{mydef}
  An $S$-fibration $X \fib \C$ (see \cite[def. 7.1]{Expose2}) is an $S$-right fibration if $X_{[s]} \fib \C_{[s]}$ is a right fibration for every $s\in S$.
\end{mydef}

\begin{mydef}
  Suppose $\C$ is an $S$-$\infty$-category.
  Let $(\ul\Cat_{\infty,S})_{/\ul{\C}}$ denote the $S$-slice category 
  (see \cite[not. 4.29]{Expose2}).
  For $s \in S$, write $\ul{s} = (S_{s/} \to S) \in \Cat_{\infty,S}$.
  Let 
  \[
    (\ul\Cat_{\infty,S})_{/\ul{\C}}^{S-right} \subseteq (\ul\Cat_{\infty,S})_{/\ul{\C}}
  \]
  denote the full subcategory spanned by $\ul{s}$-right fibrations 
  \[
    \left( X \fib \C \times_S \ul{s} \right) \in \left(\Cat_{\infty,\ul{s}}\right)_{/\C \times_S \ul{s}} \simeq \left( (\ul\Cat_{\infty,S})_{/\ul\C} \right)_{[s]}.
  \]
\end{mydef}

\begin{thm}(\cite{Expose1}, Proposition 8.3) \label{parametrized_st_unst} 
  Suppose $ \C \fib S$ is an $S$-category.
  Then there is a natural equivalence of $S$-categories 
  \[
    Y \colon \ul{\operatorname{Psh}}_S(\C) \iso (\ul\Cat_{\infty,S})_{/\ul{\C}}^{S-right}.
  \]
  If $x \in \C_{[s]}$ then $Y$ sends the representable presheaf $j(x)$ to the $\ul{s}$-right fibration
  \[
    \left( \C_{/\ul{x}} \fib \C \times_S \ul{s} \right) \in \left(\Cat_{\infty,\ul{s}}\right)_{/\C \times_S \ul{s}}^{\ul{s}-right} \simeq \left( (\ul\Cat_{\infty,S})_{/\ul\C}^{S-right} \right)_{[s]}.
  \]
\end{thm}

For the following statements, let $B \in \ul\Top_{S}$ be an $S$-$\infty$-groupoid and
\(
  j \colon B \to  \ul{\operatorname{Psh}}_S(B)
\)
its parametrized Yoneda embedding ( \cite[thm. 10.5]{Expose1} ).
\begin{lem}
  Suppose $X \fib B$ is an $S$-right fibration of $S$-$\infty$-categories.
  Then $X$ is an $S$-$\infty$-groupoid.
\end{lem}
\begin{proof}
  We have to show that the coCartesian fibration $X\fib S$ is a left fibration.
  By \cite[prop. 2.4.2.4]{HTT} it is enough to show that each fiber $X_{[s]}$ is a Kan complex.
  \cite[prop. 2.4.2.4]{HTT} also guarantees that $B_{[s]}$ is a Kan complex.
  Since $X_{[s]} \to B_{[s]}$ is a right fibration over a Kan complex, we deduce that $X_{[s]}$ is indeed a Kan complex.
\end{proof}

\begin{cor} \label{parametrized_st_unst_over_gpd}
    Let $\ul{B} \to S$ be an $S$-space (a left fibration over S). 
  There is a natural equivalence of $S$-categories 
  \[
    Y \colon \ul{\operatorname{Psh}}_S(\ul{B}) \iso (\ul\Top_{S})_{/\ul{B}}.
  \]
  If $x \in \ul{B}_{[s]}$, then $Y$ sends the representable presheaf $j(x)$ to 
  \[
    \left( \ul{B}_{/\ul{x}} \fib \ul{B} \times_S \ul{s} \right) \in \left(\Top_{\ul{s}}\right)_{/\ul{B} \times_S \ul{s}} \simeq \left( (\ul\Top_S)_{/\ul{B}} \right)_{[s]}.
  \]
\end{cor}

\subsection{Parametrized preserves and Day convolution}
\label{sec:param-pres-day}
Let $\ul \C$ be a $G$-$\infty$-category. If $\ul\C$ has a $G$-symmetric monoidal structure $\ul\C^\otimes$, then $\Psh(\ul\C)$ has a $G$-symmetric monoidal structure \( \Psh(\ul\C)^\otimes \to \GFin_* \) given by the $G$-Day convolution of \cite[sec. 6]{Parametrized_algebra} with respect to $G$-symmetric monoidal structure on $\ul\C$ and the Cartesian $G$-symmetric monoidal structure on $\ulTopG$.
Our goal in this subsection is \cref{prop:GSM_G_LKE}; informally, it states that parametrized left Kan extension along the Yoneda embedding $j:\ul\C \to \Psh(\ul\C)$ takes a $G$-symmetric monoidal functor from $\ul\C$ to a $G$-symmetric monoidal functor from $\Psh(\ul\C)$.

We will need the following statement:
\begin{lem}[{\cite[cor. 6.0.12]{Parametrized_algebra}}]
  The parametrized Yoneda embedding $j: \ul\C \to \Psh(\ul\C)$ extends to a $G$-symmetric monoidal $G$-functor $j^\otimes:\ul\C^\otimes \to \Psh(\ul\C)^\otimes$.
\end{lem}

We use the notion of a $G$-cocomplete $G$-$\infty$-category from \cite[def. 5.12]{Expose2}, 
and the notion of a distributive $G$-symmetric monoidal $G$-$\infty$-category from \cite[def. 3.2.4]{Parametrized_algebra}.
Note that the essentially unique $G$-symmetric monoidal structure of $\ul\Sp^G$ of \cite[cor. 3.28]{Nardin_thesis} is distributive by construction.

Let \( F^\otimes: \ul\C^\otimes \to \ul\E^\otimes \) be a $G$-symmetric monoidal functor, with underlying $G$-functor $F:\ul\C \to \ul\E$.
If the underlying $G$-$\infty$-category $\ul\E$ is $G$-cocomplete, then $F^\otimes$ admits a $G$-operadic left Kan extension along $j^\otimes$,
\begin{align*}
  (j^\otimes)_! F^\otimes: \Psh(\ul\C)^\otimes \to \ul\E^\otimes,
\end{align*}
constructed in \cite[sec. 4.3]{Parametrized_algebra}.
\begin{prop}[{\cite[prop. 4.3.3]{Parametrized_algebra}}]
  Let $F : \ul\C^\otimes \to \ul\E^\otimes, p^\otimes: \ul\C^\otimes \to \ul\D^\otimes $ be lax $G$-symmetric monoidal functors, and  \( (p^\otimes)_! F^\otimes: \ul\D^\otimes \to \ul\E^\otimes \) the $G$-operadic left Kan extension of $F^\otimes$ along $p^\otimes$.
  Assume that the $G$-symmetric monoidal structure of $\ul\E^\otimes$ is distributive.
  Then the underlying $G$-functor of $(p^\otimes)_! F^\otimes$ is equivalent to the $G$-left Kan extension of $F:\ul\C\to \ul\E$ along $p:\ul\C \to \ul\D$.
\end{prop}
If follows that if the $G$-symmetric monoidal structure of $\ul\E^\otimes$ is
distributive, then the $G$-functor $j_!F: \Psh(\ul\C) \to \ul\E$ can be extended
to a lax $G$-symmetric monoidal functor $(j_!F)^\otimes: \Psh(\ul\C)^\otimes \to
\ul\E^\otimes$ given by the $G$-operadic left Kan extension $(j^\otimes)_!
F^\otimes$.
We show that $(j^\otimes)_! F^\otimes$ is in fact $G$-symmetric
monoidal.
\begin{prop} \label{prop:GSM_G_LKE}
  Let \( F^\otimes: \ul\C^\otimes \to \ul\E^\otimes \) be a $G$-symmetric monoidal functor from a small $G$-symmetric monoidal $\infty$-category $\ul\C^\otimes$ to a distributive $G$-symmetric monoidal $G$-$\infty$-category $\ul\E^\otimes$, with $G$-cocomplete underlying $G$-$\infty$-category $\ul\E$.
  Then $(j_!F)^\otimes: \Psh(\ul\C)^\otimes \to \ul\E^\otimes$ is a $G$-symmetric monoidal functor.
\end{prop}
In the course of the proof we use the following notation.
\begin{notation} \label{notation:GSM_parametrized_fiber}
  For a coCartesian fibration  $\ul\C^\otimes \fib \GFin_*$ and
  $I=(U\to G/H) \in \GFin_*$,
  \begin{enumerate}
    \item Recall that $\ul\C^\otimes_I$ is the fiber of $\ul\C^\otimes \to \GFin_*$ over $I$.
    \item Let $\ul\C^\otimes_{<I>}$ denote the $\ul{G/H}$-category constructed by
      pulling back along $\sigma_{<I>} \colon \ul{G/H} \to \GFin_*$ as in  \cite[def. B.0.4]{GFH}.
  \end{enumerate}
\end{notation}
\begin{proof}
  We have to check that the lax $G$-symmetric monoidal functor  $(j_!F)^\otimes: \Psh(\ul\C)^\otimes \to \ul\E^\otimes$ is $G$-symmetric monoidal.
  The idea of the proof is simple: reduce to the case of parametrized representable presheaves, where the claim is clear. 
  The argument is a bit convoluted due to the involved definition of a distributive $G$-symmetric monoidal structure. 
  
  Let $I\in \GFin_*, I=(U\to G/H)$.
  Parametrized Day convolution defines a distributive $G$-symmetric monoidal structure on $\Psh(\ul\C)$, so the $\ul{G/H}$-functor 
  \[
    \otimes_I : \Psh(\ul\C)^\otimes_{<I>} \simeq \prod_I \Psh(\ul\C) \ultimes \ul{U} \to \Psh(\ul\C) \ultimes \ul{G/H}
  \]
  of \cite[def. B.0.11]{GFH} is distributive (see \cite[def. 3.15]{Nardin_thesis}).
  Here $\prod_I : \ul\Cat_\infty^{\ul{U}} \to \ul\Cat_\infty^{\ul{G/H}}$ is the right adjoint of \( (- \times_{\ul{G/H}} \ul{U} ) \colon \ul\Cat_\infty^{\ul{G/H}} \to \ul\Cat_\infty^{\ul{U}} \), and the equivalence is homotopy inverse to the parametrized Segal map of \cite[rem. B.0.9]{GFH}.

  We have to show that for every $X \in \Psh(\ul\C)^\otimes_I \simeq \left( \Psh(\ul\C)^\otimes_{<I>} \right)_{[G/H]} $ the lax structure map 
  \begin{equation}
    \otimes_I (j_!F^\otimes(X)) \to j_! F ( \otimes_I X) 
    \label{lax_str_map}
  \end{equation}
  is an equivalence. 
  We first reduce to representable presheaves.
  By \cref{U_Psh_as_reps_colim} we can write $X$ as a $\ul{U}$-colimit \( X \simeq \ul{U}-colim (j \chi) \) for some $\ul{U}$-diagram \( \chi: K \to \ul\C\ultimes \ul{U} \). 
  Inspect the following diagram:
  \[
    \begin{tikzcd}
      \otimes_I ( j_!F^\otimes (X)) \ar{r} \ar[dash]{d}{\sim} & j_!F(\otimes_I (X)) \ar[dash]{d}{\sim} \\ 
      \otimes_I ( j_!F^\otimes (\ul{U}\operatorname{-\colim}(j \chi) )) \ar{r} \ar[dash,"\sim", "(1)"']{d} & j_!F( \otimes_I ( \ul{U}\operatorname{-\colim}(j \chi)) ) \ar[dash,"\sim", "(2)"']{d} \\
      \otimes_I ( \ul{U}\operatorname{-\colim}(j_!F^\otimes (j \chi) )) \ar[dash,"\sim", "(2)"']{d} & j_!F( \ul{G/H}\operatorname{-\colim}(\otimes_I ( j \chi)) ) \ar[dash,"\sim", "(1)"']{d} \\ 
      \ul{G/H}\operatorname{-\colim}(\otimes_I (j_!F^\otimes (j \chi) )) \ar{r} & \ul{G/H}\operatorname{-\colim}( j_!F^\otimes ( \otimes_I ( j \chi)) ).
    \end{tikzcd}
  \]
  The commutativity of the diagram follows from the naturality of the lax structure map \eqref{lax_str_map}.
  The equivalences marked (1) follow from the fact that $j_!F$ strongly preserves $G$-colimits, and the equivalences marked (2) follow from the distributivity of $G$-Day convolution.

  By naturality of the lax structure map \eqref{lax_str_map} it is therefore enough to show that the lax structure map \( \otimes_I ( j_!F^\otimes (j^\otimes X) ) \to j_!F(\otimes_I j^\otimes X) \) is an equivalence for $X\in \ul\C^\otimes_I$.
  This follows from inspecting the following diagram
  \[
    \begin{tikzcd}
      & \otimes_I j_! F(j^\otimes X) \ar["(1)"']{dl} \ar{d}{(3)}  & \ar["\sim"']{l} \otimes_I F^\otimes (X) \ar{d}{(4)} \\
      j_!F( \otimes_I (j^\otimes X)) \ar{r}{(2)} & j_! F( j (\otimes_I X)) & \ar["\sim"']{l} F (\otimes_I X) ,
    \end{tikzcd}
  \]
  as we now explain.
  We wish to show that the diagonal map marked (1) is an equivalence. 
  Since the parametrized Yoneda embedding $j$ is $G$-symmetric monoidal, its lax structure map  
  \( \otimes_I (j^\otimes X) \to  j (\otimes_I X) \)
  is an equivalence. 
  It follows that it is still an equivalence after applying $j_!F$, showing that the map (2) is also an equivalence.
  Therefore it is enough to show that the map marked (3) is an equivalence.
  Note that the map marked (3) is the lax structure map of the composition $j_!F^\otimes \circ j^\otimes$.
  We now use the fact that the parametrized Yoneda embedding is fully faithful (\cite[thm. 10.4]{Expose1}) together with \cite[prop. 10.5]{Expose2} to deduce that the associated natural transformation $ F \to j_!F \circ j$ is a natural equivalence.
  It follows that the left-pointing horizontal maps in the diagram are equivalences (the square commutes by naturality).  
  Hence it is enough to prove that the map marked (4) is an equivalence, which is clear since it is the lax structure map of a $G$-symmetric monoidal functor $F^\otimes$.
\end{proof}

\subsection{Maximal $G$-$\infty$-subgroupoid and $G$-symmetric monoidal structures}
We recall the definition of the maximal $G$-$\infty$-subgroupoid of an $G$-$\infty$-category $\ul\C$, and verify that a $G$-symmetric monoidal structure on $\ul\C$ induces a $G$-symmetric monoidal on its maximal $G$-$\infty$-subgroupoid.
Recall that a $G$-$\infty$-groupoid, or a $G$-space, is a $G$-$\infty$-category \( \ul\G \fib \OGop\) in which every edge is coCartesian (\cite[def. 1.1]{Expose1}). 
By \cite[2.4.2.4]{HTT} this happens precisely when \( \ul\G \fib \OGop\) is a left fibration. 

Let $\ul\C \fib \OGop$ be a $G$-$\infty$-category. 
The \myemph{maximal $G$-subgroupoid of $\ul\C$} is the subcategory $\ul\C^\simeq \subset \ul\C$ spanned by all objects and all coCartesian edges.
By construction, $\ul\C^\simeq \subset \ul\C$ is the maximal $G$-$\infty$-subcategory (\cite[sec. 4]{Expose1}) which is a $G$-$\infty$-groupoid. 
Note that for every orbit $W\in \OGop$, the morphisms in the fiber $(\ul\C^\simeq)_{[W]}$ are coCartesian edges in $\ul\C$ over $id_W$, which by \cite[prop. 2.4.1.5]{HTT} are exactly equivalences over $id_W$. Hence we have \(  (\ul\C^\simeq)_{[W]} = (\ul\C_{[W]})^\simeq \) as subsets of $\ul\C_{[W]}$.

\begin{construction}
  Suppose $\ul\C^\otimes \fib \GFin_* $ is a $G$-symmetric monoidal $G$-$\infty$-category.
  Define $\ul\C^\otimes_{coCart} \subset \ul\C^\otimes$ as the full subcategory spanned by the coCartesian morphisms over $\GFin_*$.
\end{construction}
\begin{lem} \label{GSM_str_on_max_G_subgpd}
  The composition \( \ul\C^\otimes_{coCart} \subset \ul\C^\otimes \fib \GFin_*
  \) is a coCartesian fibration which defines a $G$-symmetric monoidal structure
  on $\ul\C^{\simeq}$, the maximal $G$-$\infty$-subgroupoid of $\ul\C$. 
\end{lem}
During the proof we use the notation $\C^\otimes_I$ for the fiber of $\C^\otimes$ over $I\in \GFin_*$, see \cref{notation:GSM_parametrized_fiber}.
\begin{proof}
  The map \( \ul\C^\otimes_{coCart} \to \GFin_* \) is a left fibration by \cite[cor. 2.4.2.5]{HTT}.
  Pulling back \( \ul\C^\otimes_{coCart} \subset \ul\C^\otimes \to \GFin_* \) over the $G$-functor $\sigma_{<G/G>} : \OGop \to \GFin_*, \, [G/H] \mapsto (G/H \xto{=} G/H) $ we see that the underlying $G$-$\infty$-category of $\ul\C^\otimes_{coCart}$ is the full $G$-subcategory of $\ul\C$ spanned by the coCartesian morphisms, i.e., the full $G$-$\infty$-subgroupoid of $\ul\C$. 
  Note that morphisms in the fiber $(\ul\C^\otimes_{coCart})_{[W]}$ are coCartesian edges in $\ul\C^\otimes$ over $id_W$, which by \cite[prop. 2.4.1.5]{HTT} are exactly equivalences over $id_W$.
  Hence we have \(  (\ul\C^\otimes_{coCart})_I = (\ul\C^\otimes_{I})^\simeq \) as subsets of $\ul\C^\otimes_I$.

  Finally, we have to show that the Segal maps of $\ul\C^\otimes_{coCart}$ are equivalences. 
  Let $I=(U\to G/H) \in (\GFin_*)_{[G/H]}$, and consider the Segal map associated to $I$.
  Since $\ul\C^\otimes$ is a $G$-symmetric monoidal $G$-$\infty$-category, the Segal map \( \ul\C^\otimes_{I} \iso \prod_{W\in \orb(U)} \ul\C_{[W]} \) is an equivalence of $\infty$-categories.
  Recall that the Segal map is defined as a product of functors \( \ul\C^\otimes_{I} \to \ul\C_{[W]} \), defined by choosing coCartesian lifts of specified inert morphisms in $\GFin_*$.
  The maximal groupoid functor preserves products, so applying it to the Segal map above produces an equivalence
  \( (\ul\C^\otimes_{I})^\simeq \iso \prod_{W\in \orb(U)} (\ul\C_{[W]})^\simeq \).
  Using the equalities \( (\ul\C^\otimes_{coCart})_{I} = (\ul\C^\otimes_{I})^\simeq \) and \( (\ul\C_{[W]})^\simeq  =(\ul\C^\simeq)_{[W]} \) we write the above equivalence as
  \(
    (\ul\C^\otimes_{coCart})_{I}  \iso  \prod_{W\in \orb(U)} (\ul\C^\simeq)_{[W]},
  \)  
  which is exactly the Segal map of $\ul\C^\otimes_{coCart}$.
\end{proof}
\begin{ex}
  The $G$-symmetric monoidal structure on $\ul\Sp^G$ induces a $G$-symmetric monoidal structure on its maximal subgroupoid $(\ul\Sp^G)^\simeq $.
\end{ex}

\subsection{$G$-symmetric monoidal categories and $G$-commutative algebras}
By  \cite[sec.\ 3.1]{Nardin_thesis}, a $G$-symmetric monoidal category is an $G$-commutative monoid in $\ul\Cat_\infty^G$.
\begin{thm}[{\cite[thm. 2.32]{Nardin_thesis}}]
  There is an equivalence of $\infty$-categories 
  \[
    \CMon_G(\ul\Cat_\infty^G) \simeq \CAlg_G((\ul\Cat_\infty^G)^\times)
  \]
  between the $\infty$-category $\CMon_G(\ul\Cat_\infty^G)$ of $G$-symmetric monoidal $G$-$\infty$-categories 
  and the $\infty$-category $\CAlg_G((\ul\Cat_\infty^G)^\times)$ of $G$-commutative algebras in $\ul\Cat_\infty^G$, with respect to the $G$-Cartesian $G$-symmetric monoidal structure.
\end{thm}
The equivalence of \cite[thm. 2.32]{Nardin_thesis} restricts to give the following:
  \begin{cor}
    There is an equivalence of $\infty$-categories
  \[
    \CMon_G(\ulTopG) \simeq \CAlg_G((\ulTopG)^\times)
  \]
  between the $\infty$-category $\CMon_G(\ulTopG)$ of $G$-symmetric monoidal $G$-$\infty$-groupoids 
  and the $\infty$-category $\CAlg_G((\ulTopG)^\times)$ of $G$-commutative algebras in $\ulTopG$. 
\end{cor}
\begin{ex} \label{G_Picard_as_G_comm_alg}
  The Picard $G$-space $\Pic(\ul\C)$ admits $G$-symmetric monoidal structure (see \cref{G_Picard}), and therefore defines a $G$-commutative monoid in $\ulTopG$.
  Applying the result above we can consider $\Pic(\ul\C)$  as a $G$-commutative algebra in $(\ulTopG)^\times$.
\end{ex}

\subsection{Slicing $G$-symmetric monoidal categories} \label{G_slice_GSM_str}
\begin{mydef} \label{def:G_slice}
  Let $\ul\C^\otimes$ be a $G$-symmetric monoidal $G$-$\infty$-category and $A: \GFin_* \to \ul\C^\otimes$ a $G$-commutative algebra. 
  Define $\ul\C^\otimes_{/\ul{A}} \to \GFin_*$ by applying the construction \cite[def. 2.2.2.1]{HA} for $K=\Delta^0, \, S=\GFin_*$ and $S \times K \to S$ the identity map. 
\end{mydef}
The following statement is a result of \cite[sec. 2.2.2]{HA} together with the $G$-Segal conditions of a $G$-symmetric monoidal $\infty$-category.
\begin{prop} 
  The map \( \ul\C^\otimes_{/\ul{A}} \to \GFin_* \) defines a $G$-symmetric monoidal
  $G$-$\infty$-category, with underlying $G$-$\infty$-category equivalent to the
  parametrized slice $\ul\C_{/\ul{A}}\to \OGop$ of \cite[not. 4.29]{Expose2}.
\end{prop}
We will use the following description of $\mathcal{O}$-algebras in $\ul\C^\otimes_{/\ul{A}}$.
\begin{prop} \label{O_algs_in_G_slice}
  Let $\mathcal{O}^\otimes \to \GFin_*$ be a $G$-$\infty$-operad. 
  The $\infty$-category \( \Alg_{\mathcal{O}}(\ul\C^\otimes_{/\ul{A}}) \) of
  $\mathcal{O}$-algebras in $\ul\C^{\otimes}_{/\ul{A}}$ is equivalent to the slice $\infty$-category
  \(\Alg_{\mathcal{O}}(\ul\C)_{/A}\) of $\mathcal{O}$-algebras over $A$.
\end{prop}

\begin{proof}
  The proof follows from unraveling the definitions and is skipped.
\end{proof}

\subsection{Parametrized symmetric monoidal straightening/unstraightening}

The following theorem is known to the experts, although its proof is
  not in the literature. A proof of symmetric monoidal
  straightening-unstraightening (but not parametrized) can be found in
  \cite{Ramzi}. We used the following theorem in \cref{thm:G_Thom_as_GSM_functor} to
  extend the Thom spectrum $G$-functor $\Th$ to a $G$-symmetric monoidal functor.
\begin{thm}(Folklore) \label{GSM_straightening_unstraightening}
  Suppose $ B^\otimes \fib \ulFin^S_*$ is an $S$-symmetric monoidal $S$-$\infty$-groupoid.
  Then the natural equivalence of \cref{parametrized_st_unst_over_gpd} extends to an $S$-symmetric monoidal equivalence 
  \[
    \ul{\operatorname{Psh}}_S(B)^\otimes \iso (\ul\Top_{S})_{/\ul{B}}^\otimes ,    
  \]
  where the $S$-symmetric monoidal structure on the right hand side is given by \cref{G_slice_GSM_str}
  and the $S$-symmetric monoidal structure on the left hand side is given by $S$-Day convolution (\cite[sec. 4.3]{Parametrized_algebra}).
\end{thm}

\subsection{Parametrized presheaves and parametrized colimits} 
In this subsection we state some properties of the parametrized presheaf
category that we used in \cref{sec:param-pres-day}.

Let $\Psh(\ul\C) = \ulFun_G(\ul\C^{vop}, \ulTopG )$ denote the parametrized presheaf $G$-$\infty$-category of a small $G$-$\infty$-category $\ul\C \to \OGop$, 
and let  $j:\ul\C \cof \Psh(\ul\C)$ be the parametrized Yoneda embedding of \cite[def. 10.2]{Expose1}.

We will construct $G$-functors out of $\Psh(\ul\C)$ using parametrized $G$-left Kan extension (see \cite[sec. 10]{Expose2} and \cite[def. 2.12]{Expose4}).
Specifically, we will use \cite[thm. 11.5]{Expose2}.
Let \( \Fun^L_G(\ul\C,\ul\D) \subseteq \Fun_G(\ul\C,\ul\D) \) denote the full subcategory of $G$-functors which strongly preserve $G$-colimits (\cite[def. 11.2]{Expose2}).
\begin{thm}[Shah]
  Let $\ul\C$ be a $G$-$\infty$-category and let $\ul\E$ be a $G$-cocomplete $G$-$\infty$-category. 
  Then restriction along the $G$-Yoneda embedding $j:\ul\C \to \Psh(\ul\C)$ defines an equivalence of $\infty$-categories 
  \begin{align*}
    \Fun^L_G(\Psh(\ul\C), \ul\E ) \iso \Fun_G(\ul\C,\ul\E)
  \end{align*}
  with inverse given by $G$-left Kan extension along $j$.
\end{thm}

Unsurprisingly, every parametrized presheaf is equivalent to a parametrized colimit of representable presheaves. 
Before giving a formal statement we recall the relevant definition of parametrized colimits in $\Psh(\ul\C)$.
Let $G/H\in \OGop$ be an orbit and $I \to \ul{G/H}$ be a $\ul{G/H}$-category.
Keeping in mind the equivalence
\[
  \Psh(\ul\C)_{[G/H]} \simeq \Fun_G( \ul{G/H} , \Psh(\ul\C) ) \simeq \Fun_{\ul{G/H}} ( \ul{G/H} , \Psh(\ul\C) \ultimes \ul{G/H} ) ,
\]
we define a $\ul{G/H}$-functor
\[
  \Delta_I : \Psh(\ul\C)_{[G/H]} \simeq \Fun_{\ul{G/H}} (\ul{G/H}, \Psh(\ul\C) \ultimes \ul{G/H} ) \to \Fun_{\ul{G/H}} ( I, \Psh(\ul\C) \ultimes \ul{G/H} ) ,
\]
induced by precomposition with the structure map $I \to \ul{G/H}$.
By definition $\ul{G/H}$-colimits in $\Psh(\ul\C)$ along $I$-shaped diagrams are given by the left adjoint 
\[
  \ul{G/H}\operatorname{-\colim} \colon \Fun_{\ul{G/H}} ( I, \Psh(\ul\C) \ultimes \ul{G/H} ) \adj \Psh(\ul\C)_{[G/H]}  : \Delta_I.
\]
See \cite[def. 1.15]{Nardin_thesis} for details.
\begin{lem} \label{Psh_as_reps_colim}
  Let  \( X\in \Psh(\ul\C)\) be a presheaf over $G/H\in \OGop$.
  Then $X$ is equivalent to a $\ul{G/H}$-colimit of a diagram of representable presheaves
  \[
    K \xto{\chi} \ul\C \ultimes \ul{G/H} \xto{j \ultimes \ul{G/H}} \Psh(\ul\C) \ultimes \ul{G/H}
  \]
  for some $\ul{G/H}$-functor \( \chi: K \to \ul\C \ultimes \ul{G/H} \).
\end{lem}
\begin{proof}
  By the $G$-Yoneda lemma, \cite[lem. 11.1]{Expose2}, the identity functor \( Id:\Psh(\ul\C) \to \Psh(\ul\C) \) is a $G$-left Kan extension of $j$ along itself. 
  By \cite[thm. 10.4]{Expose2} we can express the value of this $G$-left Kan extension on $X$ as a $\ul{G/H}$-colimit
  \begin{align*}
    X = Id(X) \simeq \ul{G/H}\operatorname{-\colim} \left( \ul\C_{/\ul{X}} \to \ul\C \ultimes \ul{G/H} \xto{j \ultimes \ul{G/H}} \Psh(\ul\C) \ultimes \ul{G/H} \right) ,
  \end{align*}
  where $\ul\C_{/\ul{X}} = \ul\C \times_{\Psh(\ul\C)} \Psh(\ul\C)_{/\ul{X}}$ is the pullback of the $G$-slice category $\Psh(\ul\C)_{/\ul{X}}$ (\cite[not. 4.29]{Expose2}) along the $G$-Yoneda embedding $j$.
\end{proof}
\begin{cor} \label{U_Psh_as_reps_colim}
  Let $U$ be a finite $G$-set and let \( X : \ul{U} \to \Psh(\ul\C) \ultimes \ul{U} \) be $\ul{U}$-functor.
  Then there exists a $\ul{U}$-functor \( \chi: K \to \ul\C \ultimes \ul{U} \), such that the $\ul{U}$-colimit of 
  \[
    K \xto{\chi} \ul\C \ultimes \ul{U} \xto{j \ultimes \ul{U}} \Psh(\ul\C) \ultimes \ul{U}
  \]
  is equivalent to $X$.
\end{cor}
\begin{proof}
  Decompose $U=\coprod_{W\in \orb(U)} W$ into orbits. 
  The result follows from \cref{Psh_as_reps_colim} and the equivalence
  \[
    \prod_W \Cat^{\ul{W}}_\infty \iso \Cat^{\ul{U}}_\infty , \quad (\ul{\C_W} \fib \ul{W})_{W\in\orb(U)} \mapsto \left( \coprod_W \ul{\C_W} \fib \coprod_W \ul{W} = \ul{U} \right),
  \]
  where coproducts and products are indexed over $W\in \orb(U)$.
\end{proof}

\section{The Real topological Hochschild homology of  $H\protect\underline{\mathbb{Z}}$ by Jeremy Hahn and Dylan Wilson}

In this appendix we explain how the results of the main body of the paper allow one to calculate the Real topological Hochschild homology of the Eilenberg--MacLane Mackey functor $\HZ$.  In particular, we deduce the following theorem, which verifies a conjecture of Dotto, Moi, Patchkoria, and Reeh \cite[p. 136]{Real_THH}.

\begin{thm} \label{app:main-thm}
There is an equivalence of $\HZ$-module spectra
$$\THR(\HZ) \simeq \HZ \oplus \bigoplus_{k \ge 2} \Sigma^{k\rho-1} H\underline{\mathbb{Z}/k}.$$
\end{thm}

Dotto, Moi, Patchkoria, and Reeh were able to prove that Theorem
\ref{app:main-thm} holds after localization at any odd prime, and so also after
localization away from $2$ \cite[Theorem 5.24 \& Corollary 5.25]{Real_THH}. 
However, they did not have methods to calculate $\THR(\HZ)_{(2)} \simeq \THR(\HZ_{(2)})$.  On the other hand, the main body of this paper provides methods to calculate the $\THR$ of Thom spectra, and the authors of this appendix previously constructed $\HZ_{(2)}$ as a $C_2$-equivariant Thom spectrum in \cite{HahnWilson}.  These results were combined in Corollary \ref{cor-em}(ii) of the main body to prove that
$$\THR(\HZ_{(2)}) \simeq \HZ_{(2)} \otimes \Sigma^{\infty}_+\Omega^\sigma (S^{\lambda +1}\langle \lambda +1 \rangle).$$
The main contribution of the appendix is to observe that this can be made more explicit:

\begin{lem} \label{app:main-lem}
There is an equivalence of $\HZ_{(2)}$-module spectra
$$\HZ_{(2)} \otimes \Sigma^{\infty}_+ \Omega^\sigma (S^{\lambda +1}\langle \lambda +1 \rangle) \simeq \HZ_{(2)} \oplus \bigoplus_{k \ge 1} \Sigma^{k\rho-1} H\underline{\mathbb{Z}/k}_{(2)}.$$
\end{lem}
We deduce Lemma \ref{app:main-lem} from the non-equivariant calculation of $\mathrm{THH}(H\mathbb{Z})$ together with the following $C_2$-equivariant fact, which we prove before $2$-localization:
\begin{lem} \label{app:red-lem}
There is a cofiber sequence of $\HZ$-module spectra
$$\HZ \otimes \Sigma^{\infty}_+ \Omega^{\sigma} S^{\lambda+1} \langle \lambda+1 \rangle \to \bigoplus_{k \ge 0} \Sigma^{k\rho} \HZ \to \bigoplus_{s \ge 1} \Sigma^{s \rho} \HZ.$$
\end{lem}

\begin{proof}[Proof of Lemma \ref{app:red-lem}]
Applying $\Omega^{\sigma}$ to the definition of $S^{\lambda+1} \langle \lambda+1\rangle$ yields a fiber sequence of $C_2$-equivariant spaces
$$\Omega^{\sigma} S^{\lambda+1} \langle \lambda+1 \rangle \to \Omega^{\sigma} S^{\lambda+1} \to \Omega^{\sigma} \mathrm{K}(\lambda+1,\underline{\mathbb{Z}}),$$
where $\Omega^{\sigma} \mathrm{K}(\lambda+1,\underline{\mathbb{Z}}) = \Omega^{\sigma} \mathrm{K}(2\sigma+1,\underline{\mathbb{Z}}) \simeq \mathrm{K}(\sigma+1,\underline{\mathbb{Z}}) \simeq \mathbb{CP}^{\infty}_{\mathbb{R}}.$

In particular, since $\mathbb{CP}^{\infty}_{\mathbb{R}}$ classifies Real line bundles, it follows that the \emph{cofiber} of the map $\Omega^{\sigma} S^{\lambda+1} \langle \lambda+1\rangle \to \Omega^{\sigma} S^{\lambda+1}$ is the Thom space of a Real line bundle $\mathcal{L}$ over $\Omega^{\sigma} S^{\lambda+1}$.  Using the fact that $\HZ$ is Real oriented, we conclude that there is a cofiber sequence of $\HZ$-modules
$$\HZ \otimes \Sigma^{\infty}_+ \Omega^{\sigma} S^{\lambda+1} \langle \lambda+1 \rangle \to \HZ \otimes \Sigma^{\infty}_+ \Omega^{\sigma} S^{\lambda+1} \to \HZ \otimes \Sigma^{\infty}_+ \left( \Omega^{\sigma} S^{\lambda+1} \right)^{\mathcal{L}} \simeq \Sigma^{\rho} \HZ \otimes \Sigma^{\infty}_+ \Omega^{\sigma} S^{\lambda+1}.$$
By \cite[Theorem 4.3]{HillAlgebras}, there is a James splitting
$$\Sigma^{\infty}_+ \Omega^{\sigma} S^{\lambda+1} \simeq \Sigma^{\infty}_+ \Omega^{\sigma} \Sigma^{\sigma} S^{\rho} \simeq S^0 \oplus S^{\rho} \oplus S^{2\rho} \oplus S^{3\rho} \oplus \cdots.$$
In particular, there is a cofiber sequence of $\HZ$-modules
$$\HZ \otimes \Sigma^{\infty}_+ \Omega^{\sigma} S^{\lambda+1} \langle \lambda+1 \rangle \to \bigoplus_{k \ge 0} \Sigma^{k\rho} \HZ \to \bigoplus_{s \ge 1} \Sigma^{s \rho} \HZ,$$
as desired.
\end{proof}

\begin{proof}[Proof of Lemma \ref{app:main-lem}]
By $2$-localizing the result of Lemma \ref{app:red-lem}, we learn that
$$\THR(\HZ_{(2)}) \simeq \HZ_{(2)} \otimes \Sigma^{\infty}_+ \Omega^{\sigma} S^{\lambda+1} \langle \lambda+1 \rangle$$
may be calculated as the fiber of a map $f$ of $\HZ_{(2)}$-module spectra
$$f:\bigoplus_{k \ge 0} \Sigma^{k\rho} \HZ_{(2)}  \to \bigoplus_{s \ge 1} \Sigma^{s \rho} \HZ_{(2)}.$$
Since the domain of $f$ is a direct sum of free $\HZ_{(2)}$-module spectra, $f$ is determined by a sequence of elements
$f_k \in \pi_{k\rho} \left( \bigoplus_{s \ge 1} \Sigma^{s \rho} \HZ_{(2)} \right).$  The $RO(C_2)$-graded homotopy groups of $\HZ$, as nicely displayed for example in \cite[p.6]{GreenlessApproaches}, show that there are no classes in $\pi_{k\rho} (\Sigma^{s\rho}\HZ_{(2)})$ unless $k=s$.  Furthermore, one sees that any such class is determined by its underlying non-equivariant class in $\pi_{2k}(\Sigma^{2k} H\mathbb{Z}_{(2)}) \cong \Z_{(2)}$, and in particular the map $f$ is entirely determined by its underlying non-equivariant map
$$f_{\text{underlying}}:\bigoplus_{k \ge 0} \Sigma^{2k} H\mathbb{Z}_{(2)}   \to \bigoplus_{s \ge 1} \Sigma^{2s} H\mathbb{Z}_{(2)}.$$
The fiber of $f_{\text{underlying}}$ must agree with the known non-equivariant calculation 
$$\mathrm{THH}(H\mathbb{Z}_{(2)}) \simeq H\mathbb{Z}_{(2)} \oplus \bigoplus_{s \ge 1} \Sigma^{2s-1} (H\mathbb{Z}/s)_{(2)},$$
which determines the map $f_{\text{underlying}}$ well enough to determine the fiber of $f$ up to equivalence.
\end{proof}

\begin{proof}[Proof of Theorem \ref{app:main-thm}]
As with any $C_2$-equivariant spectrum, there is a pullback square
$$
\begin{tikzcd}
\THR(\HZ) \arrow{r} \arrow{d} & \THR(\HZ)[\frac{1}{2}]  \arrow{d} \\
\THR(\HZ)_{(2)} \arrow{r} & \THR(\HZ) \otimes H\underline{\mathbb{Q}}.
\end{tikzcd}
$$
The $2$-local Lemma \ref{app:main-lem} allows to calculate the lower left corner
of the square, while the result \cite[Corollary 5.25]{Real_THH} 
of Dotto, Moi, Patchkoria, and Reeh calculates the upper right.  From these results, we learn that the square is a direct sum of squares
$$
\begin{tikzcd}
H\underline{\mathbb{Z}} \arrow{d} \arrow{r} & H\underline{\mathbb{Z}}[1/2] \arrow{d} \\
H\underline{\mathbb{Z}}_{(2)} \arrow{r} & H\underline{\mathbb{Q}},
\end{tikzcd}
$$
and, for all $k \ge 1$,
$$
\begin{tikzcd}
\Sigma^{k \rho-1} H\underline{\mathbb{Z}/k} \arrow{d} \arrow{r} & \Sigma^{k\rho-1} H\underline{\mathbb{Z}/k}[1/2] \arrow{d} \\
\Sigma^{k\rho-1} H\underline{\mathbb{Z}/k}_{(2)} \arrow{r} & 0.
\end{tikzcd}
$$

\end{proof}

\end{appendices}

\bibliography{references}{}
\bibliographystyle{plain}

\end{document}